\documentclass[a4paper]{amsart}

\usepackage{amsmath}
\usepackage{amsfonts}
\usepackage{amssymb}
\usepackage{amsthm}
\usepackage{amscd}
\usepackage{xcolor}
\usepackage{float}
\usepackage{tikz-cd}
\usetikzlibrary{decorations.pathmorphing}
\usetikzlibrary{shapes.geometric}
\usetikzlibrary{patterns}
\usetikzlibrary{calc}
\usetikzlibrary{positioning}
\usepackage[T1]{fontenc}
\usepackage[english]{babel}
\usepackage[square,sort,comma,numbers]{natbib}

\usepackage[shortlabels]{enumitem}
\usepackage{mathtools}
\usepackage{hyperref}

\theoremstyle{plain}
\newtheorem{theorem}{Theorem}[section]
\newtheorem*{theorem*}{Theorem}
\newtheorem{lemma}[theorem]{Lemma}
\newtheorem{proposition}[theorem]{Proposition}

\theoremstyle{definition}
\newtheorem{definition}[theorem]{Definition}
\newtheorem{convention}[theorem]{Convention}
\theoremstyle{remark}
\newtheorem{remark}[theorem]{Remark}
\newtheorem{example}[theorem]{Example}

\newcommand{\dd}{\colon}

\definecolor{dark-green}{RGB}{14,150,2}
\newcommand{\gpoint}{\color{dark-green}{\circ}}
\newcommand{\rpoint}{\color{red} \bullet}

\DeclareMathOperator{\Hom}{Hom}
\DeclareMathOperator{\Kb}{\mathsf{K}^\mathsf{b}}
\DeclareMathOperator{\Db}{\mathsf{D}^\mathsf{b}}
\newcommand{\proj}{\textrm{-}\mathrm{proj}}
\newcommand{\modl}{\textrm{-}\mathrm{mod}}
\newcommand{\simSt}{\sim_\mathsf{St}}
\newcommand{\simBa}{\sim_\mathsf{Ba}}
\newcommand{\simInfSt}{\sim_{\infty\mathsf{St}}}
\newcommand{\Z}{\mathbb{Z}}

\newcommand{\mcB}{\mathcal{B}}
\newcommand{\mcC}{\mathcal{C}}
\newcommand{\mcD}{\mathcal{D}}

\newcommand{\mcG}{\mathcal{G}}

\newcommand{\mcL}{\mathcal{L}}

\newcommand{\mcO}{\mathcal{O}}
\newcommand{\mcP}{\mathcal{P}}
\newcommand{\mcQ}{\mathcal{Q}}
\newcommand{\mcR}{\mathcal{R}}
\newcommand{\mcS}{\mathcal{S}}
\newcommand{\mcT}{\mathcal{T}}

\newcommand{\nameto}[1]{\; \tikz{\node[] (text) {\scriptsize $#1$}; \draw[->] (text.south west) -- (text.south east);} \;}
\newcommand{\cone}{\mathrm{cone}}
\newcommand{\id}{\mathrm{id}}

\title[Semiorthogonal decompositions for gentle algebras]%
{Semiorthogonal decompositions for bounded derived categories of gentle algebras}
\author{Jakub Kopřiva}
\author{Jan Šťovíček}
\date{\today}

\begin{document}

\begin{abstract}
We study semiorthogonal decompositions of bounded derived categories of gentle algebras and how they are manifested in the geometric model of these categories as constructed by Opper, Plamondon and Schroll in \cite{opper2018geometric}.
We prove that there is a one-to-one correspondence between such semiorthogonal decompositions and suitable cuts of the marked surface underlying the geometric model.
Our main tool is the characterization of basis morphisms between indecomposable objects due to Arnesen, Laking and Pauksztello in~\cite{arnesen2016morphisms}.
\end{abstract}

\maketitle
\tableofcontents

\section*{Introduction}
Gentle algebras arose as generalizations of iterated tilted algebras of type $A$ in \cite{assem1981generalized} and \cite{assem1987iterated}. In representation theory, gentle algebras have been of interest mainly due to the fact that they form a rich, but sufficiently well-behaved class of algebras. However, gentle algebras appear naturally in many other areas of mathematics as well, a preeminent example being the field of homological mirror symmetry (for example, see \cite{haiden2014flat}). Gentle algebras have been extensively studied, and their representation theory is well-understood.

Recently, bounded derived categories of gentle algebras have been a focus of substantial research activity that followed up on \cite{bekkert2003indecomposables}, which gave a characterization of their indecomposable objects: morphisms between these indecomposable were fully described in \cite{arnesen2016morphisms}, their mapping cones were calculated in \cite{ccanakcci2019mapping} and \cite{ccanakcci2021corrigendum}, and a geometric model for these derived categories has been established in \cite{opper2018geometric} and used, among other things, to characterize silting objects thereof in \cite{amiot2019complete} and give complete derived invariants in \cite{opper2019auto} and \cite{amiot2019complete}.

The geometric model has already proven to be a useful tool for finer study of bounded derived categories of gentle algebras as well: for example, silting objects therein and related notions are studied in \cite{chang2020geometric} and \cite{chang2022recollements} using specific ways to cut the marked surface underlying the geometric model.
A very similar model has been also developed to study categories of modules over gentle algebras~\cite{baur2021geometric}.

\smallskip

In this paper, we study semiorthogonal decompositions of bounded derived categories of gentle algebras. Semiorthogonal decompositions are a useful tool for understanding triangulated categories because they allow us to view an entire triangulated category as constructed from simpler subcategories; they are studied in algebraic geometry for the derived category of coherent sheaves on algebraic varieties (for example, see \cite{kuznetsov2014semiorthogonal}), algebraic topology (where they are often referred to as Bousfield localizations, \cite[\S3]{hovey1997stablehomotopy}) as well as representation theory. %Semiorthogonal decompositions also provide simple examples of Verdier localisation.

We show that, for gentle algebras, semiorthogonal decompositions of their bounded derived categories can be understood as suitable cuts of the marked surfaces in their geometric models, which can be understood as a partial converse to the results in \cite{chang2020geometric} and \cite{chang2022recollements}.

Our main result is the following theorem:

\begin{theorem*}[Theorem \ref{TDecompositionsProperSequences}]
Let $\Lambda$ be a gentle algebra. There is a one-to-one correspondence between semiorthogonal decompositions of $\Db(\Lambda\modl)$ with arbitrarily many terms and proper sequences of good cuts of the marked surface associated to~$\Lambda$.
\end{theorem*}

The strategy of proving this theorem is rather straight-forward. Firstly, we simplify the situation by studying only two-term semiorthogonal decompositions, and we formulate necessary conditions for approximations of the indecomposable projectives of the gentle algebras. In the geometric model, these necessary conditions are then showed to imply that the indecomposable of approximations of projectives yield a bipartite admissible dissection of the marked surface.

On the other hand, we show that each bipartite admissible dissection of the marked surface gives rise to a semiorthogonal decomposition of $\Kb(\Lambda \proj)$, but many such dissection give the same semiorthogonal decomposition. Therefore, we seek an invariant to distinguish between \textit{equivalent} bipartite admissible dissections, and we find it in a way how the surface can be cut using the dissection. Using this invariant, called \textit{good cut}, we establish that there is a one-to-one correspondence between semiorthogonal decompositions of $\Kb(\Lambda \proj)$ with arbitrarily many terms and proper sequences of good cuts of the marked surface.
Subsequently, we prove that semiorthogonal decompositions of $\Db(\Lambda \modl)$ are uniquely determined extensions of semiorthogonal decompositions of $\Kb(\Lambda \proj)$.
%, and we extend the result to semiorthogonal decompositions with arbitrary number of terms.

\section{Preliminaries and conventions}
In this section, we define semiorthogonal decompositions and gentle algebras, and, subsequently, we review some results on the bounded derived category of a gentle algebra and its geometric model that will prove useful in the following sections. We also set some conventions used throughout this text as adopted from~\cite{arnesen2016morphisms}.

\subsection{Semiorthogonal decompositions}
In this subsection, we give a definition of a semiorthogonal decomposition of a triangulated category, and we review some facts about them. % that will prove useful later in this text.

\begin{definition}[Right and left admissible subcategory]\label{DAdmissibleSubcategory}
Let $\mcT$ be a triangulated category and $\mcC \subseteq \mcT$ be a full triangulated subcategory; we say that $\mcC$ is \emph{right admissible} if the inclusion functor of $\mcC$ into $\mcT$ admits a right adjoint. Similarly, $\mcC$ is \emph{left admissible} if the inclusion functor has a left adjoint.
\end{definition}

\begin{definition}[Semiorthogonal decomposition]\label{DSemiorthogonalDecomposition}
Let $\mcT$ be a triangulated category. Suppose that $\mcC_1, \dots, \mcC_m$ are full triangulated subcategories of $\mcT$ and that $\mcT_i$ is the smallest triangulated subcategory of $\mcT$ containing $\mcC_1, \dots, \mcC_i$. If $\mcT_m = \mcT$, $\mcC_i$ is right admissible subcategory of $\mcT_i$ for all $1 \leq i \leq m$, and $\Hom_\mcT(\mcC_j, \mcC_i) = 0$ for all $1 \leq i < j \leq n$, then we write $\langle \mcC_1, \dots, \mcC_m \rangle$, and we say that $\langle \mcC_1, \dots, \mcC_m \rangle$ is a \emph{semiorthogonal decomposition} of $\mcT$. 
\end{definition}

%For the purposes of this text, we require that all terms $\mcC_i$ in a semiorthogonal decomposition $\langle \mcC_1, \dots, \mcC_m \rangle$.

Before we recall some useful properties of semiorthogonal decompositions, we need to introduce some notation. Suppose that $\mcB$ is a subcategory of a triangulated category $\mcT$, we denote $\mcB^\bot = \{D \in \mcT \, | \, \forall B \in \mcB: \Hom_\mcT(B, D) = 0 \}$ and, similarly, ${^\bot}\mcB = \{D \in \mcT \, | \, \forall B \in \mcB: \Hom_\mcT(D, B) = 0 \}$.

\begin{proposition}[Lemma 3.1 in \cite{bergh2000abstract}]\label{PSemiorthogonalDecomposition}
Let $\mcT$ be a triangulated category, and let $\mcC_1$ and $\mcC_2$ be full triangulated subcategories of $\mcT$ such that $\Hom_\mcT(\mcC_2, \mcC_1) = 0$. Then the following are equivalent:
\begin{enumerate}[(i)]
    \item $\mcC_1$ and $\mcC_2$ generate $\mcT$ as a triangulated category.
    \item For every $D \in \mcT$ there exists a distinguished triangle $C_2 \to D \to C_1 \to C_2[1]$ with $C_2 \in \mcC_2$ and $C_1 \in \mcC_1$.
    \item The inclusion functor $i_1$ of $C_1$ into $\mcT$ has a left adjoint $j_1$, and $C_2 = {^\bot}\mcC_1$.
    \item The inclusion functor $i_2$ of $C_2$ into $\mcT$ has a right adjoint $j_2$, and $C_1 = \mcC_2^\bot$.
\end{enumerate}
Moreover, $i_2 j_2 D \to D \to i_1 j_1 D \to i_2 j_2 D[1]$ is the distinguished triangle in (ii); these distinguished triangles are also functorial.
\end{proposition}
\begin{proof}
This is a variation of Lemma 3.1 in \cite{bondal1990representation}. Although, the ideas behind it can be traced back as far as \cite{beilinson1982faisceaux}.
\end{proof}

It can be observed that a semiorthogonal decomposition $\langle \mcC_1, \dots, \mcC_m \rangle$ can be viewed as a sequence of two term semiorthogonal decompositions $\langle \mcT_i, \mcC_{i+1} \rangle$ of $\mcT_{i+1}$, $1 \leq i \leq m-1$, where $\mcT_i$ is the smallest triangulated subcategory of $\mcT$ containing $\mcC_1, \dots, \mcC_i$. %, and vice versa.

For more facts on triangulated categories, we refer the reader to \cite{holm2010triangulated}, for example. The reader may find more information on semiorthogonal decompositions in Chapter 10 therein.

\subsection{Gentle algebras}
In this subsection, we define gentle algebras, and, following \cite{arnesen2016morphisms}, we establish some conventions for working with them.

\begin{definition}\label{DGentleAlgebras}
Let $Q = (Q_0, Q_1)$ be a finite connected quiver. In accordance with \cite{arnesen2016morphisms}, we read paths in $Q$ from right to left. A finite-dimensional bound path algebra $\Lambda \cong kQ/I$ is called \emph{gentle} if:
\begin{enumerate}
    \item Each vertex $x \in Q_0$ is a source and a target of at most two arrows.
    \item For any arrow $a \in Q_1$ there is at most one arrow $b \in Q_1$ such that $ab \notin I$ and $c \in Q_1$ such that $ac \in I$.
    \item For any arrow $a \in Q_1$ there is at most one arrow $b \in Q_1$ such that $ba \notin I$ and $c \in Q_1$ such that $ca \in I$.
    \item The ideal $I$ is generated by paths of length two.
\end{enumerate}
\end{definition}

For the reader's convenience, we also follow the Convention 1.2 of \cite{arnesen2016morphisms}. If we denote the indecomposable projective left $\Lambda$ corresponding to a vertex $x \in Q_0$ as $P(x)$, then we can observe that there is bijection between paths from $x$ to $y$ that do not lie in $I$ and canonical basis elements of $\Hom_\Lambda(P(y), P(x))$ (given a path $p$ from $x$ to $y$ in $Q$ and not in $I$, we have a map from $P(y) \to P(x)$ that maps $u$, a path from $y$ in $Q$, to $up$). Therefore, by abuse of notation, we identify a path $p$ from $x$ to $y$ in $Q$ and not in $I$ with its corresponding basis element of $\Hom_\Lambda(P(y), P(x))$.

Following \cite{arnesen2016morphisms}, we also assume that the ground field $k$ is algebraically closed. Finally, we identify $\Db(\Lambda \modl)$ with the equivalent triangulated category $K^{-,b}(\Lambda \proj)$ of complexes of finitely generated projective $\Lambda$-modules bounded from the right with bounded cohomology. In the following text, $\Lambda = kQ/I$ will be a gentle algebra over an algebraically closed field $k$.

\begin{example}[A running example from~\cite{arnesen2016morphisms}]\label{ERunningExample}
For future reference, in order to illustrate our subsequent definitions and results, we introduce here a particular example of a gentle algebra.
This algebra is taken from one of our main sources, \cite[p.~6]{arnesen2016morphisms}, and
it is given by the following quiver with the quadratic relations indicated by the dotted arrows.
$$
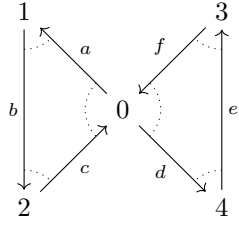
\begin{tikzpicture}
\node (V0) at ( 0.0, 0.0) {$0$};
\node (V1) at (-1.3, 1.3) {$1$};
\node (V2) at (-1.3,-1.3) {$2$};
\node (V3) at ( 1.3, 1.3) {$3$};
\node (V4) at ( 1.3,-1.3) {$4$};

\draw[dotted] (V0)+( 0.35,-0.35) arc (-45 :  45 : 0.5cm);
\draw[dotted] (V0)+(-0.35, 0.35) arc (135 : 225 : 0.5cm);
\draw[dotted] (V1)+(    0,-0.5 ) arc (270 : 315 : 0.5cm);
\draw[dotted] (V2)+(    0, 0.5 ) arc ( 90 :  45 : 0.5cm);
\draw[dotted] (V3)+(    0,-0.5 ) arc (270 : 225 : 0.5cm);
\draw[dotted] (V4)+(    0, 0.5 ) arc ( 90 : 135 : 0.5cm);

\path[commutative diagrams/.cd, every arrow, every label]
  (V0) edge node[swap] {$a$} (V1)
  (V1) edge node[swap] {$b$} (V2)
  (V2) edge node[swap] {$c$} (V0)
  (V0) edge node[swap] {$d$} (V4)
  (V4) edge node[swap] {$e$} (V3)
  (V3) edge node[swap] {$f$} (V0);
\end{tikzpicture}
$$
\end{example}

\subsection{Objects in the derived category of a gentle algebra}
In this subsection, we recall the combinatorial description of indecomposable objects in the derived category of a gentle algebra using homotopy strings and bands from Section~4 in \cite{bekkert2003indecomposables}. We also draw on discussions of this combinatorial description in Sections~2 in \cite{arnesen2016morphisms} and \cite{opper2018geometric}.

This subsection is structured in a straightforward way: at first, we define the combinatorial objects (homotopy strings, homotopy bands, and infinite homotopy strings) used for the characterization of objects in $\Db(\Lambda\proj)$, then we show how to construct corresponding complexes (string complexes, band complexes, and infinite string complexes), and, finally, we formulate the characterization theorem of \cite{bekkert2003indecomposables}.

At first, we consider formal inverses $\overline{a}$ to arrows in $a \in Q_1$ such that $s(\overline{a}) = t(a)$ and $t(\overline{a}) = s(a)$. Also, we define that $\overline{\overline{a}} = a$ for all $a \in Q_1$.
For a path $p = a_1 \dots a_m$ in $Q$ with $a_1, \dots, a_m \in Q_1$, the \emph{inverse path} $\overline{p}$ is defined as $\overline{a_m} \dots \overline{a_1}$.
Arrows in paths are composed from right to left as usual, that is $t(a_i)=s(a_{i-1})$ for all $1 < i \le m$.

\begin{definition}[{Walks and generalized walks; after discussion at the beginning of \cite[Section~4]{bekkert2003indecomposables}}]\label{DGeneralizedWalks}
A \textit{walk} $w$ is a sequence $w_n \dots w_1$ where each $w_i$ is an arrow in $Q_1$ or an inverse arrow in $\overline{Q_1}$ and where $t(w_{i-1}) = s(w_{i})$ for $1 < i \leq n$. Trivial walks corresponding to vertices $v \in Q_0$ are also allowed.

A \textit{generalized walk} $w$ is a sequence $w_n \dots w_1$ where each $w_i$ is a non-trivial path $p$ in $Q$ or an inverse thereof, $\overline{p}$ and where consecutive paths connect at endpoints, $t(w_{i-1}) = s(w_{i})$ for $1 < i \leq n$. The inverse $\overline{w}$ of $w$ is defined naturally.
\end{definition}

\begin{remark}\label{RConventionsForUnfoldedDiagrams}
Given a walk or a generalized walk $w = w_n \dots w_1$, we usually draw it, as in \cite{arnesen2016morphisms}, in the form of a diagram
$$\begin{tikzcd}
\bullet \arrow[dash]{r}{w_n} & \bullet \arrow[dash, dotted]{r} & \bullet \arrow[dash]{r}{w_1} & \bullet
\end{tikzcd}$$
where the line segment labeled with $w_i$ stands for an arrow pointing to the right if $w_i$ is a path in $Q$, and to the left if $w_i$ is an inverse of a path in $Q$.
Note that the arrows in such a diagram have the \emph{opposite} orientation compared to the corresponding paths in the quiver $Q$.
For instance, referring to Example~\ref{ERunningExample},
the generalized walk $w=\overline{e}\overline{f}cbp$, where $p$ is the path $p=af$, is depicted by the diagram
$$\begin{tikzcd}
\bullet \arrow[r, leftarrow, "e"] &
\bullet \arrow[r, leftarrow, "f"] &
\bullet \arrow[r, "c"] &
\bullet \arrow[r, "b"] &
\bullet \arrow[r, "af"] &
\bullet
\end{tikzcd}
$$

The reason for this peculiarity is that, following the conventions of~\cite{arnesen2016morphisms}, our diagrams represent sequences of maps between indecomposable projective left $\Lambda$-modules corresponding to individual paths or their inverses in the walk, and such maps are given by right multiplication by the paths.
In our particular example, the diagram of the generalized walk thus stands for
$$\begin{tikzcd}
P(4) \arrow[r, leftarrow, "-\cdot e"] &
P(3) \arrow[r, leftarrow, "-\cdot f"] &
P(0) \arrow[r, "-\cdot c"] &
P(2) \arrow[r, "-\cdot b"] &
P(1) \arrow[r, "-\cdot af"] &
P(3).
\end{tikzcd}
$$
\end{remark}

\begin{definition}[{Homotopy letters and homotopy strings; \cite[Subsection~4.1]{bekkert2003indecomposables}, \cite[Subsection~2.1]{arnesen2016morphisms}, \cite[Definition~2.1]{opper2018geometric}}]\label{DHomotopyStrings}
%A \textit{string} is a walk $w = w_n \dots w_1$ such that $w_{i} \neq \overline{w_{i-1}}$ for $1 < i \leq n$ and such there is no sub-walk $w'$ with $w'$ or $\overline{w'}$ lying in $I$. We say that $w$ is \emph{direct} if $w_i \in Q$ for all $1 \leq i \leq n$. The walk $w$ is \emph{inverse} if $\overline{w}$ is direct.
A \emph{direct homotopy letter} is a non-trivial path $p$ in $Q$ which is not contained in $I$.
An \emph{inverse homotopy letter} is by definition the inverse $\overline{p}$ of a direct homotopy letter $p$.
A \emph{homotopy letter} is a common name for direct and inverse homotopy letters.

A (finite reduced) \emph{homotopy string} $w = w_n \dots w_1$ is a (possibly trivial) generalized walk such that:
\begin{enumerate}[(i)]
    \item it consists of homotopy letters $w_i$, for $1 \leq i \leq n$;
    \item if $w_i$ and $w_{i-1}$ are both direct or inverse, then $w_i w_{i-1} \in I$ or $\overline{w_i w_{i-1}} \in I$, respectively, for $1 < i \leq n$;
    \item if $w_i$ is direct and $w_{i-1}$ is inverse or $w_i$ is inverse and $w_{i-1}$ is direct, then $w_i$ and $\overline{w_{i-1}}$ do not start with the same arrow or $\overline{w_i}$ and $w_{i-1}$ do not end with the same arrow, respectively, for $1 < i \leq n$.
\end{enumerate}
\end{definition}

\begin{definition}[{Gradings and homotopy bands; \cite[Subsection~4.1]{bekkert2003indecomposables}, \cite[Subsection~2.2]{arnesen2016morphisms}}]\label{DGrandingsHomotopyBands}
Let $w = w_n \dots w_1$ be a homotopy string. We define a \emph{grading} $\mu$ on $w$ to be a function $\mu\dd\{0, \dots, n\} \to \Z$ such that $\mu(i-1) = \mu(i) + 1$ if $w_i$ is a direct homotopy letter and $\mu(i-1) = \mu(i) - 1$ if $w_i$ is an inverse homotopy letter for all $1 \leq i \leq n$.
We also denote by $\overline{\mu}\dd\{0, \dots, n\} \to \Z$ the function given by $\overline{\mu}(i)=\mu(n-i)$, which is a grading on the inverse homotopy string $\overline{w}$.

A non-trivial homotopy string $w = w_n \dots w_1$ is called a \emph{homotopy band} if:
\begin{enumerate}
    \item its endpoints coincide, $s(w_1) = t(w_n)$;
    \item there exists a grading $\mu$ on $w$ such that $\mu(0) = \mu(n)$;
    \item $w$ is not a proper power of another homotopy string;
    \item one of $w_1, w_n$ is direct and the other is inverse.
\end{enumerate}
\end{definition}

\begin{remark} ~
\begin{enumerate}
\item Regarding homotopy strings and homotopy bands, we opted to for the terminology used in \cite{arnesen2016morphisms} and \cite{opper2018geometric}. In \cite{bekkert2003indecomposables}, they are referred to as generalized strings and generalized bands.

\item As usual, we will consider natural equivalence relations on the sets of homotopy strings and homotopy bands.
For homotopy strings, the equivalence $\simSt$ is given by identifying $w$ and $\overline{w}$.
For bands, we define $\simBa$ by considering two homotopy bands equivalent if they differ only by cyclic rotation and possibly taking inverse.
The equivalence relation extends in an obvious way to pairs $(w,\mu)$ where $w$ is a homotopy string or band and $\mu$ is a grading.
\end{enumerate}
\end{remark}

\begin{remark}\label{RGradedDiagrams}
%If we have a homotopy string or a band $w = w_n \dots w_1$ with grading $\mu$, we always assume that $w_i$ are direct or inverse strings that satisfy the conditions in Definition~\ref{DHomotopyStrings}. Consequently, homotopy strings and homotopy bands will be represented using the \textit{unfolded diagrams} of \cite{arnesen2016morphisms}. Given a homotopy string $w = w_1 \dots w_n$ with a grading $\mu$, we represent it with the following unfolded diagram:
If we have a homotopy string $w = w_n \dots w_1$ with a grading $\mu$,
we represent it using an \textit{unfolded diagram} as in \cite{arnesen2016morphisms}, which is just the diagram for the corresponding walk as explained in Remark~\ref{RConventionsForUnfoldedDiagrams} together with values of the function $\mu$ above the vertices of the diagram:
%If $w$ is a homotopy string, we represent it with the following unfolded diagram:
$$\begin{tikzcd}[row sep=0.1cm]
\mu(n) & \mu(n-1) & \mu(1) & \mu(0) \\
\bullet \arrow[dash]{r}{w_n} & \bullet \arrow[dash, dotted]{r} & \bullet \arrow[dash]{r}{w_1} & \bullet
\end{tikzcd}$$
%The arrow $w_i$ points to the right if $w_i$ is a direct string, and it points to the left if $w_i$ is an inverse string. For convenience, the grading will be often omitted from the diagrams.
A particular example using the homotopy string from Remark~\ref{RConventionsForUnfoldedDiagrams} is
$$\begin{tikzcd}[row sep=0.1cm]
2 & 1 & 0 & 1 & 2 & 3
\\
\bullet \arrow[r, leftarrow, "e"] &
\bullet \arrow[r, leftarrow, "f"] &
\bullet \arrow[r, "c"] &
\bullet \arrow[r, "b"] &
\bullet \arrow[r, "af"] &
\bullet
\end{tikzcd}$$

If $w = w_n \dots w_1$ is a homotopy band with a grading $\mu$, its \emph{unfolded diagram} by definition infinitely repeats both to the left and to the right. We usually also decorate the homotopy letter $w_1$ with a fixed scalar $\lambda \in k^{\times}$, because this is an additional datum needed to define a corresponding one-dimensional band complex (see Definition~\ref{DOneDimBands} below):
$$\begin{tikzcd}[row sep=0.1cm, arrows = {decorate = false, decoration={snake, segment length=2mm, amplitude=0.25mm}}]
& & \mu(0) & \mu(n-1) & \mu(1) & \mu(0) & \\
{} & \arrow[dash, decorate = true]{l} \bullet \arrow[dash]{r}{\lambda w_1} & \bullet \arrow[dash]{r}{w_n} & \bullet \arrow[dash, dotted]{r} & \bullet \arrow[dash]{r}{\lambda w_1} & \bullet \arrow[dash]{r}{w_n} & \bullet \arrow[dash, decorate = true]{r} & {}
\end{tikzcd}$$
%The meaning of the scalar $\lambda$ will become clear when we describe construction of one-dimensional band complexes using homotopy bands later in this subsection.
In order to see a particular example, consider the band $w=\overline{d}\overline{e}\overline{f}cba$ for the gentle algebra from Example~\ref{ERunningExample} (see \cite[Example 2.3]{arnesen2016morphisms}). The corresponding unfolded diagram then reads
$$\begin{tikzcd}[row sep=0.1cm, arrows = {decorate = false, decoration={snake, segment length=2mm, amplitude=0.25mm}}]
& & 3 & 2 & 1 & 0 & 1 & 2 & 3 & &
\\
{} &
\arrow[dash, decorate = true]{l}
\bullet  \arrow[r, "\lambda a"] &
\bullet \arrow[r, leftarrow, "d"] &
\bullet \arrow[r, leftarrow, "e"] &
\bullet \arrow[r, leftarrow, "f"] &
\bullet \arrow[r, "c"] &
\bullet \arrow[r, "b"] &
\bullet \arrow[r, "\lambda a"] &
\bullet \arrow[r, leftarrow, "d"] &
\bullet \arrow[dash, decorate = true]{r} &
{}
\end{tikzcd}$$
\end{remark}

If the underlying algebra $\Lambda$ has infinite global dimension, we need to work with infinite homotopy strings in addition to homotopy strings and homotopy bands in order to describe all indecomposable objects in $\Db(\Lambda\modl)$. Infinite homotopy strings arise from oriented cycles $a_m \dots a_1$ in $Q$ that have full relations, meaning that $a_i a_{i-1} \in I$ for $1 < i \leq n$ and $a_1 a_m \in I$.

\begin{definition}[{Infinite homotopy strings; \cite[Subsection 4.3 above Lemma 5]{bekkert2003indecomposables}, \cite[Definition~2.6 and the following discussion]{arnesen2016morphisms}}]\label{DInfiniteHomotopyStrings}
Let $w = w_n \dots w_1$ be a homotopy string such that $w_n$ is direct (in particular $w$ is non-trivial).
%with a grading $\mu$ such that $\mu(w_n) \leq \mu(w_i)$ for all $0 \leq i \leq n$ (specifically, $w_n$ is direct).
We say that $w$ is \emph{left resolvable} if there exists a cycle $a_m \dots a_1$ in $Q$ with full relations and $a_1 w$ is a homotopy string. % for some $1 < j \leq m$.
Moreover, we call $w$ \emph{primitive left resolvable} if $w$ is left resolvable, but $w_{n-1} \dots w_1$ is not.
The notions of \emph{right resolvability} and \emph{primitive right resolvability} are defined dually so that $w$ is (primitive) right resolvable if and only if the inverse homotopy string $\overline{w}$ is (primitive) left resolvable.

For $w$ that is left resolvable
%with $a_j$, $1 < j \leq m$, from
with a cycle $a_m \dots a_1$ in $Q$ with full relations, we form a \emph{left infinite homotopy} string ${^\infty w}$ by adding countably many copies of $a_m \dots a_1$ to the left of $w$ obtaining:
$$ {^\infty w} = \dots a_m a_{m-1} \dots a_1 \dots a_m a_{m-1} \dots a_1 w_1 \dots w_n.$$
For a right resolvable homotopy string, we form a \emph{right infinite homotopy} string $w^\infty$ in similar way.
If a homotopy string is both left and right resolvable, we can also form a \emph{two-sided infinite homotopy string} ${^\infty w^\infty}$ by combining the two constructions.
\end{definition}

\begin{remark} ~
\begin{enumerate}
\item Infinite homotopy strings are not specifically named in \cite{bekkert2003indecomposables}, so we opted to the term used in \cite{arnesen2016morphisms} and \cite{opper2018geometric}, which is compatible with naming of homotopy strings and homotopy bands.

\item Every left infinite homotopy string ${^\infty w}$ is determined by a unique primitive left resolvable homotopy string $w$. Analogously, every right infinite homotopy string $w^\infty$ is determined by a unique primitive right resolvable homotopy string $w$, and similarly for two-sided infinite homotopy strings.

\item If $w$ is left resolvable, any grading $\mu\dd\{0, \dots, n\} \to \Z$ uniquely extends to a grading ${^\infty\mu}\dd\{0, \dots, n, n+1, n+2, \dots\} \to \Z$ of the left infinite homotopy string ${^\infty w}$, and we can depict ${^\infty w}$ with ${^\infty\mu}$ in the form of an unfolded diagram analogous to that of ordinary homotopy strings as in Remark~\ref{RGradedDiagrams}.
Analogous comments apply to right and two-sided infinite homotopy strings.

For instance, if $w=a\overline{d}$ for the algebra from Example~\ref{ERunningExample} with the grading $\mu\dd\{0, \dots, n\} \to \Z$ such that $\mu(0)=0$, then we can extend $w$ to ${^\infty w^\infty}$ and $\mu$ to ${^\infty\mu^\infty}\dd\Z\to\Z$ and draw ${^\infty w^\infty}$ and ${^\infty \mu^\infty}$ in the form of an unfolded diagram
$$\begin{tikzcd}[row sep=0.1cm, column sep=0.7cm, arrows = {decorate = false, decoration={snake, segment length=2mm, amplitude=0.25mm}}]
& & -2 & -1 & 0 & 1 & 0 & -1 & -2 &
\\
{} &
\arrow[dash, decorate = true]{l}
\bullet \arrow[r, "a"] &
\bullet \arrow[r, "c"] &
\bullet \arrow[r, "b"] &
\bullet \arrow[r, "a"] &
\bullet \arrow[r, leftarrow, "d"] &
\bullet \arrow[r, leftarrow, "e"] &
\bullet \arrow[r, leftarrow, "f"] &
\bullet \arrow[r, leftarrow, "d"] &
\bullet \arrow[dash, decorate = true]{r} &
{}
\end{tikzcd}$$

\item We can define formal inverses of infinite homotopy strings or infinite homotopy strings with a grading in a natural way. If $w$ is a primitive left resolvable homotopy string, we put $\overline{^\infty w} = \overline{w}^\infty$. Similarly, if $w$ is both primitive left and right resolvable, we put $\overline{^\infty w^\infty} = {^\infty\overline{w}^\infty}$.
That is, left infinite homotopy strings invert to right infinite homotopy strings and vice versa; whereas, inverting a two-sided infinite homotopy string produces yet another two-sided infinite homotopy string.
We again define an equivalence relation $\simInfSt$ on the set of all infinite homotopy strings as well as on the set of all pairs consisting of an infinite homotopy string and a grading, which identifies a (graded) infinite homotopy string with its inverse.
\end{enumerate}
\end{remark}

%\begin{remark}
%Infinite homotopy strings can be also graphically represented using an unfolded diagram. Suppose we can form a left infinite homotopy string from $w = w_n \dots w_1$ by adding infinitely many copies of a cycle $a_m \dots a_1$ in $Q$ with full relations, then the resulting left infinite homotopy string has the following unfolded diagram:
%$$\begin{tikzcd}
%N: & \arrow[dotted, dash]{r} & \bullet \arrow{r}{a_m} & \bullet \arrow[dotted, dash]{r} & \bullet \arrow{r}{a_1} & \bullet \arrow{r}{w_n} & \bullet \arrow[dotted, dash]{r} & \bullet \arrow[dash]{r}{w_1} & \bullet
%\end{tikzcd}$$
%\end{remark}

\begin{convention}
From now on we will always assume that homotopy strings, infinite homotopy strings, and homotopy bands are equipped with some grading (which may remain only implicit in some arguments).
\end{convention}

\begin{definition}[{{}String and infinite string complexes; \cite[Definition~2]{bekkert2003indecomposables}, \cite[Subsection 2.1]{arnesen2016morphisms}}]\label{DStringComplexes}
To a homotopy string $w = w_n \dots w_1$ or an infinite homotopy string $w$ with grading $\mu$ we associate a complex of projectives $P_{(w, \mu)}$ as follows:
\begin{itemize}
    \item the projective module in cohomological degree $j$ is a (necessarily finite) direct sum of indecomposable projectives:
    $$\bigoplus_{\mu(i) = j} P(t(w_i))$$
%    $$\bigoplus_{0 \leq i \leq n: \mu(i) = j} P(t(w_i))$$
    where $t(w_0)$ stands for $s(w_1)$ by convention if $w$ is finite or only left infinite;
    %and $\mu(0) = \mu(1) + 1$ if $w_1$ is direct and $\mu(0) = \mu(1) - 1$ if $w_1$ is inverse;
    \item the differential is defined componentwise as follows: its only non-zero components are $P(t(w_i)) \overset{-\cdot w_i}{\longrightarrow} P(s(w_i))$ if $w_i$ is direct and $P(s(w_i)) \overset{-\cdot\overline{w_i}}{\longrightarrow} P(t(w_i))$ if $w_i$ is inverse, for all indices $i$.
\end{itemize}
If $w$ is finite, call the resulting complex $P_{(w, \mu)}$ a \emph{string complex}.
If $w$ is infinite, we call it an \emph{infinite string complex}.
%An analogous construction can be applied to an infinite homotopy string $w'$ with a grading $\mu'$ to produce an infinite string complex $P_{(w', \mu')}$.
\end{definition}

\begin{remark} ~
\begin{enumerate}
\item If $w$ is a homotopy string, then clearly $P_{(w, \mu)}\in\Kb(\Lambda\proj)$.
If $w$ is infinite, then $P_{(w, \mu)}$ has finitely generated projective components, it is bounded on the right and has a bounded cohomology by~\cite[Lemma~5(1)]{bekkert2003indecomposables}. Hence, up to quasi-isomorphism, we can consider $P_{(w, \mu)}$ as an object of $\Db(\Lambda\modl)$.
\item Two string complexes $P(w_1, \mu_1)$ and $P(w_2, \mu_2)$ are isomorphic if and only if $(w_1, \mu_1) \simSt (w_2, \mu_2)$.
Two infinite string complexes $P(w_1, \mu_1)$ and $P(w_2, \mu_2)$ are isomorphic if and only if $(w_1, \mu_1) \simInfSt (w_2, \mu_2)$.
A string complex is never isomorphic to an infinite string complex.
We refer to \cite[Theorem 3]{bekkert2003indecomposables} for details.
\end{enumerate}
\end{remark}

\begin{definition}[{One-dimensional band complexes; \cite[Definition 3]{bekkert2003indecomposables}, \cite[Subsection~2.2]{arnesen2016morphisms}}]\label{DOneDimBands}
To a homotopy band $w = w_n \dots w_1$ with grading $\mu$ and a fixed scalar $\lambda \in k^{\times}$, we associate a complex of projectives $B_{(w, \mu), \lambda, 1}$, where we consider the indices $i$ of modulo $n$:
\begin{itemize}
    \item the projective module in degree $j$ is a direct sum of indecomposable projectives:
    $$\bigoplus_{0 \leq i < n: \mu(i) = j} P(t(w_i)),$$
    in the same way as for string complexes;
    \item the differential is defined similarly to string complexes if the length of the string is $n\ge 3$:
    its only non-zero components are $P(t(w'_i)) \overset{-\cdot w'_i}{\longrightarrow} P(s(w'_i))$, where $w'_i = w_i$ if $w_i$ is direct and $w'_i = \overline{w_i}$ if $w_i$ is inverse, with the exception of $i=1$ where the component is $P(t(w'_1)) \overset{-\cdot \lambda w'_1}{\longrightarrow} P(s(w'_1))$;
    \item in the somewhat degenerate case $n=2$, we necessarily have $s(w_1)=s(w_2)$, $t(w_1)=t(w_2)$ and the entire complex has two terms and is of the form
    $$\begin{tikzcd}
      P(t(w'_1)) \arrow[rr, "-\cdot(w'_2 + \lambda w'_1)"] &&
      P(s(w'_1)).
    \end{tikzcd}$$
%    \item the components of the differential are also the same as for string complexes except for the component corresponding to $w_1$: if $n \geq 3$, we put ; if $n = 2$, there is only one non-zero component of the differential given by $w_n + \lambda w_1$.
\end{itemize}
We refer to such complexes as to \emph{one-dimensional band complexes}. 
\end{definition}

\begin{remark}\label{RHigherDimensionalBands}
For each homotopy band $w = w_n \dots w_1$ with grading $\mu$, a fixed scalar $\lambda \in k^{\times}$ and $n>1$, there are also indecomposable \emph{higher-dimensional band complexes} $B_{(w, \mu), \lambda, n}\in\Kb(\Lambda\proj)$.
We refer to \cite[Definition~3]{bekkert2003indecomposables} and especially \cite[Section 5]{arnesen2016morphisms} for a detailed account.
As far as we are concerned, it is only important that for each $w$, $\mu$ and $\lambda$ and $n \ge 1$, we have a triangle in $\Kb(\Lambda\proj)$ of the form
$$\begin{tikzcd}
B_{(w, \mu), \lambda, n} \arrow{r} &
B_{(w, \mu), \lambda, n-1} \arrow{r} \oplus B_{(w, \mu), \lambda, n+1} \arrow{r} &
B_{(w, \mu), \lambda, n} \arrow{r} &
B_{(w, \mu), \lambda, n}[1],
\end{tikzcd}$$
where $B_{(w, \mu), \lambda, 0}=0$ by convention if $n=1$.
In particular, if $\mcC\subseteq\Kb(\Lambda\proj)$ is a full subcategory closed under extensions and summands, then $B_{(w, \mu), \lambda, n}\in\mcC$ for some $n\ge 1$ if and only if $B_{(w, \mu), \lambda, n}\in\mcC$ for \emph{all} $n\ge 1$.
%However, higher-dimensional complexes will not be of interest for the purposes of this article, so we refer the reader to the aforementioned references for more detail.
\end{remark}

\begin{remark}
The only non-trivial isomorphisms between band complexes are $B_{(w_1, \mu_1), \lambda, n} \cong B_{(w_2,\mu_2), \lambda^\varepsilon, n}$, where $(w_1,\mu_1) \simBa (w_2,\mu_2)$ and $\varepsilon=\pm1$, where the sign depends on whether the passage from $w_1$ to $w_2$ can be achieved only by a cyclic rotation or whether we need to pass to an inverse homotopy band as well.
There are no isomorphisms between band complexes and homotopy string complexes or infinite homotopy string complexes.
\end{remark}

Finally, having defined all necessary notions, we can state the result of \cite{bekkert2003indecomposables} that characterizes indecomposable objects of $\Db(\Lambda \modl)$ as string, infinite string, and band complexes:

\begin{theorem}[{Characterization of indecomposables in $\Db(\Lambda\modl)$; \cite[Theorem~3]{bekkert2003indecomposables}}]\label{TIndecomposablesStringsBands}
There are one-to-one correspondences between:
\begin{enumerate}
\item isomorphism classes of indecomposable objects of $\Kb(\Lambda\proj)$ and isomorphism classes of string and band complexes;
\item isomorphism classes of indecomposable objects of $\Db(\Lambda\modl)$ which are not in $\Kb(\Lambda\proj)$ and isomorphism classes of infinite string complexes.
\end{enumerate}
\end{theorem}

\subsection{Morphisms in the derived category of a gentle algebra}
After having described the indecomposable objects in $\Db(\Lambda\modl)$, we focus on morphisms between them. In this subsection, we present a combinatorial description of morphisms, as given in \cite{arnesen2016morphisms}.

We structure this subsection similarly to the previous one. We begin by giving a combinatorial description of certain distinguished maps between string, infinite string, and one-dimensional complexes, and we conclude by stating the main result of \cite{arnesen2016morphisms} that gives a basis of maps between both in the category of complexes over $\Lambda\proj$ and the bounded derived category $\Db(\Lambda\modl)$ in terms of those special maps.

Before we define any maps between string, infinite string, and one-dimensional complexes, we note that by virtue of their definition we can describe the maps using their unfolded diagrams (cf. discussion at the beginning of Section 3 in \cite{arnesen2016morphisms}). Suppose that $M$ and $N$ are string, infinite string, or one-dimensional band complexes and that $f\colon M \to N$ is a map of complexes between them. The map $f$ consists of maps $f^i: M^i \to N^i$; since we have defined $M^i$ and $N^i$ as direct sums of certain indecomposable projectives, $f^i$ can be thought of as a matrix of maps between them given by linear combination of paths between corresponding vertices in the unfolded diagrams of $M$ and $N$.

We note that there are examples of how maps between unfolded diagrams translate to maps between corresponding complexes throughout Section 3.1 in \cite{arnesen2016morphisms}.

In accordance with \cite{arnesen2016morphisms}, we omit the scalar $\lambda \in k^\times$ for one-dimensional complexes from the unfolded diagrams to allow for more clarity in the definitions. However, the reader is encouraged to keep in mind that it is still implicitly present.

\begin{definition}[Single and singleton single maps; Definition 3.1 and 3.7 in \cite{arnesen2016morphisms}]\label{DSingletonSingleMaps}
Suppose that $M$ and $N$ are string, infinite string, or one-dimensional band complexes and that $p$ is a non-stationary path in the quiver $Q$. Then the following configuration of unfolded diagrams:
$$\begin{tikzcd}[arrows = {decorate = false, decoration={snake, segment length=2mm, amplitude=0.25mm}}, /tikz/column 1/.append style={anchor=base east}]
    M: & \arrow[dash, decorate = true]{l} \bullet \arrow[dash]{r}{m_{i+1}} & \bullet \arrow[]{d}{p} \arrow[dash]{r}{m_i} & \bullet \arrow[dash, decorate = true]{r} & {} \\
    N: & \arrow[dash, decorate = true]{l} \bullet \arrow[dash]{r}{n_{i+1}} & \bullet \arrow[dash]{r}{n_i} & \bullet \arrow[dash, decorate = true]{r} & {}
\end{tikzcd}$$
gives rise to a \emph{single map} from $M$ to $N$ if the following conditions are met:
\begin{itemize}
	\item[\textbf{(L1)}] if $m_{i+1}$ is direct, then $m_{i+1} p = 0$;
	\item[\textbf{(R1)}] if $m_i$ is inverse, then $m_i p = 0$;
	\item[\textbf{(L2)}] if $n_{i+1}$ is inverse, then $p n_{i+1} = 0$;
	\item[\textbf{(R2)}] if $n_i$ is direct, then $p n_i = 0$.
\end{itemize}
The map is called \emph{singleton single map} if it arises from one of the following four configurations (up to inverting one of the homotopy strings) that satisfy the conditions above:
\begin{enumerate}[(i)]
    \item $$\begin{tikzcd}[arrows = {decorate = false, decoration={snake, segment length=2mm, amplitude=0.25mm}}, /tikz/column 1/.append style={anchor=base east}]
    M: & \arrow[dash, decorate = true]{l} \bullet \arrow[dash]{r}{m_{i+1}} & \bullet \arrow[]{d}{p} & \\
    N: & & \bullet \arrow[dash]{r}{n_i} & \bullet \arrow[dash, decorate = true]{r} & {}
\end{tikzcd}$$

	\item $$\begin{tikzcd}[arrows = {decorate = false, decoration={snake, segment length=2mm, amplitude=0.25mm}}, /tikz/column 1/.append style={anchor=base east}]
    M: & \arrow[dash, decorate = true]{l} \bullet \arrow[dash]{r}{m_{i+1}} & \bullet \arrow[]{d}{p} \arrow{r}{m_i}[swap]{=p p_R} & \bullet \arrow[dash, decorate = true]{r} & {} \\
    N: & & \bullet \arrow[dash]{r}{n_i} & \bullet \arrow[dash, decorate = true]{r} & {}
\end{tikzcd}$$

    \item $$\begin{tikzcd}[arrows = {decorate = false, decoration={snake, segment length=2mm, amplitude=0.25mm}}, /tikz/column 1/.append style={anchor=base east}]
    M: & \arrow[dash, decorate = true]{l} \bullet \arrow[dash]{r}{m_{i+1}} & \bullet \arrow[]{d}{p} & \\
    N: & \arrow[dash, decorate = true]{l} \bullet \arrow{r}{n_{i+1}}[swap]{=p_L p} & \bullet \arrow[dash]{r}{n_i} & \bullet \arrow[dash, decorate = true]{r} & {}
\end{tikzcd}$$

    \item $$\begin{tikzcd}[arrows = {decorate = false, decoration={snake, segment length=2mm, amplitude=0.25mm}}, /tikz/column 1/.append style={anchor=base east}]
M: & \arrow[dash, decorate = true]{l} \bullet \arrow[dash]{r}{m_{i+1}} & \bullet \arrow[]{d}{p} \arrow{r}{m_i}[swap]{= p p_R} & \bullet \arrow[dash, decorate = true]{r} & {}\\
N: & \arrow[dash, decorate = true]{l} \bullet \arrow{r}{n_{i+1}}[swap]{= p_L p} & \bullet \arrow[dash]{r}{n_i} & \bullet \arrow[dash, decorate = true]{r} & {}
\end{tikzcd}$$
\end{enumerate}
The paths $p$, $p_L$ and $p_R$ are required to be non-stationary. Moreover, in cases (i) and (ii), if $m_{i+1}$ is inverse, $p$ is not an initial homotopy substring of $m_{i+1}$ and $m_{i+1}$ is not an initial homotopy substring of $p$; in cases (i) and (iii), if $n_i$ is inverse, $p$ is not a terminal homotopy substring of $n_i$ and $n_i$ is not $n_i$ is not a terminal homotopy substring of~$p$.
\end{definition}

\begin{remark}
In the unfolded diagrams above and also later in the text, we follow the convention that the arrows without specified direction, which are depicted as $\begin{tikzcd} \bullet \arrow[dash]{r} & \bullet \end{tikzcd}$, may stand also for a zero component.
On the other hand, arrows indicated by $\begin{tikzcd} \bullet \arrow{r} & \bullet \end{tikzcd}$  and $\begin{tikzcd} \bullet & \arrow{l} \bullet \end{tikzcd}$ must be non-zero.

The fact that the unfolded diagram possibly continues to the left, to the right, and in between is indicated by $\begin{tikzcd}[arrows = {decorate = false, decoration={snake, segment length=2mm, amplitude=0.25mm}}] {} & \arrow[dash, decorate = true]{l} \bullet \end{tikzcd}$, $\begin{tikzcd}[arrows = {decorate = false, decoration={snake, segment length=2mm, amplitude=0.25mm}}] \bullet \arrow[dash, decorate = true]{r} & {} \end{tikzcd}$, and $\begin{tikzcd}[arrows = {decorate = false, decoration={snake, segment length=2mm, amplitude=0.25mm}}] \bullet \arrow[dash, dotted]{r} & \bullet \end{tikzcd}$, respectively.
\end{remark}

\begin{remark}
The last two conditions in the definition above guarantee two things: due to their first parts, the fact that the configurations (i) to (iv) for singleton single maps are disjoint, and, due to their second parts, non-existence of non-trivial chain homotopies to other single or double maps, so the maps of complexes are indeed singleton in their homotopy class.
\end{remark}

\begin{definition}[Double and singleton double maps; Definition 3.3 and 3.8 in \cite{arnesen2016morphisms}]\label{DSingletonDoubleMaps}
Suppose that $M$ and $N$ are string, infinite string, or one-dimensional band complexes and that $f$ is a non-stationary path in the quiver $Q$. Then the following configuration of unfolded diagrams:
$$\begin{tikzcd}[arrows = {decorate = false, decoration={snake, segment length=2mm, amplitude=0.25mm}}, /tikz/column 1/.append style={anchor=base east}]
M: & \arrow[dash, decorate = true]{l} \bullet \arrow[dash]{r}{m_{i+1}} & \bullet \arrow[']{d}{p_L} \arrow{r}{m_i} & \bullet \arrow[]{d}{p_R} \arrow[dash]{r}{m_{i-1}} & \bullet \arrow[dash, decorate = true]{r} & {} \\
N: & \arrow[dash, decorate = true]{l} \bullet \arrow[dash]{r}{n_{i+1}} & \bullet \arrow{r}{n_i} & \bullet \arrow[dash]{r}{n_{i-1}} & \bullet \arrow[dash, decorate = true]{r} & {}
\end{tikzcd}$$
gives rise to a \emph{double map} from $M$ to $N$ if the following conditions are met:
\begin{itemize}
    \item[\textbf{(C)}] $p_L n_i = m_i p_R$;
	\item[\textbf{(L1)}] if $m_{i+1}$ is direct, then $m_{i+1} p_L = 0$;
	\item[\textbf{(R1)}] if $m_{i-1}$ is inverse, then $m_{i-1} p_R = 0$;
	\item[\textbf{(L2)}] if $n_{i+1}$ is inverse, then $p_L n_{i+1} = 0$;
	\item[\textbf{(R2)}] if $n_{i-1}$ is direct, then $p_R n_{i-1} = 0$.
\end{itemize}
The map is called \emph{singleton double map} if it arises from the configuration that satisfies the conditions above and, moreover, there exists a non-stationary path $p'$ so that:
$$\begin{tikzcd}[arrows = {decorate = false, decoration={snake, segment length=2mm, amplitude=0.25mm}}, /tikz/column 1/.append style={anchor=base east}]
M: & \arrow[dash, decorate = true]{l} \bullet \arrow[dash]{r}{m_{i+1}} & \bullet \arrow[']{d}{p_L} \arrow{r}{m_i}[swap]{=p_L p'} & \bullet \arrow[]{d}{p_R} \arrow[dash]{r}{m_{i-1}} & \bullet \arrow[dash, decorate = true]{r} & {} \\
N: & \arrow[dash, decorate = true]{l} \bullet \arrow[dash]{r}{n_{i+1}} & \bullet \arrow{r}{n_i}[swap]{=p' p_R} & \bullet \arrow[dash]{r}{n_{i-1}} & \bullet \arrow[dash, decorate = true]{r} & {}
\end{tikzcd}$$
\end{definition}

\begin{definition}[Graph maps; Definition 3.9 in \cite{arnesen2016morphisms}]\label{DGraphMaps}
Suppose that $M$ and $N$ are string, infinite string, or one-dimensional band complexes. Consider a maximal overlap of the unfolded diagrams as follows:
$$\begin{tikzcd}[arrows = {decorate = false, decoration={snake, segment length=2mm, amplitude=0.25mm}}, /tikz/column 1/.append style={anchor=base east}]
M: & \arrow[dash, decorate = true]{l} \bullet \arrow[dash]{r}{m_L} & \bullet \arrow[equal]{d} \arrow[dash]{r}{\ell_p} & \bullet \arrow[equal]{d} \arrow[dash]{r}{\ell_{p-1}} & \bullet \arrow[dash, dotted]{r} & \bullet \arrow[dash]{r}{\ell_2} & \bullet \arrow[equal]{d} \arrow[dash]{r}{\ell_1} & \bullet \arrow[equal]{d} \arrow[dash]{r}{m_R} & \bullet \arrow[dash, decorate = true]{r} & {}\\
N: & \arrow[dash, decorate = true]{l} \bullet \arrow[dash]{r}{n_L} & \bullet \arrow[dash]{r}{\ell_p} & \bullet \arrow[dash]{r}{\ell_{p-1}} & \bullet \arrow[dash, dotted]{r} & \bullet \arrow[dash]{r}{\ell_2} & \bullet \arrow[dash]{r}{\ell_1} & \bullet \arrow[dash]{r}{n_R} & \bullet \arrow[dash, decorate = true]{r} & {}
\end{tikzcd}$$
It gives rise to a \emph{graph map} from $M$ to $N$ if one the following endpoint conditions is met
and one of dual endpoint conditions \textbf{(RG1)}, \textbf{(RG2)}, or \textbf{(RG$\infty$)} is met as well:
\begin{itemize}
	\item[\textbf{(LG1)}] if $m_L$ and $n_L$ are both direct or both inverse, then there exists a non-stationary path $f_L$ such that $m_L = f_L n_L$ or $n_L = m_L f_L$, respectively;
	\item[\textbf{(LG2)}] if $m_L$ and $n_L$ are neither both direct nor both inverse, then $m_L$ is zero or inverse and $n_L$ is zero or direct; 
	\item[\textbf{(LG$\infty$)}] both strings continue infinitely to the left: the diagram continues infinitely to the left with commuting squares in which its vertical maps are isomorphisms;
\end{itemize}
\end{definition}

\begin{definition}[Quasi-graph maps; Definition 3.11 in \cite{arnesen2016morphisms}]\label{DQuasiGraphMaps}
Suppose that $M$ and $N$ are string, infinite string, or one-dimensional band complexes. Then a maximal overlap of the unfolded diagrams as follows:
$$\begin{tikzcd}[arrows = {decorate = false, decoration={snake, segment length=2mm, amplitude=0.25mm}}, /tikz/column 1/.append style={anchor=base east}]
M: & \arrow[dash, decorate = true]{l} \bullet \arrow[dash]{r}{m_L} & \bullet \arrow[equal]{d} \arrow[dash]{r}{\ell_p} & \bullet \arrow[equal]{d} \arrow[dash]{r}{\ell_{p-1}} & \bullet \arrow[dash, dotted]{r} & \bullet \arrow[dash]{r}{\ell_2} & \bullet \arrow[equal]{d} \arrow[dash]{r}{\ell_1} & \bullet \arrow[equal]{d} \arrow[dash]{r}{m_R} & \bullet \arrow[dash, decorate = true]{r} & {}\\
N: & \arrow[dash, decorate = true]{l} \bullet \arrow[dash]{r}{n_L} & \bullet \arrow[dash]{r}{\ell_p} & \bullet \arrow[dash]{r}{\ell_{p-1}} & \bullet \arrow[dash, dotted]{r} & \bullet \arrow[dash]{r}{\ell_2} & \bullet \arrow[dash]{r}{\ell_1} & \bullet \arrow[dash]{r}{n_R} & \bullet \arrow[dash, decorate = true]{r} & {}
\end{tikzcd}$$
gives rise to a \emph{quasi-graph map} from $M$ to $N[1]$ provided that
\begin{itemize}
  \item none of the endpoint conditions \textbf{(LG1)}, \textbf{(LG2)}, \textbf{(LG$\infty$)} \textbf{(RG1)}, \textbf{(RG2)}, or \textbf{(RG$\infty$)} holds,
  \item if $p=0$ and both $m_L$ and $n_R$ are direct, then $m_L n_R=0$,
  \item if $p=0$ and both $m_R$ and $n_L$ are inverse, then $m_R n_L=0$.
\end{itemize}
The quasi-graph map is represented by single maps from $M$ to $N[1]$ given by $\ell_p, \dots, \ell_1$ and possibly also double maps arising from $m_L, n_L$ and $m_R, n_R$ (cf. Definition 3.12 in \cite{arnesen2016morphisms}).
\end{definition}

\begin{remark}
The additional conditions in the last definition when $p=0$ do not seem to be treated in the literature, but they are necessary to define the corresponding single map $M\to N[1]$. Consider for example the quasi-graph map situation:
$$\begin{tikzcd}[arrows = {decorate = false, decoration={snake, segment length=2mm, amplitude=0.25mm}}, /tikz/column 1/.append style={anchor=base east}]
M: & \arrow[dash, decorate = true]{l} \bullet \arrow{r}{m_L} & \bullet \arrow[equal]{d} & & {}\\
N: & & \bullet \arrow{r}{n_R} & \bullet \arrow[dash, decorate = true]{r} & {}
\end{tikzcd}$$
The corresponding map $M\to N[1]$ is then represented for example by the following unfolded diagram, and this makes sense only provided that $m_L n_R=0$:
$$\begin{tikzcd}[arrows = {decorate = false, decoration={snake, segment length=2mm, amplitude=0.25mm}}, /tikz/column 1/.append style={anchor=base east}]
M: & \arrow[dash, decorate = true]{l} \bullet \arrow{r}{m_L} & \bullet \arrow{d}{n_R} & & {}\\
N: & \bullet \arrow{r}{n_R} & \bullet \arrow[dash]{r} & \bullet \arrow[dash, decorate = true]{r} & {}
\end{tikzcd}$$
\end{remark}

\begin{theorem}[Basis of maps between string, infinite string, and one-dimensional band complexes; Proposition 4.1 and Theorem 3.15 in \cite{arnesen2016morphisms}]\label{TBasisMaps}
Suppose that $M$ and $N$ are string, infinite string, or one-dimensional band complexes. Then the following statements hold:
\begin{enumerate}[(i)]
\item graph maps, single maps, and double maps form a $k$-linear basis of the space $\Hom_{C(\Lambda\proj)}(M, N)$;
\item graph maps, singleton single maps, singleton double maps, and quasi-graph maps form a $k$-linear basis of $\Hom_{\Db(\Lambda\modl)}(M, N)$;
\end{enumerate}
\end{theorem}

\begin{remark}
In light of the theorem above, we refer to graph maps, single maps, and double maps as to basis maps. A graph map is not homotopic to a scalar multiple of another basis map as proved in Subsection 4.4 in \cite{arnesen2016morphisms}. Proposition 4.8 in \cite{arnesen2016morphisms} illustrates that quasi-graph maps represent homotopy classes of homotopy non-trivial single or double maps that are homotopic to some other single or double maps (cf. Definition 3.12 \cite{arnesen2016morphisms}). The remaining homotopy non-trivial single and double maps are called singleton because they are not homotopic to any scalar multiple of any other basis map (cf. Remark 4.7 in \cite{arnesen2016morphisms})
\end{remark}

The study of morphisms including higher-dimensional band complexes can be essentially reduced to the one-dimensional case with some additional effort as shown in Section 5 in \cite{arnesen2016morphisms}.

\subsection{Geometric model of a derived category of a gentle algebra}
There is an elegant geometric model of a derived category of gentle algebra developed by \cite{opper2018geometric}; structure of the derived category is captured by curves on a marked surface and their intersections. In this subsection, we present the geometric model in a formalism of \cite{amiot2019complete}, which differs slightly from the one of \cite{opper2018geometric} (and also for the essentially equivalent model~\cite{baur2021geometric} which was developed to describe the module category). Successively, we focus on several aspects of the geometric model, namely: marked surfaces and how they model gentle algebras, (graded) curves on the marked surface as models of objects in the derived category, morphisms between objects represented by intersections of curves, and representation of mapping cones in the geometric model.

\subsubsection{Marked surfaces and their dissections}
\begin{definition}[Marked surface; Definition 1.7 in \cite{amiot2019complete}]\label{DMarkedSurfaces}
Let $S$ be the interior of a compact oriented open smooth surface with boundary $\partial S$; let $M = M_{\gpoint} \cup M_{\rpoint}$ be finite set of marked points on $S \cup \partial S$ such that each connected component of $\partial S$ contains at least one marked point and that $\gpoint$-marked points (elements of $M_{\gpoint}$) and $\rpoint$-marked points (elements of $M_{\rpoint}$) alternate on each connected component of $\partial S$, and let $P = P_{\gpoint} \cup P_{\rpoint}$ be a finite set of marked points in $S$ called punctures (similarly as for marked points, elements of $P_{\gpoint}$ and $P_{\rpoint}$ are called $\gpoint$-punctures and $\rpoint$-punctures, respectively). If $\partial S$ is empty, it is required that both $P_{\gpoint}$ and $P_{\rpoint}$ be non-empty.
We refer to elements of $M_{\gpoint} \cup P_{\gpoint}$ and $M_{\rpoint} \cup P_{\rpoint}$ as to $\gpoint$-points and $\rpoint$-points, respectively.
\end{definition}

\begin{definition}[$\gpoint$-arcs and $\rpoint$-arcs; Definition 1.8 in \cite{amiot2019complete}]\label{DMarkedPointArcs}
A \emph{$\gpoint$-arc} is smooth map $\gamma$ from $(0,1)$ to $S \, \backslash \, P$ such that its endpoints, $\lim_{x \to 0} \gamma(x)$ and $\lim_{x \to 1} \gamma(x)$, lie in $M_{\gpoint} \cup P_{\gpoint}$ and that $\gamma$ is not contractible to an element in $M_{\gpoint} \cup P_{\gpoint}$. The notion of a \emph{$\rpoint$-arcs} is defined similarly, with $\rpoint$-marked points and $\rpoint$-punctures as their endpoints.
\end{definition}

\begin{definition}[Admissible dissection; Definition 1.9 in \cite{amiot2019complete}]\label{DAdmissibleDissection}
 A collection of pairwise non-intersecting and pairwise different \emph{$\gpoint$-arcs} $\{\gamma_1, \dots, \gamma_r\}$ on a marked surface $(S, M, P)$ is \emph{admissible} if its arcs do not enclose a subsurface containing no punctures of $P_{\rpoint}$ and with no boundary segment of $S$ on its boundary. A maximal admissible collection of $\gpoint$-arcs is called an \emph{admissible $\gpoint$-dissection}. These notions are used for $\rpoint$-arcs in an analogous manner.
\end{definition}

\begin{remark}
Throughout this text, the term admissible dissection stands for an admissible $\gpoint$-dissection.
\end{remark}

\begin{proposition}[Properties of admissible dissections; Propositions 1.11 to 1.13 in \cite{amiot2019complete}]\label{PAdmissibleDissectionProperties}
Let $(S, M, P)$ be a marked surface, and let $\Delta$ be an admissible dissection. Then the following statements hold:
\begin{enumerate}[(i)]
    \item The number of $\gpoint$-arcs in $\Delta$ equals $|M_{\gpoint}| + |P| + b + 2g - 2$ arcs, where $g$ is the genus of $S$ and $b$ is the number of connected components of $\partial S$.
    \item The connected components of the complement of $\Delta$ in $S \, \backslash \, P$ are homeomorphic to either an open disk with precisely one $\rpoint$-marked point on its boundary or to an open punctured open disk with no $\rpoint$-marked point on its boundary and precisely one $\rpoint$-puncture in the interior. 
    \item There exists, up to homotopy, a unique $\rpoint$-admissible dissection $\Delta^{\ast}$ such that each $\gpoint$-arc of $\Delta$ intersects exactly one $\rpoint$-arc of $\Delta^{\ast}$.
\end{enumerate}
\end{proposition}

\begin{definition}[Algebra associated to an admissible dissection; Definition 1.2 in \cite{amiot2019complete}]\label{DDissectionAssociatedAlgebra}
Let $(S, M, P)$ be a marked surface, and let $\Delta$ its admissible dissection. We set the algebra associated to the admissible dissection $\Delta$, $A(\Delta)$, to be the quotient of a path algebra of the quiver $Q(\Delta)$ over $k$ by the ideal $I(\Delta)$ defined as follows:
\begin{itemize}
    \item The vertices of $Q(\Delta)$ are in bijection with $\gpoint$-arcs in $\Delta$.
    \item There is an arrow $i \to j$ of $Q(\Delta)$ whenever $\gpoint$-arcs corresponding to $i$ and $j$ meet at the same $\gpoint$-marked point or $\gpoint$-puncture such that $j$ immediately follows $i$ the counter-clockwise order around the $\gpoint$-point (this means that there is no $\gpoint$-arc of $\Delta$ with the $\gpoint$-point as an endpoint between those corresponding to $i$ and $j$).
    \item The ideal $I(\Delta)$ is generated by the following relations: $i \to j$ and $j \to k$ compose to zero if the $\gpoint$-arcs corresponding to $i$ and $j$ meet at the endpoint of the $\gpoint$-arc corresponding to $j$ other than where the $\gpoint$-arcs corresponding $\gpoint$-arcs to $j$ and $k$ meet. 
\end{itemize}
\end{definition}

\begin{theorem}[Correspondence between marked surfaces and gentle algebras; Proposition 1.21 in \cite{opper2018geometric}]\label{TMarkedSurfaceGentle}
There is a bijection $[(S, M, P), \Delta] \to A(\Delta)$ between the set of homotopy classes of marked surfaces $(S, M, P)$ having no $\gpoint$-punctures ($P_{\gpoint} = \emptyset$) with an admissible dissection $\Delta$ to the set of isomorphism classes of gentle algebras over $k$.
\end{theorem}

\subsubsection{Models of objects in the derived category}
Henceforth, we assume that the marked surface $(S, M, P)$ has no $\gpoint$-punctures, in other words $P_{\gpoint} = \emptyset$.

\begin{definition}[$\gpoint$-infinite arc; after Definition 1.19 in \cite{opper2018geometric}]\label{DInfiniteArc}
 Let $(S, M, P)$ be a marked surface. An \emph{infinite $\gpoint$-arc} is a smooth map $\gamma$ from $(0,1)$ to $S \, \backslash \, P$ such that
\begin{itemize}
\item $\lim_{x \to 0} \gamma(x) \in M_{\gpoint}$ and $\lim_{x \to 1} \gamma(x) \in P_{\rpoint}$ (i.e.\ it goes from a $\gpoint$-marked point to a $\rpoint$-puncture) or
\item $\lim_{x \to 0} \gamma(x) \in P_{\rpoint}$ and $\lim_{x \to 1} \gamma(x) \in P_{\rpoint}$ (i.e.\ it goes from one $\rpoint$-puncture to another).
\end{itemize}
\end{definition}

\begin{remark}
The arcs are called infinite because of the they can be for various purposes more appropriately viewed as not really ending in $\rpoint$-punctures, but rather infinitely wrapping around these $\rpoint$-punctures in the counter-clockwise direction.
\end{remark}

\begin{definition}[Graded arcs and closed curves; Definition 2.4 in \cite{amiot2019complete} and Definition 2.10 in \cite{opper2018geometric}]\label{DGradedCurves}
Let $(S, M, P)$ be a marked surface, and let $\Delta$ its admissible dissection. Suppose that $\gamma$ is a $\gpoint$-arc, $\gpoint$-infinite arc, or closed curve in $S \, \backslash \, P$. We assume that $\gamma$ intersects the arcs of $\Delta^{\ast}$ minimally and transversally. A \emph{grading} $f$ on $\gamma$ is a function $f\colon \gamma \cap \Delta^{\ast} \to \Z$, where $\gamma \cap \Delta^{\ast}$ is the totally (if $\gamma$ is an arc) or cyclically (if $\gamma$ is a closed curve) ordered set of intersections of $\gamma$ with the $\rpoint$-arcs in $\Delta^{\ast}$. The grading $f$ must satisfy the following condition: assume that $q$ succeeds $p$ in $\gamma \cap \Delta^{\ast}$, then $\gamma$ enters a disk enclosed by $\rpoint$-arcs in $\Delta^{\ast}$ and a segment of $\partial S$ at $p$ and leaves it at $q$; there is a single $\gpoint$-point (cf. Proposition \ref{PAdmissibleDissectionProperties}(ii) above for admissible $\rpoint$-dissections) in the disk or on its boundary, and $f(q) = f(p) + 1$ if the $\gpoint$ lies to the left of $\gamma$, or $f(q) = f(p) - 1$ otherwise.
\end{definition}

\begin{remark}
As it is remarked in \cite{opper2018geometric} (Remark 2.11), grading of $\gpoint$-arc, $\gpoint$-infinite arc, or closed curve, if it exists, is fully determined by its value at a single intersection, and $\gpoint$-arc and $\gpoint$-infinite arc can always be equipped with a grading.
\end{remark}

\begin{theorem}[Correspondence between objects of $\Db(A(\Delta)\modl)$ and graded arcs and closed curves; Theorem 2.12 in \cite{opper2018geometric}]\label{TObjectsGradedCurves}
Let $(S, M, P)$ be a marked surface, and let $\Delta$ its admissible dissection. Then the following holds:
\begin{itemize}
    \item Graded $\gpoint$-arcs are in bijection with isomorphism classes of string complexes in $\Db(A(\Delta)\modl)$.
    \item The pairs of graded closed curves together with isomorphism classes of indecomposable finite-dimensional $k[X^{\pm1}]$-modules are in bijection with isomorphism classes of band complexes in $\Db(A(\Delta)\modl)$.
    \item Graded $\gpoint$-infinite arcs are in bijection with isomorphism classes of infinite string complexes in $\Db(A(\Delta)\modl)$.
\end{itemize}
\end{theorem}

\begin{remark}
The proof of \ref{TObjectsGradedCurves} in \cite{opper2018geometric} consists of assigning a homotopy string, homotopy band, or infinite homotopy string to $\gamma$, a $\gpoint$-arc, $\gpoint$-infinite arc, or closed curve, respectively. Vertices of unfolded diagram of the homotopy string, homotopy band, or infinite homotopy string correspond to successive intersections with the dual dissection $\Delta^{\ast}$ (grading may be assigned as in Definition~
\ref{DGradedCurves}). Arrows between the vertices are constructed as follows: if two successive vertices in the unfolded diagram correspond to $p$ and $q$ in $\gamma \cap \Delta^{\ast}$, then $\gamma$ enters a disk $\mcD$ enclosed by $\rpoint$-arcs in $\Delta^{\ast}$ and a segment of $\partial S$ at $p$ and leaves it at $q$. Inside the disk $\mcD$, $\gamma$ crosses $\gpoint$-arcs $\gamma_{n+1}, \dots, \gamma_1$ in that order, whose endpoint is the single $\gpoint$ on the segment of $\partial S$ of the boundary of the disk. Suppose that the single $\gpoint$ lies to the left of $\gamma$ in the disk $\mcD$. By Definition \ref{DDissectionAssociatedAlgebra} above, there are arrows $a_i: \gamma_{i+1} \to \gamma_i$ in $Q(\Delta)$ for $1 \leq i \leq n$ such that $a_n \dots a_i$ is a direct string for $Q(\Delta)$. Therefore, there is a direct string $a_n \dots a_1$ between the vertices corresponding to successive $p, q \in \gamma \cap \Delta^{\ast}$. If the single $\gpoint$ lies to the left of $\gamma$ in the disk $\mcD$, an inverse string that sits between the vertices corresponding to $p, q \in \gamma \cap \Delta^{\ast}$ is obtained similarly.
\end{remark}

\subsubsection{Models of morphisms in the derived category}
In Theorem 3.1 in \cite{opper2018geometric}, a correspondence between basis morphisms in $\Db(A(\Delta)\modl)$ between indecomposable objects, in the sense of \cite{arnesen2016morphisms} (graph maps, quasi-graph maps, singleton single maps, and singleton double maps; cf. Theorem \ref{TBasisMaps}), and \textit{oriented graded intersections} of the corresponding $\gpoint$-arcs, $\gpoint$-infinite arcs, and closed curves is established. For the purposes of this text, we do not need the full detail of the characterization in \cite{opper2018geometric}; it is enough to note that if $(\gamma_1, f_1)$ is an arc, infinite arc or a closed curve, and so is $(\gamma_2, f_2)$, the canonical maps from $P_{(\gamma_1, f_1)}$ to $P_{(\gamma_2, f_2)}$ correspond to intersections and common endpoints of $\gamma_1$ and $\gamma_2$ where their gradings agree locally.

An intersection $p \in \gamma_1 \cap \gamma_2$ or a common endpoint $p$ lies in a disk $\mcD$ enclosed by $\rpoint$-arcs in $\Delta^{\ast}$ and a segment of $\partial S$ or on the segment of $\partial S$ on the boundary of $\mcD$, respectively. Denote $J$ to be the set of $\gamma_1 \cap \Delta^{\ast}$ and $\gamma_2 \cap \Delta^{\ast}$ adjacent to $p$ on $\gamma_1$ and $\gamma_2$, respectively. We say that the gradings $f_1$ and $f_2$ agree locally if there are $q, r \in J$ such that $r$ immediately, among elements of $J$, follows $q$ in the counter-clockwise order around $p$, $q \in \gamma_1 \cap \Delta^{\ast}$, $r \in \gamma_2 \cap \Delta^{\ast}$, and $f_1(q) = f_2(r)$. This is summarised the following figure, which is an adaptation of Figure 9 in Remark 3.8 in \cite{opper2018geometric}:

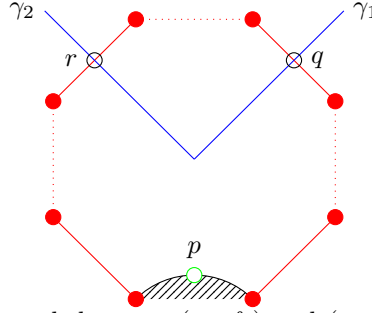
\begin{figure}[H]
    \centering
    
    \begin{tikzpicture}[x = 1cm, y = 1cm]
    
    \node[regular polygon, regular polygon sides = 8, minimum size = 4 cm] (Oct) at (0,0) {};
        
    \foreach \x in {1,...,8}
    {\node[coordinate, name = {OctNode\x}] at (Oct.corner \x) {};}
    
    \draw[red, dotted] (OctNode1) -- (OctNode2);
    \draw[red] (OctNode2) -- (OctNode3) node[coordinate, midway, name = Int2] {};
    \draw[red, dotted] (OctNode3) -- (OctNode4);
    \draw[red] (OctNode4) -- (OctNode5);
    \draw[red] (OctNode6) -- (OctNode7);
    \draw[red, dotted] (OctNode7) -- (OctNode8);
    \draw[red] (OctNode8) -- (OctNode1) node[coordinate, midway, name = Int1] {};
    
    \draw (OctNode5) to[out = 45, in = 135] node[coordinate, midway, name = Green] {} (OctNode6);
    
    \fill[pattern = north east lines] (OctNode5) to[out = 45, in = 135] (OctNode6) -- cycle;
    
    \draw[blue] (0,0) -- ($1.5*(Int1)$) node[at end, right, black] {$\gamma_1$};
    \draw[blue] (0,0) -- ($1.5*(Int2)$) node[at end, left, black] {$\gamma_2$};
    
    \foreach \x in {1,...,8}
    {\draw[red, fill = red] (OctNode\x) circle (0.1);}
    
    \draw[green, fill = white] (Green) circle (0.1);
    
    \draw (Int1) circle (0.1);
    \draw (Int2) circle (0.1);
    
    \node[above = 0.1cm] at (Green) {$p$};
    \node[right = 0.1cm] at (Int1) {$q$};
    \node[left = 0.1cm] at (Int2) {$r$};
    
    \end{tikzpicture}

    \caption{The graded curves $(\gamma_1, f_1)$ and $(\gamma_2, f_2)$ intersect at a point which lies inside a connected component $\mcD$ of the complement of $\Gamma^*$ in $S \backslash P$, a disk, or they meet at the unique $\gpoint$-point $p$ which lies on its boundary. Their intersections, $r, q$, with $\rpoint$-arcs in $\Gamma^*$ which form the boundary have equal grading, $f_1(q) = f_2(r),$ and $r$ follows $q$ in the counter-clockwise direction as viewed from $p$.}
    
    \label{FGradedIntersection}
    
\end{figure}

There are a two exceptions to this: if $P_{(\gamma, f)}$ is a one-dimensional band complex, the identity on $P_{(\gamma, f)}$ and the map $\xi: P_{(\gamma, f)} \to \tau P_{(\gamma, f)}[1] = P_{(\gamma, f)}[1]$ from the Auslander-Reiten triangle $\tau P_{(\gamma, f)} \to E \to P_{(\gamma, f)} \overset{\xi}{\to} \tau P_{(\gamma, f)}[1]$ are not represented by such an intersection.

Moreover, as an easy corollary formulated in Remark 3.8 in \cite{opper2018geometric}, we obtain that:
\begin{itemize}
	\item If $\gamma_1$, $\gamma_2$ intersect at $p \in S \, \backslash \, P$ and $f_1$ is a grading on $\gamma_1$, then there exists a grading $f_2$ on $\gamma_2$ such that $p$ corresponds to a map from $P_{(\gamma_1, f_1)}$ to $P_{(\gamma_2, f_2)}$ and from $P_{(\gamma_2, f_2)}$ to $P_{(\gamma_1, f_1[1])}$.
	\item If $\gamma_1$, $\gamma_2$ intersect at $p \in M_{\gpoint}$ and $f_1$ is a grading on $\gamma_1$, then $p$ corresponds to a map either from $P_{(\gamma_1, f_1)}$ to $P_{(\gamma_2, f_2)}$ or from $P_{(\gamma_2, f_2)}$ to $P_{(\gamma_1, f_1)}$ for some unique grading $f_2$ on $\gamma_2$. There is a morphism from $P_{(\gamma_1, f_1)}$ to $P_{(\gamma_2, f_2)}$ if $\gamma_2$ follows $\gamma_1$ in the counter-clockwise order around $p$, and vice versa.
\end{itemize}

Calculation mapping cones of the basis maps between the indecomposable objects can also be performed in the geometric model as illustrated in Theorem 4.1 in \cite{opper2018geometric}. This result on mapping cones is a geometric analogue of the characterization of mapping cones of the basis maps as given in \cite{ccanakcci2019mapping} and \cite{ccanakcci2021corrigendum}.

A basis map between two indecomposable objects is represented by an oriented graded intersection of corresponding graded curves in the geometric model; as an object the mapping cone of a basis map is given by resolving the intersection of the corresponding graded curves (the mapping cone may not be indecomposable). Specifically, if the basis map is represented by an oriented graded intersection at a common endpoint, its mapping cone is given by concatenation of the two graded curves (cf. Propositions 1.6, 1.20 and 3.7 in \cite{amiot2019complete}).

\section{Conditions on approximations of projectives}
The purpose of this section is to find necessary conditions for approximations of indecomposable projectives arising from a semiorthogonal decomposition of bounded derived category of $\Lambda$, a gentle algebra. Specifically, given $\langle \mcL, \mcR \rangle$, a semiorthogonal decomposition of $\Db(\Lambda \modl)$, there is a distinguished triangle $R_i \to P(i) \to L_i \to R_i[1]$ with $R_i \in \mcR$ and $L_i \in \mcL$ for every vertex $i$ of the underlying quiver, and we seek to understand what $R_i$ and $L_i$ may look like for all vertices $i$. In this section, we show that there are very restrictive conditions on such approximations, a fact that we exploit in the next section.

First, given $f\colon M \to N$, a graph map, singleton single map or singleton double map between string complexes, we study the situation when there is corresponding map $f'\colon N \to M[1]$ and what type this map is.
%We discuss this for graph maps and, then, for singleton single and singleton double maps.

Our considerations are motivated by Remark 3.8 in \cite{opper2018geometric}, which implies that such a map $f$ has a corresponding map $f'$ if and only if it is represented by an oriented graded intersection of corresponding $\gpoint$-arcs lies in interior of the associated marked surface. This motivation may also help explain the choice of terminology in the next definition.

\begin{definition}[Boundary graph maps]\label{DBoundaryGrapMaps}
Suppose that $f\colon M \to N$ is graph map (as in Definition \ref{DGraphMaps} above) between string complexes in $\Db(\Lambda\modl)$; we say that it is \emph{boundary graph map} if it arises from a maximal overlap of the unfolded diagrams as follows:
$$\begin{tikzcd}[arrows = {decorate = false, decoration={snake, segment length=2mm, amplitude=0.25mm}}, /tikz/column 1/.append style={anchor=base east}]
M: & \arrow[dash, decorate = true]{l} \bullet \arrow[dash]{r}{m_L} & \bullet \arrow[equal]{d} \arrow[dash]{r}{\ell_p} & \bullet \arrow[equal]{d} \arrow[dash]{r}{\ell_{p-1}} & \bullet \arrow[dash, dotted]{r} & \bullet \arrow[dash]{r}{\ell_2} & \bullet \arrow[equal]{d} \arrow[dash]{r}{\ell_1} & \bullet \arrow[equal]{d}\\
N: & \arrow[dash, decorate = true]{l} \bullet \arrow[dash]{r}{n_L} & \bullet \arrow[dash]{r}{\ell_p} & \bullet \arrow[dash]{r}{\ell_{p-1}} & \bullet \arrow[dash, dotted]{r} & \bullet \arrow[dash]{r}{\ell_2} & \bullet \arrow[dash]{r}{\ell_1} & \bullet
\end{tikzcd}$$
where $m_L$ and $n_L$ satisfy an endpoint condition (LG1) or $m_L, n_L$ satisfy (LG2) and $n_L m_L \neq 0$ if they are both non-zero.
\end{definition}

\begin{remark}
Note that situation that $n_L$ and $m_L$ are both non-zero and satisfy (LG2) with $n_L m_L = 0$ may not happen unless $p = 0$, meaning that the overlap between the unfolded diagram is trivial, formed by a single vertex.
\end{remark}

\begin{lemma}\label{LBoundaryGrapMaps}
Let $f\colon M \to N$ be a graph between string complexes in $\Db(\Lambda\modl)$ given by a maximal overlap of unfolded diagrams from $M$ to $N$ by Definition 3.9 in \cite{arnesen2016morphisms}. The same overlap of unfolded diagrams considered in the other direction, from $N$ to $M$, gives rise to a quasi-graph map $f'\colon N \to M[1]$ if and only if $f$ is not a boundary graph map.
\end{lemma}
\begin{proof}
At first, suppose that $p > 0$, in other words the maximal overlap between unfolded diagrams of $M$ and $N$ is given by a non trivial homotopy string. By Definition 3.9 in \cite{arnesen2016morphisms}, the overlap of unfolded diagrams of $M$ and $N$ is as follows:
$$\begin{tikzcd}[arrows = {decorate = false, decoration={snake, segment length=2mm, amplitude=0.25mm}}, /tikz/column 1/.append style={anchor=base east}]
M: & \arrow[dash, decorate = true]{l} \bullet \arrow[dash]{r}{m_L} & \bullet \arrow[equal]{d} \arrow[dash]{r}{\ell_p} & \bullet \arrow[equal]{d} \arrow[dash]{r}{\ell_{p-1}} & \bullet \arrow[dash, dotted]{r} & \bullet \arrow[dash]{r}{\ell_2} & \bullet \arrow[equal]{d} \arrow[dash]{r}{\ell_1} & \bullet \arrow[equal]{d} \arrow[dash]{r}{m_R} & \bullet \arrow[dash, decorate = true]{r} & {}\\
N: & \arrow[dash, decorate = true]{l} \bullet \arrow[dash]{r}{n_L} & \bullet \arrow[dash]{r}{\ell_p} & \bullet \arrow[dash]{r}{\ell_{p-1}} & \bullet \arrow[dash, dotted]{r} & \bullet \arrow[dash]{r}{\ell_2} & \bullet \arrow[dash]{r}{\ell_1} & \bullet \arrow[dash]{r}{n_R} & \bullet \arrow[dash, decorate = true]{r} & {}
\end{tikzcd}$$
such that and one of (LG1), (LG2) and one of (RG1), (RG2) endpoint conditions are met. Suppose $f$ is non-boundary, so $m_L \neq n_L$ and $m_R \neq n_R$. We show that if we consider this overlap of unfolded diagrams from $N$ to $M$ instead, none of the endpoint conditions hold.
    
Suppose that (LG1) holds for the overlap from $M$ to $N$; arrows $m_L$ and $n_L$ are either both direct or inverse, and there exists a non-stationary path $p_L$ such that $m_L = p_L n_L$ if they are direct and $m_L p_L = n_L$ if they are inverse. This means that $n_L$ is a sub-path of $m_L$ if arrows $m_L$ and $n_L$ are direct and $m_L$ is a sub-path of $n_L$ if arrows $m_L$ and $n_L$ are inverse.
    
If (LG2) holds for the overlap from $M$ to $N$, then arrows $m_L$ and $n_L$ are neither both direct nor inverse, and $m_L$ being non-zero implies that it is inverse and $n_L$ being non-zero implies that it is direct.
    
Therefore, if (LG1) holds for the overlap from $M$ to $N$, it cannot hold for the overlap from $N$ to $N$ as this would imply that $n_L = m_L$, and (LG2) cannot hold here either since the arrows $n_L$ are $m_L$ are equally oriented; on the other hand, should (LG2) hold for the overlap from $M$ to $N$, (LG1) cannot hold due to different orientation of $n_L$ and $m_L$ and neither can (LG2), which would imply that $n_L = m_L = 0$.
    
The right endpoint conditions are discussed dually. Since none of the endpoint conditions are met for the overlap from $N$ to $M$, it yields a quasi-graph map from $N$ to $M[1]$ by Definition \ref{DQuasiGraphMaps} above.

Should $f$ be boundary with $m_R = n_R = 0$, the overlap considered from $N$ to $M$ does not yield a quasi-graph map since it satisfies (RG2).
%it is easy to show that the leftmost quasi-graph map representative thereof is homotopically trivial, and this overlap if considered from $N$ to $M$ does not yield a quasi-graph map.

Suppose next that $p=0$, i.e.\ the overlap is trivial.
The case analysis is more complicated then since the overlap from $N$ to $M$ can be depicted in two different ways and we must check whether \emph{either of them} yields a quasi-graph map:
$$\begin{tikzcd}[arrows = {decorate = false, decoration={snake, segment length=2mm, amplitude=0.25mm}}, /tikz/column 1/.append style={anchor=base east}]
N: & \arrow[dash, decorate = true]{l} \bullet \arrow[dash]{r}{n_L} & \bullet \arrow[equal]{d} \arrow[dash]{r}{n_R} & \bullet \arrow[dash, decorate = true]{r} & {} &
N: & \arrow[dash, decorate = true]{l} \bullet \arrow[dash]{r}{n_L} & \bullet \arrow[equal]{d} \arrow[dash]{r}{n_R} & \bullet \arrow[dash, decorate = true]{r} & {} &
\\
M: & \arrow[dash, decorate = true]{l} \bullet \arrow[dash]{r}{m_L} & \bullet \arrow[dash]{r}{m_R} & \bullet \arrow[dash, decorate = true]{r} & {} &
M: & \arrow[dash, decorate = true]{l} \bullet \arrow[dash]{r}{m_R} & \bullet \arrow[dash]{r}{m_L} & \bullet \arrow[dash, decorate = true]{r} & {}
\end{tikzcd}$$

Should $m_R=0=n_R$ and (LG1) hold, it is easy to see that none of these yields a quasi-graph map, as the following diagrams show that when both $m_L$ and $n_L$ are direct, then (RG2) holds (the case with inverse homotopy letters is similar):
$$\begin{tikzcd}[arrows = {decorate = false, decoration={snake, segment length=2mm, amplitude=0.25mm}}, /tikz/column 1/.append style={anchor=base east}]
N: & \arrow[dash, decorate = true]{l} \bullet \arrow{r}{n_L} & \bullet \arrow[equal]{d} & {} &
N: & \arrow[dash, decorate = true]{l} \bullet \arrow{r}{n_L} & \bullet \arrow[equal]{d} & & {} &
\\
M: & \arrow[dash, decorate = true]{l} \bullet \arrow{r}{m_L} & \bullet & {} &
M: & & \bullet \arrow[<-]{r}{m_L} & \bullet \arrow[dash, decorate = true]{r} & {}
\end{tikzcd}$$

Similarly, the situation where $m_R=0=n_R$ and (LG2) hold with both $m_L$ and $n_L$ non-zero leads to unfolded diagrams:
$$\begin{tikzcd}[arrows = {decorate = false, decoration={snake, segment length=2mm, amplitude=0.25mm}}, /tikz/column 1/.append style={anchor=base east}]
N: & \arrow[dash, decorate = true]{l} \bullet \arrow{r}{n_L} & \bullet \arrow[equal]{d} & {} &
N: & \arrow[dash, decorate = true]{l} \bullet \arrow{r}{n_L} & \bullet \arrow[equal]{d} & & {} &
\\
M: & \arrow[dash, decorate = true]{l} \bullet \arrow[<-]{r}{m_L} & \bullet & {} &
M: & & \bullet \arrow{r}{m_L} & \bullet \arrow[dash, decorate = true]{r} & {}
\end{tikzcd}$$
The first of these never defines a quasi-graph map as (RG2) is satisfied, and the other defines a quasi-graph map precisely when $n_Lm_L=0$, that is when $f$ is not a boundary map. If one of $m_L$ and $n_L$ or both of them vanish, the arguments are similar.

Suppose finally that $p=0$ and one of $m_L$ and $n_L$ is non-zero, as is one of $m_R$ and $n_R$.
If (LG2) and (RG2) hold simultaneously for an overlap which gives the graph map, we assume without loss of generality that $n_L m_L \neq 0$ or one of them is zero and $n_R m_R \neq 0$ or one of them is zero (cf. Definition 5.1 in \cite{ccanakcci2019mapping}). This can be always arranged for by inverting one of the unfolded diagrams. 
This implies that $n_Lm_R=0=n_Rm_L$ and the diagram define a quasi-graph map $N\to M[1]$.
%We thus avoid the situation where (LG2) would hold with $n_L m_L = 0$ or dually for (RG2). In this setting, we show how obtain a pair of homotopic single or double maps of complexes from $N$ to $M[1]$ using the reasoning of the proof of Proposition 4.8 in \cite{arnesen2016morphisms}.

Suppose we are in the situation where (LG1) and (RG2) simultaneously hold. If both $m_R$ and $n_R$ are non-zero, we have the following diagram
$$\begin{tikzcd}[arrows = {decorate = false, decoration={snake, segment length=2mm, amplitude=0.25mm}}, /tikz/column 1/.append style={anchor=base east}]
M: & \arrow[dash, decorate = true]{l} \bullet \arrow[dash]{r} & \bullet \arrow[dotted]{d}[']{p_L} \arrow{r}{m_L} & \bullet \arrow[equal]{d} \arrow{r}{m_R} & \bullet \ar[dash]{r} & \bullet \arrow[dash, decorate = true]{r} & {} \\
N: & \arrow[dash, decorate = true]{l} \bullet \arrow[dash]{r} & \bullet \arrow{r}{n_L} & \bullet \ar[<-]{r}{n_R} & \bullet \arrow[dash]{r} & \bullet \arrow[dash, decorate = true]{r} & {}
\end{tikzcd}$$
Then we can build the following well-defined single map from $N$ to $M[1]$ representing a quasi-graph map: 
$$\begin{tikzcd}[arrows = {decorate = false, decoration={snake, segment length=2mm, amplitude=0.25mm}}, /tikz/column 1/.append style={anchor=base east}]
N: & \arrow[dash, decorate = true]{l} \bullet \arrow[dash]{r} & \bullet \arrow{d}{n_L} \arrow{r}{n_L} & \bullet \arrow[dash]{r} & \bullet \arrow[dash, decorate = true]{r} & {} \\
M[1]: & \arrow[dash, decorate = true]{l} \bullet \arrow{r}{m_L} & \bullet \arrow[dash]{r} & \bullet \arrow[dash]{r} & \bullet \arrow[dash, decorate = true]{r} & {}
\end{tikzcd}$$
In fact, the overlap viewed from $N$ to $M$ violates (LG1) as well as (RG2) as long as at least one of $m_R$ and $n_R$ is non-zero. Hence, the quasi-graph map can be defined in the same way also in this case.

Finally, analogous arguments easily take care also of the case where $p=0$ and (LG1) and (RG1) are simultaneously satisfied for the overlap from $M$ to $N$.
\end{proof}

\begin{lemma}\label{LDualMorphisms}
Suppose that $f\colon M \to N$, between $M, N \in \Db(\Lambda\modl)$ string complexes is a singleton single map (as in Definition \ref{DSingletonSingleMaps}) or a singleton double map (as in Definition \ref{DSingletonDoubleMaps}). If $f$ is a singleton single map is of type $(ii)$, then there is a singleton single map $f'\colon N \to M[1]$ of type $(iii)$, and vice versa. If $f$ is a singleton double map, then there is a singleton single map $f'\colon N \to M[1]$ of type $(iv)$, and vice versa.
\end{lemma}
\begin{proof}
Suppose we have a singleton single map $f\colon M \to N$ in configuration $(ii)$ as follows:
$$\begin{tikzcd}[arrows = {decorate = false, decoration={snake, segment length=2mm, amplitude=0.25mm}}, /tikz/column 1/.append style={anchor=base east}]
    M: & \arrow[dash, decorate = true]{l} \bullet \arrow[dash]{r}{m_{i+1}} & \bullet \arrow[]{d}{p} \arrow{r}{m_{i}}[swap]{=pp _R} & \bullet \arrow[dash, decorate = true]{r} & {} \\
    N: & & \bullet \arrow[dash]{r}{n_{i}} & \bullet \arrow[dash, decorate = true]{r} & {}
\end{tikzcd}$$
such that $n_{i+1}$ if inverse and viewed as an path does not start with $p$. We claim that there exists the following single map $f'$ of type $(iii)$:
$$\begin{tikzcd}[arrows = {decorate = false, decoration={snake, segment length=2mm, amplitude=0.25mm}}, /tikz/column 1/.append style={anchor=base east}]
    N: & & \bullet \arrow[]{d}{p_R} \arrow[dash]{r}{n_{i}} & \bullet \arrow[dash, decorate = true]{r} & {}\\
    M[1]: & \arrow[dash, decorate = true]{l} \bullet \arrow{r}{m_{i}}[swap]{=pp _R} & \bullet \arrow[dash]{r}{m_{i-1}} & \bullet \arrow[dash, decorate = true]{r} & {}
\end{tikzcd}$$
Suppose that $n_{i}$ is inverse, it corresponds to a path from its target to its source. Also, $p$ corresponds to such a path, and has the same starting vertex as $n_{i}$. We want to show that $p$ and $n_{i}$ do not start with the same arrow. If they did, due to $\Lambda$ being gentle, one would need to be a sub-path of the other. The situation of $p$ being a not necessarily proper sub-path of $n_{i}$ is excluded by definition. On the other hand, provided that $n_{i}$ is a proper sub-path of $p$, we show that the map $f$ is null-homotopic. Consider the following homotopy with $p'$ a non-stationary path for which $n_i p' = p$ as paths in $Q$:
$$\begin{tikzcd}[arrows = {decorate = false, decoration={snake, segment length=2mm, amplitude=0.25mm}}, /tikz/column 1/.append style={anchor=base east}]
    M: & \arrow[dash, decorate = true]{l} \bullet \arrow[dash]{r}{m_{i+1}} & \bullet \arrow[]{d}{p} \arrow[dotted]{dl}[']{p'} \arrow{r}{m_{i}}[swap]{=pp _R} & \bullet \arrow[dash, decorate = true]{r} & {} \\
    N: & \arrow[dash, decorate = true]{l} \bullet \arrow{r}{n_{i}} & \bullet &
\end{tikzcd}$$
If $m_{i+1}$ is direct, we have that $0 = m_{i+1} m_{i} = m_{i+1} p p_R = m_{i+1} p' n_{i} p_R$, and, given that $p' n_{i} p_R \neq 0$, $m_{i+1} p'$ needs to be zero. So we have a homotopy between $f$ and the zero map.
    
Because $n_{i}$ and $p$ do not start with the same arrow, $p_R n_{i}$ as path in $Q$ needs to lie in $I$, and $n_{i} p_R = 0$ as maps between projectives over $\Lambda$.
    
Now, assume that $m_{i-1}$ is direct; we need to show that $p_R m_{i-1} = 0$. We already have that $p p_R m_{i-1} = 0$. Since $p p_R \neq 0$, $p_R m_{i-1}$ equals zero. The map $f'$ is therefore well-defined. The other direction follows dually.\\\\
Let $f\colon M \to N$ be a singleton double map. Such map gives rise to the following unfolded diagrams:
$$\begin{tikzcd}[arrows = {decorate = false, decoration={snake, segment length=2mm, amplitude=0.25mm}}, /tikz/column 1/.append style={anchor=base east}]
M: & \arrow[dash, decorate = true]{l} \bullet \arrow[dash]{r}{m_{i+1}} & \bullet \arrow[']{d}{p_L} \arrow{r}{m_i}[swap]{=p_L p'} & \bullet \arrow[]{d}{p_R} \arrow[dash]{r}{m_{i-1}} & \bullet \arrow[dash, decorate = true]{r} & {} \\
N: & \arrow[dash, decorate = true]{l} \bullet \arrow[dash]{r}{n_{i+1}} & \bullet \arrow{r}{n_i}[swap]{=p' p_R} & \bullet \arrow[dash]{r}{n_{i-1}} & \bullet \arrow[dash, decorate = true]{r} & {}
\end{tikzcd}$$
The non-stationary path $p'$ gives rises to the following well-defined singleton single map in the configuration $(iv)$:
$$\begin{tikzcd}[arrows = {decorate = false, decoration={snake, segment length=2mm, amplitude=0.25mm}}, /tikz/column 1/.append style={anchor=base east}]
N: & \arrow[dash, decorate = true]{l} \bullet \arrow[dash]{r}{n_{i+1}} & \bullet \arrow[]{d}{p'} \arrow{r}{n_i}[swap]{= p' p_R} & \bullet \arrow[dash, decorate = true]{r} & {}\\
M[1]: & \arrow[dash, decorate = true]{l} \bullet \arrow{r}{m_i}[swap]{= p_L p'} & \bullet \arrow[dash]{r}{m_{i-1}} & \bullet \arrow[dash, decorate = true]{r} & {}
\end{tikzcd}$$
The other direction follows dually.
\end{proof}
Now, in order to establish some constraints on approximations of indecomposable projectives in semiorthogonal decompositions, we study the following problem and its dual variant: given a map $g: N \to P$ with $N$ an indecomposable object in $\Db(\Lambda \modl)$ and $P$ corresponding to shift of an indecomposable projective, what maps $f\colon M \to N$ compose to zero with $g$? It turns out that we may find some suitable maps with the property for all types of indecomposables in $\Db$, which imposes strong restrictions on approximations of indecomposable projectives.

\begin{proposition}\label{PStringsProjectives}
Suppose that $f\colon M \to N$, between $M, N \in \Db(\Lambda\modl)$ string complexes. 
\begin{enumerate}
\item If $f$ is non-zero and homotopically equivalent to a non-singleton single or double map or to a singleton single map in configurations $(iii)$ or $(iv)$ and there is a map $g: N \to P$ such that $P$ is projective and concentrated in single degree $i$, then $fg$ is null-homotopic.
\item Dually, if $f$ is non-zero and homotopically equivalent to a non-singleton single or double map or to a singleton single map in configurations $(ii)$ or $(iv)$ and there is a map $g: P \to M$ such that $P$ is projective and concentrated in single degree $i$, then $gf$ is null-homotopic.
\end{enumerate}
\end{proposition}
\begin{proof}
In order to prove this proposition, we consider $M, N$, and $P$ as complexes of projectives and the maps between them as maps of complexes. We discuss three possible cases:
\begin{enumerate}[(i)]
    \item \textit{The map $f\colon M \to N$ is a single map that is not singleton and not null-homotopic.} The discussion in the proof of Proposition 4.8 in \cite{arnesen2016morphisms} yields that there are two possible configurations for such a map in terms of unfolded diagrams of $M$ and $N$:
    $$\begin{tikzcd}[arrows = {decorate = false, decoration={snake, segment length=2mm, amplitude=0.25mm}}, /tikz/column 1/.append style={anchor=base east}]
    M: & \arrow[dash, decorate = true]{l} \bullet & \bullet \arrow[dash]{l}[']{m_{i+1}} \arrow[]{d}{p} \arrow{r}{m_{i} = p} & \bullet \arrow[dash, decorate = true]{r} & {}\\
    N: & \arrow[dash, decorate = true]{l} \bullet \arrow[dash]{r}{n_{i+1}} & \bullet \arrow[dash]{r}{n_{i}} & \bullet \arrow[dash, decorate = true]{r} & {}
    \end{tikzcd}$$
    or
    $$\begin{tikzcd}[arrows = {decorate = false, decoration={snake, segment length=2mm, amplitude=0.25mm}}, /tikz/column 1/.append style={anchor=base east}]
    M: & \arrow[dash, decorate = true]{l} \bullet & \bullet \arrow[dash]{l}[']{m_{i+1}} \arrow[]{d}{p} \arrow[dash]{r}{m_{i}} & \bullet \arrow[dash, decorate = true]{r} & {}\\
    N: & \arrow[dash, decorate = true]{l} \bullet \arrow{r}{n_{i+1} = p} & \bullet \arrow[dash]{r}{n_{i}} & \bullet \arrow[dash, decorate = true]{r} & {}
    \end{tikzcd}$$
    If the target of $p$ is of degree other than $i$, then $fg$ is zero. Assume, therefore, that the target of $P$ is of degree $i$. In terms of the unfolded diagram of $N$, $g$ can be thought of being given by maps of projectives in degree $i$ to the projective at degree $i$ in $P$. For discussing the composition with $g: N \to P$, it suffices to examine composition of $p$ with the map $q$ going from its target projective to the projective at degree $i$ in $P$.
    
    In unfolded diagrams, this gives rise to the following two situations:
    $$\begin{tikzcd}[arrows = {decorate = false, decoration={snake, segment length=2mm, amplitude=0.25mm}}, /tikz/column 1/.append style={anchor=base east}]
    M: & \arrow[dash, decorate = true]{l} \bullet & \bullet \arrow[dash]{l}[']{m_{i+1}} \arrow[]{d}{p} \arrow{r}{m_{i} = p} & \bullet \arrow[dash, decorate = true]{r} & {}\\
    N: & \arrow[dash, decorate = true]{l} \bullet \arrow[dash]{r}{n_{i+1}} & \bullet \arrow[]{d}{q} \arrow[dash]{r}{n_{i}} & \bullet \arrow[dash, decorate = true]{r} & {} \\
    P: & & \bullet &
    \end{tikzcd}$$
    or
    $$\begin{tikzcd}[arrows = {decorate = false, decoration={snake, segment length=2mm, amplitude=0.25mm}}, /tikz/column 1/.append style={anchor=base east}]
    M: & \arrow[dash, decorate = true]{l} \bullet & \bullet \arrow[dash]{l}[']{m_{i+1}} \arrow[]{d}{p} \arrow[dash]{r}{m_{i}} & \bullet \arrow[dash, decorate = true]{r} & {}\\
    N: & \arrow[dash, decorate = true]{l} \bullet \arrow{r}{n_{i+1} = p} & \bullet \arrow[]{d}{q} \arrow[dash]{r}{n_{i}} & \bullet \arrow[dash, decorate = true]{r} & {} \\
    P: & & \bullet &
    \end{tikzcd}$$
    In the second situation, $pq$ needs to zero so that $g$ is a well-defined map of complexes, and, thus, $fg = 0$. In the first situation, provided that $pq \neq 0$, we can construct a homotopy with a single non-zero component showing $fg$ is null-homotopic as follows:
    $$\begin{tikzcd}[arrows = {decorate = false, decoration={snake, segment length=2mm, amplitude=0.25mm}}, /tikz/column 1/.append style={anchor=base east}]
        M: & \arrow[dash, decorate = true]{l} \bullet & \bullet \arrow[dash]{l}[']{m_{i+1}} \arrow[]{d}{pq} \arrow{r}{m_{i} = p} & \bullet \arrow[dotted]{dl}{q} \, \arrow[dash, decorate = true]{r} & {} \\
        P: & & \bullet &
    \end{tikzcd}$$
    Due to the fact that $pq \neq 0$, any potentially inverse $m_{i+2}$ has to compose to zero with $q$, so this is a homotopy with a zero map.
    
    \item \textit{The map $f\colon M \to N$ is a double map that is not singleton and not null-homotopic.} By Lemma 4.13 in \cite{arnesen2016morphisms}, there is only one possible configuration of unfolded diagrams:
    $$\begin{tikzcd}[arrows = {decorate = false, decoration={snake, segment length=2mm, amplitude=0.25mm}}, /tikz/column 1/.append style={anchor=base east}]
        M: & \arrow[dash, decorate = true]{l} \bullet \arrow[dash]{r}{m_{i+1}} & \bullet \arrow[']{d}{p_L} \arrow[]{r}{m_i = p_L} & \bullet \arrow[]{d}{p_R} \arrow[dash]{r}{m_{i-1}} & \bullet \arrow[dash, decorate = true]{r} & {}\\
        N: & \arrow[dash, decorate = true]{l} \bullet \arrow[dash]{r}{n_{i+1}} & \bullet \arrow[]{r}{n_i = p_R} & \bullet \arrow[dash]{r}{n_{i-1}} & \bullet \arrow[dash, decorate = true]{r} & {}
    \end{tikzcd}$$
    We proceed similarly as in the previous case. Target of $p_L$ or $p_R$ needs to be in degree $i$; otherwise, the composition $fg$ is trivial. Suppose that $q$ is the component of $g$ going from target of $p_L$ or $p_R$ to the projective at degree $i$ in $P$.
    
    First, we deal with the situation that $p_L$ is in degree $i$. This yields the following unfolded diagram:
    $$\begin{tikzcd}[arrows = {decorate = false, decoration={snake, segment length=2mm, amplitude=0.25mm}}, /tikz/column 1/.append style={anchor=base east}]
        M: & \arrow[dash, decorate = true]{l} \bullet \arrow[dash]{r}{m_{i+1}} & \bullet \arrow[']{d}{p_L} \arrow[]{r}{m_i = p_L} & \bullet \arrow[]{d}{p_R} \arrow[dash]{r}{m_{i-1}} & \bullet \arrow[dash, decorate = true]{r} & {}\\
        N: & \arrow[dash, decorate = true]{l} \bullet \arrow[dash]{r}{n_{i+1}} & \bullet \arrow[]{d}{q} \arrow[]{r}{n_i = p_R} & \bullet \arrow[dash]{r}{n_{i-1}} & \bullet \arrow[dash, decorate = true]{r} & {}\\
        P: & & \bullet & &
    \end{tikzcd}$$
    As in the previous case, we can exhibit a homotopy between $fg$ and the zero map provided that $p_L q \neq 0$ as follows:
    $$\begin{tikzcd}[arrows = {decorate = false, decoration={snake, segment length=2mm, amplitude=0.25mm}}, /tikz/column 1/.append style={anchor=base east}]
        M: & \arrow[dash, decorate = true]{l} \bullet \arrow[dash]{r}{m_{i+1}} & \bullet \arrow[']{d}{p_L q} \arrow[]{r}{m_i = p_L} & \bullet \arrow[dotted]{dl}{q} \arrow[dash]{r}{m_{i-1}} & \bullet \arrow[dash, decorate = true]{r} & {}\\
        P: & & \bullet & &
    \end{tikzcd}$$
    The other case, in which $p_R$ is degree $i$ is simpler. We have the following situation:
    $$\begin{tikzcd}[arrows = {decorate = false, decoration={snake, segment length=2mm, amplitude=0.25mm}}, /tikz/column 1/.append style={anchor=base east}]
        M: & \arrow[dash, decorate = true]{l} \bullet \arrow[dash]{r}{m_{i+1}} & \bullet \arrow[']{d}{p_L} \arrow[]{r}{m_i = p_L} & \bullet \arrow[]{d}{p_R} \arrow[dash]{r}{m_{i-1}} & \bullet \arrow[dash, decorate = true]{r} & {}\\
        N: & \arrow[dash, decorate = true]{l} \bullet \arrow[dash]{r}{n_{i+1}} & \bullet \arrow[]{r}{n_i = p_R} & \bullet \arrow[]{d}{q} \arrow[dash]{r}{n_{i-1}} & \bullet \arrow[dash, decorate = true]{r} & {}\\
        P: & & & \bullet &
    \end{tikzcd}$$
    In order for $g$ to be a proper map of complexes, $p_R q$ needs to be zero, which is why $fg = 0$.
    
    \item \textit{The map $f\colon M \to N$ is a singleton single map in configuration $(iii)$ or $(iv)$.} The corresponding unfolded diagram looks as follows:
    $$\begin{tikzcd}[arrows = {decorate = false, decoration={snake, segment length=2mm, amplitude=0.25mm}}, /tikz/column 1/.append style={anchor=base east}]
        M: & \arrow[dash, decorate = true]{l} \bullet & \bullet \arrow[dash]{l}[']{m_{i+1}} \arrow[]{d}{p} \arrow[dash]{r}{m_{i}} & \bullet \arrow[dash, decorate = true]{r} & {}\\
        N: & \arrow[dash, decorate = true]{l} \bullet \arrow{r}{n_{i+1} = p_L p} & \bullet \arrow[dash]{r}{n_{i}} & \bullet \arrow[dash, decorate = true]{r} & {}
    \end{tikzcd}$$
    Again, the target of $p$ needs to be equal to the source of $q$; the composition $fg$ is trivial otherwise. Denote $q$ the component of $g$ going from target of $p$ to the projective at degree $i$ in $P$. We have the following situation:
    As $g$ is a map of complexes, we have that $p_L p q = 0$. By Definition \ref{DSingletonSingleMaps} above, $p$, $p_L$ are non-stationary paths in the quiver $Q$ such that no sub-path of $p p_L$ is in the ideal $I$. The map $q$ can be thought of as a linear combination of paths $\sum_j \alpha_j q_j$ from its target to its source.
    
    Recall that $kQ$ has basis of paths in $Q$ as vector space over $k$ and that $I$ is generated as an ideal of $kQ$ by paths of length two as $\Lambda$ is gentle. We observe that $I$ as vector space over $k$ is generated by paths that lie in $I$. Suppose that $\sum_j \alpha_j q'_j \in I$ for paths $q'_j$ in $Q$ and scalars $\alpha_j \in k$. As $I$ is generated by paths in $I$, we can express $\sum_j \alpha_j q'_j$ as a linear combination of paths in $I$. However, paths form a basis $kQ$, so every $q_i$ needs to lie in $I$ already.
    
    The fact that $p_L p q = 0$ means that $\sum_j \alpha_j q_j p p_L \in I$. By our discussion, every $q_j p p_L$ lies in $I$. Since no sub-path of $p p_L$ lies in $I$, a sub-path of $q_j p$ needs to lie in $I$. This yields that $p q_j = 0$ as maps between projectives, which yields that $p q = 0$. Hence, $fg = 0$.
\end{enumerate}

The other part of this proposition follows by dual argument.
\end{proof}

\begin{proposition}\label{PBandsProjectives}
Suppose $N \in \Kb(\Lambda\proj)$ is a band complex. Denote $f\colon N[-1] \to N$ the quasi-graph map corresponding to the identity on $N[-1]$. If there is a map $g$ from $N$ to $P$ or vice versa with $P$ being projective concentrated in a single degree $i$, then $fg$ and $gf[1]$, respectively, is homotopically trivial.
\end{proposition}
\begin{proof}
This follows directly from the fact that the map $f$ is a part of an almost split triangle and that $g$ is neither a split monomorphism nor a split epimorphism.

%Without loss of generality, we may assume that $N$ is one-dimensional band complex. If it is not, we consider only quasi-graph representatives between the bottom layers of $N[-1]$ and $N$ induced by identity on $N[-1]$, which are homotopically non-trivial as seen in the proof of Proposition 5.16 in \cite{arnesen2016morphisms}.

If $N$ is a one-dimensional band complex, we can also give a direct proof as follows.
Suppose that a component of the map $g: N \to P$ is given as follows:
$$\begin{tikzcd}[arrows = {decorate = false, decoration={snake, segment length=2mm, amplitude=0.25mm}}, /tikz/column 1/.append style={anchor=base east}]
N: & \arrow[dash, decorate = true]{l} \bullet \arrow[dash]{r}{n_{i+1}} & \bullet \arrow[]{d}{q} \arrow[dash]{r}{n_{i}} & \bullet \arrow[dash, decorate = true]{r} & {} \\
P: & & \bullet &
\end{tikzcd}$$
In this case, $q$ may be any map between the projectives which are its source and target, respectively. Up to inverting the upper string, $n_{i+1}$ is inverse.
%as $n_{i+1}$ and $n_{i}$ would need to start with the same arrow otherwise, which is not permitted.
Consider the following representative of $f\colon N[-1] \to N$, which is homotopically non-trivial:
$$\begin{tikzcd}[arrows = {decorate = false, decoration={snake, segment length=2mm, amplitude=0.25mm}}, /tikz/column 1/.append style={anchor=base east}]
N[-1]: & \arrow[dash, decorate = true]{l} \bullet & \bullet \arrow{l}[']{n_{i+1}} \arrow{d}{n_{i+1}} \arrow[dash]{r}{n_{i}} & \bullet \arrow[dash, decorate = true]{r} & {} \\
N: & \arrow[dash, decorate = true]{l} \bullet \arrow[dash]{r}{n_{i+2}} & \bullet & \bullet \arrow{l}[']{n_{i+1}} \arrow[dash, decorate = true]{r} & {}
\end{tikzcd}$$
Clearly, $fg = 0$. The other part of this proposition follows by dual argument.
\end{proof}

\begin{proposition}\label{PInfiniteStringsProjectives}
    Suppose $N \in \Db(\Lambda\modl)$ be an infinite string complex, possibly two-sided. If there is a map $g$ from $N$ to $P$ or vice versa with $P$ being concentrated in a single degree $i$, then there exist a map $f\colon N[-j] \to N$ or vice versa such that $fg$ and $gf[j]$, respectively, is homotopically trivial.
\end{proposition}
\begin{proof}
For the purposes of this proof, we consider $N$ and $P$ as complexes of projectives. We assume that the unfolded diagram of $N$ consists of a string $w = \begin{tikzcd} \bullet \arrow{r}{w_n} & \bullet \arrow[dotted, dash]{r} & \bullet \arrow[dash]{r}{w_1} & \bullet \end{tikzcd}$ which is preceded by infinitely many copies of a string $a = \begin{tikzcd} \bullet \arrow{r}{a_m} & \bullet \arrow[dotted, dash]{r} & \bullet \arrow{r}{a_1} & \bullet \end{tikzcd}$ formed from a repetition-free cycle $a_m \dots a_n$ with full relations in the underlying quiver, so it has the following form:
$$\begin{tikzcd}[arrows = {decorate = false, decoration={snake, segment length=2mm, amplitude=0.25mm}}, /tikz/column 1/.append style={anchor=base east}]
N: & \arrow[dash, decorate = true]{l} \bullet \arrow{r}{a_m} & \bullet \arrow[dotted, dash]{r} & \bullet \arrow{r}{a_1} & \bullet \arrow{r}{w_n} & \bullet \arrow[dotted, dash]{r} & \bullet \arrow[dash]{r}{w_1} & \bullet
\end{tikzcd}$$
Note that, by Definition \ref{DInfiniteHomotopyStrings}, the degree of projectives right to the source of $w_n$ is at least equal to $d$, the degree of the source of $w_n$, and that $a_m \dots a_1$ is a repetition-free cyclic path in $Q$ with full relations. We also observe that $w_n = a_m w'$ for a non-trivial path $w'$ as $a_1 w_n = 0$ and $\Lambda$ is gentle.

For any $b \in \mathbb{N}$, we consider the following map from $N[-bm]$ to $N$ denoted $f_b$:
$$\begin{tikzcd}[arrows = {decorate = false, decoration={snake, segment length=2mm, amplitude=0.25mm}}, /tikz/column 1/.append style={anchor=base east}]
N[-bm]: & \arrow[dash, decorate = true]{l} \bullet \arrow[equals]{d} \arrow{r}{a_m} & \bullet \arrow[equals]{d} \arrow[dotted, dash]{r} & \bullet \arrow[equals]{d} \arrow{r}{a_1} & \bullet \arrow[equals]{d} \arrow{r}{a_m} & \bullet \arrow{d}{w'} \arrow{r}{a_1} & \bullet & \arrow[dash, decorate = true]{l} {}\\
N: & \arrow[dash, decorate = true]{l} \bullet \arrow{r}{a_m} & \bullet \arrow[dotted, dash]{r} & \bullet \arrow{r}{a_1} & \bullet \arrow{r}{w_n} & \bullet \arrow[dash]{r}{w_{n-1}} & \bullet \arrow[dash, decorate = true]{r} & {}
\end{tikzcd}$$
This is a graph map as it satisfies endpoint conditions (LG$\infty$) and (RG1) in Definition \ref{DGraphMaps}. The former condition holds automatically; whereas, the latter condition translates to $w_n = a_m w'$, which is established above.

At first, consider a map $g: P \to N$. Because $P$ is in degree $i$, it is possible to find $b \in \mathbb{N}$ large enough that $f\colon N \to N[bm]$ as above is non-zero only in degrees strictly smaller that $i$; therefore, $gf$ equals zero.

Second, we deal with maps of type $g: N \to P$, where $P$ is concentrated in a single degree $i$. Suppose first that $i < d$. Because all letters to the left of the source of $w_1$ of degree $d$ are direct and because all projectives to the right of the source of $w_1$ have degree at least $d$, situation is then the following:
$$\begin{tikzcd}[arrows = {decorate = false, decoration={snake, segment length=2mm, amplitude=0.25mm}}, /tikz/column 1/.append style={anchor=base east}]
N: & \arrow[dash, decorate = true]{l} \bullet \arrow{r}{a_j} & \bullet \arrow[]{d}{q} \arrow{r}{a_{j-1}} & \bullet \arrow[dotted, dash]{r} & \bullet \arrow{r}{a_1} & \bullet \arrow{r}{w_1} & \bullet \arrow[dotted, dash]{r} & \bullet \arrow[dash]{r}{w_n} & \bullet\\
P: & & \bullet & & & & & & &
\end{tikzcd}$$
In order for $g$ to be properly defined, we need to have $q = a_{j-1}q'$ as $a_{j-1}$ is an arrow and $\Lambda$ is gentle. The map $q'$ is a homotopy between $g$ and the zero map.

Suppose that $i > d+1$, now; then the map $f_1$ defined above is non-zero only in degrees at most $d+1$, which means that $f_1 g = 0$.

Now, we discuss the situation of $i = d$. We consider the composition $f_1 g$, and the situation is as follows:
$$\begin{tikzcd}[arrows = {decorate = false, decoration={snake, segment length=2mm, amplitude=0.25mm}}, /tikz/column 1/.append style={anchor=base east}]
N[-m]: & \arrow[dash, decorate = true]{l} \bullet \arrow[equals]{d} \arrow{r}{a_m} & \bullet \arrow[equals]{d} \arrow[dotted, dash]{r} & \bullet \arrow[equals]{d} \arrow{r}{a_1} & \bullet \arrow[equals]{d} \arrow{r}{a_m} & \bullet \arrow{d}{w'} \arrow{r}{a_1} & \bullet \arrow[dash, decorate = true]{r} & {} \\
N: & \arrow[dash, decorate = true]{l} \bullet \arrow{r}{a_m} & \bullet \arrow[dotted, dash]{r} & \bullet \arrow{r}{a_1} & \bullet \arrow[]{d}{q} \arrow{r}{w_n} & \bullet \arrow[dash]{r}{w_{n-1}} & \bullet \arrow[dash, decorate = true]{r} & {} \\
P: & & & & \bullet & & &
\end{tikzcd}$$
Although there may be some other components to $g$ other than $q$, their sources lie to the right from the source of $q$, which renders them irrelevant since they compose to zero with $f_1$. Because $a_1 q$ composes to zero, by a similar argument as above, we conclude that $q = a_m q'$ and that $q'$ is the null homotopy for $f_1 g$.

Finally, suppose that $i = d+1$. The following diagram describes the situation:
$$\begin{tikzcd}[arrows = {decorate = false, decoration={snake, segment length=2mm, amplitude=0.25mm}}, /tikz/column 1/.append style={anchor=base east}]
N[-m]: & \arrow[dash, decorate = true]{l} \bullet \arrow[equals]{d} \arrow{r}{a_m} & \bullet \arrow[equals]{d} \arrow[dotted, dash]{r} & \bullet \arrow[equals]{d} \arrow{r}{a_1} & \bullet \arrow[equals]{d} \arrow{r}{a_m} & \bullet \arrow{d}{w'} \arrow{r}{a_1} & \bullet \arrow[dash, decorate = true]{r} & {} \\
N: & \arrow[dash, decorate = true]{l} \bullet \arrow{r}{a_m} & \bullet \arrow[dotted, dash]{r} & \bullet \arrow{r}{a_1} & \bullet \arrow{r}{w_n} & \bullet \arrow[]{d}{q} \arrow[dash]{r}{w_{n-1}} & \bullet \arrow[dash, decorate = true]{r} & {} \\
P: & & & & & \bullet & &
\end{tikzcd}$$
Similarly as for $i = d$, we need to discuss about components of $g$ other than $q$. We know that $w_n = a_m w'$, and we need to have that $w_n q = 0$ for the map $g$ to be properly defined. The composition $f_1 g$ is therefore trivial. 

The proof for two-sided infinite string complexes goes along the same lines. The arguments in the discussion above are applied for both infinite parts concurrently.
\end{proof}

The following theorem, which gives substantial restrictions on approximations of indecomposable projectives in semiorthogonal decompositions, is the main result of this section:

\begin{theorem}\label{TProjectivesDecomposition}
Let $\langle\mcL, \mcR\rangle$ be a semiorthogonal decomposition of $\Db(\Lambda\modl)$, and denote $P_1, \dots, P_n \in \Db(\Lambda\modl)$ the indecomposable direct summands of $\Lambda$. Moreover, let us have the following approximations:
$$\bigoplus_{t} R(i,t)^{m_R(i,t)} \overset{\lambda(i)}{\longrightarrow} P(i) \overset{\varrho(i)}{\longrightarrow} \bigoplus_{u} L(i,u)^{m_L(i,u)} \overset{\mu(i)}{\longrightarrow} \bigoplus_{t} R(i,t)^{m_R(i,t)}[1]$$
such that $R(i,t) \in \mcR$ for all $t$ and $L(i,u) \in \mcL$ for all $u$ and they are all indecomposable. We claim that:
\begin{enumerate}[(i)]
    \item All $R(i,t)$ and $L(i',u')$ are string complexes.
    \item All $\Hom_{\Db(\Lambda\modl)}(R(i,t)[j], R(i',t'))$, $\Hom_{\Db(\Lambda\modl)}(R(i,t)[j], L(i',u))$, and $\Hom_{\Db(\Lambda\modl)}(L(i,u)[j], L(i',u'))$ have a $k$-linear basis that comprises of boundary graph maps and singleton single maps of type $(i)$.
\end{enumerate}
\end{theorem}
\begin{proof}
At first, we show that no band or infinite string complexes may be direct summands of the approximations.

Suppose that $R(i,t)$ is a band complex for some $i$ and $t$. Denote $f\colon R(i,t)[-1] \to R(i,t)$ the quasi-graph map corresponding to the identity on $R(i,t)[-1]$ and $\iota\colon R(i,t) \to \bigoplus_{t} R(i,t)^{m_R(i,t)}$ one of the canonical injections. We have the following diagram:
$$\begin{tikzcd}
R(i,t)[-1] \arrow{d}{f \iota} \arrow{r} & 0 \arrow{d}{0} \arrow{r}{} & R(i,t) \arrow{r}{-\mathrm{id}_{R(i,t)[-1]}[1]} & R(i,t) \arrow{d}{(f\iota)[1]} \\
    \bigoplus_{t} R(i,t)^{m_R(i,t)} \arrow{r}{\lambda(i)} & P_i \arrow{r}{\varrho(i)} & \bigoplus_{u} L(i,u)^{m_L(i,u)} \arrow{r}{\mu(i)} & \bigoplus_{t} R(i,t)^{m_R(i,t)}[1]
\end{tikzcd}$$
The square on the left commutes by Proposition \ref{PBandsProjectives}.
There exists a homotopically non-trivial map $g: R(i,t) \to \bigoplus_{u} L(i,u)^{m_L(i,u)}$ by the axioms of triangulated categories.
% (TR3) in \ref{DTriangulatedCategory}.
This a contradiction as $R(i,t) \in \mcR$, and there are no non-trivial morphisms from $\mcR$ to $\mcL$.

Provided that $L(i,u)$ is a band complex, we denote $f\colon L(i,u)[-1] \to L(i,u)$ the quasi-graph map corresponding to the identity on $L(i,u)[-1]$ and $\pi\colon L(i,u) \to \bigoplus_{u} L(i,u)^{m_L(i,u)}$ one of the canonical projections. We consider the following diagram:
$$\begin{tikzcd}
P_i \arrow{d}{0} \arrow{r}{\varrho(i)} & \bigoplus_{u} L(i,u)^{m_L(i,u)} \arrow{d}{\pi f[1]} \arrow{r}{\mu(i)} & \bigoplus_{t} R(i,t)^{m_R(i,t)}[1] \arrow{r}{\lambda(i)[1]} & P[1] \arrow{d}{0} \\
0 \arrow{r} & L(i,u)[1] \arrow{r}{\mathrm{id}_{L(i,u)[1]}} & L(i,u)[1] \arrow{r} & 0
\end{tikzcd}$$
Again, the square on the left commutes by Proposition \ref{PBandsProjectives}, and there exists a homotopically non-trivial map $g: \bigoplus_{t} R(i,t)^{m_R(i,t)} \to L(i,u)[1]$, which yields a contradiction.

In a similar manner, we discuss that neither any $R(i,t)$ nor any $L(i,u)$ may be an infinite string complex using \ref{PInfiniteStringsProjectives} instead of \ref{PBandsProjectives}. \\\\
Second, we prove that basis maps between indecomposable direct summands of the approximations may only be certain special graph maps and singleton single maps.
By Theorem \ref{TBasisMaps}, the vector space $\Hom_{\Db(\Lambda\modl)}(R(i',t')[j], R(i, t))$ over $k$ has a basis of graph maps, quasi-graph maps, and singleton single and double maps. Suppose there is a map $f\colon R(i',t')[j] \to R(i, t)$ in this basis that is not boundary graph map or a singleton single map of type $(i)$. Without loss of generality, $f$ is either a quasi-graph map or singleton single map of type $(iii)$ or $(iv)$. If $f$ is a graph map that is not boundary, there exists a quasi-graph map $f'\colon R(i, t) \to R(i',t')[j+1]$ by Lemma \ref{LBoundaryGrapMaps}, and we discuss maps from $R(i, t)$ to $R(i',t')[j+1]$ instead. By Lemma \ref{LDualMorphisms}, similar argument works for $f$ a singleton single map or singleton double map.

Denote $\iota$ a canonical inclusion of $R(i,t)$ into $\bigoplus_{t} R(i,t)^{m_R(i,t)}$, and consider the following diagram:
$$\begin{tikzcd}
R(i',t')[j] \arrow{d}{f \iota} \arrow{r} & 0 \arrow{d}{0} \arrow{r}{} & R(i',t')[j+1] \arrow{r}{-\mathrm{id}_{R(i',t')[j]}[1]} & R(i',t')[j+1] \arrow{d}{(f\iota)[1]} \\
    \bigoplus_{t} R(i,t)^{m_R(i,t)} \arrow{r}{\lambda(i)} & P_i \arrow{r}{\varrho(i)} & \bigoplus_{u} L(i,u)^{m_L(i,u)} \arrow{r}{\mu(i)} & \bigoplus_{t} R(i,t)^{m_R(i,t)}[1]
\end{tikzcd}$$
The argument is similar as for the case bands above. The square on the left commutes by Proposition \ref{PBandsProjectives}. There exists a homotopically non-trivial map $g: R(i',t')[j+1] \to \bigoplus_{u} L(i,u)^{m_L(i,u)}$ yielding a contradiction.\\\\
For $\Hom_{\Db(\Lambda\modl)}(R(i',t)[j], L(i,u))$ and $\Hom_{\Db(\Lambda\modl)}(L(i',u')[j], L(i,u))$, the proof unfolds in a similar fashion as above.
\end{proof}

\begin{remark}
Equivalently, the theorem above can be proved using the fact that inclusions of $\mcL$ and $\mcR$ have left and right adjoints, respectively (cf. Proposition \ref{PSemiorthogonalDecomposition}). Suppose that we have an approximation $R \to P \to L \to R[1]$ with $R \in \mcR$ and $L \in \mcL$ for $P \in \Db\Lambda \modl)$. Suppose there is $0\ne f\colon R' \to R$ such that $R' \in \mcR$ and $f$ composes to zero with the approximation morphism $R \to P$. This is a contradiction as $\Hom_{\Db(\Lambda \modl)}(R',P) \cong \Hom_\mcR(R', R)$ via the approximation morphism $R \to P$. This argument can be made dually for $L$.
\end{remark}

\section{Semiorthogonal decompositions in the geometric model}
In this section, we give a one-to-one correspondence of between two-term semiorthogonal decompositions of $\Db(\Lambda \modl)$ and certain ways of cutting the underlying marked surface $(S, M, P)$ of the geometric model of $\Db(\Lambda \modl)$. We then extend this result to semiorthogonal decompositions with more than two terms.

Unless otherwise indicated, we only consider two-term semiorthogonal decompositions in this section.

\subsection{Bipartite admissible dissections}
We begin by translating the necessary conditions posed on a semiorthogonal decomposition $\langle \mcL, \mcR \rangle$ of $\Kb(\Lambda \proj)$ by Theorem \ref{TProjectivesDecomposition} to the geometric model of $\Db(\Lambda \modl)$. 

Taking models of indecomposable direct summands of approximations of indecomposable projectives in $\langle \mcL, \mcR \rangle$, we exhibit a system of $\gpoint-$arcs that may meet only at endpoints and that an a $\gpoint-$arcs from $\mcR$ need to always $\gpoint-$arcs in the counter-clockwise order around a common endpoint.

Furthermore, we prove a statement that can be viewed as a converse to Theorem \ref{TProjectivesDecomposition} by showing that any bipartite admissible dissection gives a semiorthogonal decomposition of $\Kb(\Lambda \proj)$.

\begin{definition}[Bipartite dissection; after Definition 1.9 in \cite{amiot2019complete}]\label{DBipartiteAdmissibleDissection}
 An admissible collection of $\gpoint$-arcs $\Gamma = \{\gamma_1, \dots, \gamma_r\}$ on the marked surface $(S, M, P)$ (see Definition \ref{DAdmissibleDissection} above) together with a function $p: \Gamma \to \{1, 2\}$ is called bipartite if $p(\gamma') \geq p(\gamma)$ for every two arcs $\gamma, \gamma' \in \Gamma$ such that $\gamma$ and $\gamma'$ have the same endpoint $e \in M_{\gpoint} \cup P_{\gpoint}$ and $\gamma'$ follows $\gamma$ in the counter-clockwise order around $e$. A maximal bipartite admissible collection is referred to as bipartite dissection.
\end{definition}

\begin{proposition}\label{PBipartiteDissectionCompletion}
    Let $(\Gamma, p)$ be a bipartite admissible collection of $\gpoint$-arcs on the surface $(S, M, P)$ with $P_{\gpoint}$ empty, then there exists a bipartite admissible dissection $(\Gamma', p')$ such that $\Gamma \subseteq \Gamma'$ and $p'$ extends $p$.
\end{proposition}
\begin{proof}
Similarly as in Proposition 1.12 in \cite{amiot2019complete}, consider a connected component $\mcP$ of the complement of the $\gpoint$-arcs in $\Gamma$ in $S \, \backslash \, P$. Suppose that there is a $\rpoint$-puncture in $\mcP$ as well some some other $\rpoint$-vertices, $\rpoint$-punctures or $\rpoint$-marked points. In this case it is possible to add a $\gpoint$-arc $\gamma'$ whose two endpoints coincide and which the $\rpoint$-puncture and no other red vertices. If no arc $\gamma \in \Gamma$ with $p(\gamma) = 2$ follows $\gamma'$ in the counter-clock wise order around the sole endpoint of $\gamma'$, $\gamma$ is assigned $1$; otherwise, it is assigned $2$.

Now, suppose that there are two or more $\rpoint$-marked points in $\mcP$, that is no $\rpoint$-punctures. In this case, the boundary of $\mcP$ is made up of segments of $\partial S$ with a $\rpoint$-marked point each and of segments of successive $\gpoint$-arcs in , which alternate with one another. Because there are at least two $\rpoint$-marked points in $\mcP$, there at least two different $\gpoint$-arc segments $\Sigma_1$ and $\Sigma_2$ of the boundary of $\mcP$ separated by a single $\partial S$ segment $T$ with one $\rpoint$-marked point $t$.

Assume that $\Sigma_2$ follows $\Sigma_1$ in the counter-clockwise order around $t$. Specifically, $t$ divides $T$ into two parts $T_i$ whose boundary consists of $t$ and a $\gpoint$ marked point $s_i$ that lies on the boundary of $\Sigma_i$, for $i = 1, 2$, and we assume that $T_2$ follows $T_1$ in the counter-clockwise order around $t$. Moreover, we assume that $\Sigma_i$ are made up of successive $\gpoint$-arcs $\sigma^{(i)}_1, \dots, \sigma^{(i)}_{n_i}$, for $i = 1, 2$, such that one of the endpoints of $\sigma^{(1)}_{n_1}$ is $s_1$ and one  of the endpoints of $\sigma^{(2)}_1$ is $s_2$. Note that one of the endpoints of $\sigma^{(i)}_1$ or $\sigma^{(i)}_{n_i}$ lies on a $\partial S$ segment of the boundary of $\mcP$; therefore, no $\gpoint$-arc $\gamma \in \Gamma$ may lie to the left of $\sigma^{(i)}_1$ or to the right of $\sigma^{(i)}_{n_i}$ in the counter-clockwise order around the endpoint lying on the $\partial S$ segment.

Denote $s'_1$ the $\gpoint$-marked point lying on the $\partial S$ segment of the boundary $\mcP$. Then, it is possible to add a $\gpoint$-arc $\gamma'$ from $s'_1$ to $s_2$ that lies inside $\mcP$ and is assigned $1$. All other $\gpoint$ arcs in $\Gamma$ from $s'_1$ or $s_2$ follow $\gamma'$ in the counter-clockwise orientation, and $\gamma'$ cuts off exactly one $\rpoint$-marked point from $\mcP$; hence we may extend $(\Gamma, p)$ by $\gamma'$ to a larger bipartite admissible collection.

By the inductive procedure described above, it is possible to extend $(\Gamma, p)$ to a bipartite admissible dissection $(\Gamma', p')$.
\end{proof}

\begin{proposition}\label{PBipartiteDissectionDecomposition}
Let $\Lambda$ be a gentle algebra, and let $(S, M, P)$ be its associated marked surface. A bipartite admissible dissection $(\Gamma, p)$ yields a semiorthogonal decomposition $\langle \mcL, \mcR \rangle$ of $\Kb(\Lambda \proj)$ where $\mcL$ is generated by $\gamma \in \Gamma$ such that $p(\gamma) = 1$ and $\mcR$ is generated by $\gamma \in \Gamma$ such that $p(\gamma) = 2$.
\end{proposition}
\begin{proof}
To prove that $\langle \mcL, \mcR \rangle$ is semiorthogonal decomposition of $\Kb(\Lambda \proj)$, it suffices to show that that there are no non-zero morphisms from $\mcR$ to $\mcL$ and that $\mcL$ and $\mcR$ generate the entire triangulated category $\Kb(\Lambda \proj)$ by Proposition \ref{PSemiorthogonalDecomposition}. 

The morphisms between from $P_{(\gamma_1, f_1)}$ to $P_{(\gamma_2, f_2)}$ are up to shift given by oriented graded intersections $\gamma_1$ from $\gamma_2$; as $\gpoint$-arcs in $\Gamma$ may intersect only at endpoints, there are non-trivial morphisms from $P_{(\gamma_1, f_1)}$ to $P_{(\gamma_2, f_2)}$ only if $\gamma_2$ follows $\gamma_1$ in the counter-clockwise order around a common endpoint. In such case, it is required that $p(\gamma_2) \geq p(\gamma_1)$. So there are no non-trivial maps from $P_{(\gamma_1, f_1)}$ to $P_{(\gamma_2, f_2)}$ should $p(\gamma_1)$ be greater than $p(\gamma_2)$.

Since there are no non-trivial morphisms from the generators of $\mcR$ to the shifts generators of $\mcL$, a simple induction shows that there no non-trivial morphisms from $\mcR$ to $\mcL$.

The triangulated category $\Kb(\Lambda \proj)$ is generated by all string complexes, so it is enough to prove that they are generated by $\mcL$ and $\mcR$. Consider a $\gpoint$-arc $\sigma$ given that any string complex is of form $P_{(\sigma, h)}$ for some $\gpoint$-arc $\sigma$, and suppose that the arcs in $\Gamma$ and $\sigma$ are in a minimal position.

By Proposition 1.12 in \cite{amiot2019complete} the connected components of the complement of $\Gamma$ in $S \, \backslash \, P$ are homeomorphic either to an open disk with a single $\rpoint$-marked point on the boundary or to an open punctured disk with a $\rpoint$-puncture and no $\rpoint$-marked points on the boundary. The $\gpoint$-arc $\sigma$ crosses components $\mcP_1, \dots, \mcP_n$ of the complement of $\Gamma$ in $S \, \backslash \, P$ in this order. Inductively, we will construct a sequence: $$\gamma_1^{(1)}, \dots, \gamma_{m_1}^{(1)}, \gamma_1^{(2)}, \dots, \gamma_{m_2}^{(2)}, \dots, \gamma_1^{(n)}, \dots, \gamma_{m_n}^{(n)}$$ such that:
\begin{enumerate}[(i)]
    \item one of the endpoints of $\sigma$ is an endpoint of $\gamma_1^{(1)}$,
    \item the other endpoint of $\sigma$ is an endpoint of $\gamma_{m_n}^{(n)}$, and
    \item successive curves in the sequence share an endpoint, and
    \item the iterated concatenation of the arcs in the sequence is homotopic to $\sigma$. 
\end{enumerate}
The subsequence $\gamma_1^{(i)}, \dots, \gamma_{m_i}^{(i)}$ is constructed in a straight-forward way. If the $\rpoint$-point pertaining to the component $\mcP_i$ lies on the boundary, the segment of $\sigma$ lying in $\mcP_i$ cuts in into two components: one containing the $\rpoint$-point pertaining to $\mcP_i$ and the other without a marked point. Denote $\delta_1^{(i)}, \dots, \delta_{m'_i}^{(i)}$ the $\gpoint$-arcs in $\Gamma$ that lie or partially lie on the boundary of the component of $\mcP_i$ cut by $\sigma$ without the $\rpoint$-point pertaining to $\mcP_i$ such that $\delta_j^{(i)}$ and $\delta_{j+1}^{(i)}$ share a common endpoint for $1 \leq j \leq m'-1$ and $\sigma$ meets $\delta_1^{(i)}$ before $\delta_{m'_i}^{(i)}$. Clearly, the segment of $\sigma$ in $\mcP_i$ is homotopic to the iterated concatenation of $\delta_1^{(i)}, \dots, \delta_{m'_i}^{(i)}$.

If the $\rpoint$-point pertaining to the component $\mcP_i$ is a puncture, $\sigma$ may wind around it and intersect itself. $\delta_1^{(i)}, \dots, \delta_{m'_i}^{(i)}$ the $\gpoint$-arcs in $\Gamma$ on the boundary of $\mcP_i$ such that $\sigma$ crosses their dual $\rpoint$-arcs in $\Gamma^*$ within $\mcP_i$ in that order together with the $\gpoint$-arcs in $\Gamma$ that $\sigma$ crosses as it goes through $\mcP_i$, which would form the first and the last element of the sequence, if applicable and not included already.

The sequence is defined so that the segment of $\sigma$ that lies in $\mcP_i$ is homotopic iterated concatenation of $\delta_1^{(i)}, \dots, \delta_{m'_i}^{(i)}$. We use the fact that the segment of $\sigma$ is homotopic to the iterated concatenation of its sub-segments that lie in connected components of the complement of $\Gamma*$ in $S \backslash P$. If the segment $\sigma$ in $\mcP_i$ crosses in the connected component $\mcQ$ of the complement of $\Gamma*$ in $S \backslash P$, it has to cross two adjacent $\rpoint$-arcs on its boundary; otherwise, it would cross a $\gpoint$-arc, which it does not. Therefore, the corresponding $\gpoint$-arcs in $\delta_j^{(i)}, \delta_{j+1}^{(i)} \in \Gamma$ have a common endpoint, and their concatenation is homotopic to the respective sub-segment of $\sigma$ in $\mcQ$. This can be schematically illustrated by the following figure:
\begin{figure}[H]
    \centering
    
    \begin{tikzpicture}[x = 1cm, y = 1cm]
    \pgfmathsetmacro{\length}{15/sqrt(50+10*sqrt(5))}
    
    \node[coordinate] (Red) at (0,-0.5*\length) {};
    \node[coordinate] (Green) at (0,0.5*\length) {};
    
    \draw[dark-green] (Green) -- ++(216:\length) node[coordinate, at end] (GreenLeft) {} node[near start, below, black] {$\delta^{(i)}_j$};
    \draw[dark-green] (Green) -- ++(-36:\length) node[coordinate, at end] (GreenRight) {} node[near start, below, black] {$\delta^{(i)}_{j+1}$};
    
    \draw[red] (Red) -- ++(54:\length) node[coordinate, at end] (RedAuxRight) {};
    \draw[red] (Red) -- ++(126:\length) node[coordinate, at end] (RedAuxLeft) {};
    
    \draw[red, dotted] (RedAuxRight) -- ++(54:0.5*\length) node[near start, below right, black] {$\delta^{(i)*}_{j+1}$};
    \draw[red, dotted] (RedAuxLeft) -- ++(126:0.5*\length) node[near start, below left, black] {$\delta^{(i)*}_j$};
    
    \path (GreenLeft) -- (Red) node[coordinate, midway] (SigmaLeft) {};
    \path (GreenRight) -- (Red) node[coordinate, midway] (SigmaRight) {};
    
    \draw[blue] (SigmaLeft) to[out = 45, in = 135] (SigmaRight);
    \node[blue, below = 0*\length of SigmaLeft] {$\sigma$};
    
    \draw[dark-green, fill = white] (Green) circle (0.1);
    \draw[dark-green, fill = white] (GreenLeft) circle (0.1);
    \draw[dark-green, fill = white] (GreenRight) circle (0.1);
    
    \draw[red, fill = red] (Red) circle (0.1);
    
    \end{tikzpicture}

    \caption{The sub-segment of $\sigma$ in $Q$ and its representation by $\gpoint$-arcs in $\Gamma$}
    
    \label{FBipartiteDissectionDecomposition1}
    
\end{figure}
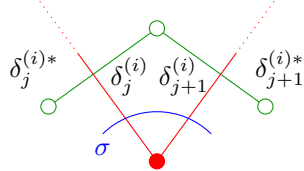
We note that the same principle has been tacitly applied in the case above, when $\mcP_i$ had a $\rpoint$-marked point on the boundary. This principle also underpins the proof of how the indecomposable objects in $\Db(\Lambda \modl)$ are reflected in its geometric model in Theorem 2.12 in \cite{opper2018geometric}.

If $i = 1$, $\delta_1^{(1)}$ shares an endpoint with $\sigma$, we define $\gamma_i^{(1)} = \delta_i^{(1)}$ for $1 \leq i \leq m$ and we set $m = m'-1$ if the $\rpoint$-marked point lies to same side of $\sigma$ in $\mcP_2$ as in $\mcP_2$ and $m = m'$ otherwise, which includes the situation when $i = n$ and $\sigma$ shares an endpoint with $\delta_{m'_1}^{(1)}$. If $i > 1$, $\sigma$ crosses $\delta_1^{(i)}$ as it enters $\mcP_i$, we define $\gamma_j^{(i)} = \delta_{j+1}^{(i)}$ for $1 \leq j \leq m$ and we set $m = m'-1$ if the $\rpoint$-marked point lies to same side of $\sigma$ in $\mcP_{i+1}$ as in $\mcP_i$ and $m = m'$ otherwise, which includes the situation when $i = n > 1$ and $\sigma$ shares an endpoint with $\delta_{m'_i}^{(i)}$.

As above, we note that the last $\gpoint$-arc $\delta^{(i)}_{m'_i}$ is included in the final sequence precisely when $\sigma$ crosses the dual $\rpoint$-arc in $\Gamma^*$ in addition to crossing the $\gpoint$-arc separating $\mcP_{i}$ and $\mcP_{i+1}$.

By definition, our sequence satisfies (i) and (ii). As for (iii), it suffices to check it only for $\gamma_{m_i}^{(i)}$ and $\gamma_{1}^{(i+1)}$. Suppose that $\sigma$ lies to the same side of the $\rpoint$-marked point in both $\mcP_i$ and $\mcP_{i+1}$; schematically, we are in the following situation up to change in orientation of $\sigma$, which is assumed to go from left to right:
\begin{figure}[H]
    \centering
    
    \begin{tikzpicture}[x = 1cm, y = 1cm]
        \def\sides{5}
        \def\pensize{3}
        
        \pgfmathsetmacro{\penfactor}{0.5*cos(180/\sides)}
        \pgfmathsetmacro{\penonecenter}{-\pensize*\penfactor}
        \pgfmathsetmacro{\pentwocenter}{\pensize*\penfactor}
        
        \node[regular polygon, regular polygon sides = 5, minimum size = \pensize cm, shape border rotate = 18] (Pen1) at (\penonecenter,0) {};
        
        \foreach \x in {1,...,\sides}
        {\node[coordinate, name = {Pen1Node\x}] at (Pen1.corner \x) {};}

        \draw[dark-green, dotted] (Pen1Node1) -- (Pen1Node2) node[coordinate, midway, name = {Sigma1}] {};
        \draw (Pen1Node2) to[out = -9, in = 71] node[coordinate, midway, name = {Pol1Red}] {} (Pen1Node3);
        \fill[pattern = north east lines] (Pen1Node2) to[out = -9, in = 71] (Pen1Node3) -- cycle;        \draw[red, fill = red] (Pol1Red) circle (0.1);
        \draw[dark-green, dotted] (Pen1Node3) -- (Pen1Node4);
        \draw[dark-green] (Pen1Node4) -- (Pen1Node5) node[coordinate, midway, name = {Sigma2}] {};
        \draw[dark-green] (Pen1Node5) -- (Pen1Node1) node[midway, below left, black] {$\gamma_{m_i}^{(i)}$};
        
        \node[regular polygon, regular polygon sides = 5, minimum size = \pensize cm, shape border rotate = -18] (Pen2) at (\pentwocenter,0) {};
        
        \foreach \x in {1,...,\sides}
        {\node[coordinate, name = {Pen2Node\x}] at (Pen2.corner \x) {};}
        
        \draw[dark-green] (Pen2Node1) -- (Pen2Node2) node[midway, below right, black] {$\gamma_1^{(i+1)}$};
        \draw[dark-green, dotted] (Pen2Node3) -- (Pen2Node4);
        \draw (Pen2Node4) to[out = 99, in = 189] node[coordinate, midway, name = {Pol2Red}] {} (Pen2Node5);
        \fill[pattern = north west lines] (Pen2Node4) to[out = 99, in = 189] (Pen2Node5) -- cycle;           \draw[red, fill = red] (Pol2Red) circle (0.1);
        \draw[dark-green, dotted] (Pen2Node5) -- (Pen2Node1) node[coordinate, midway, name = {Sigma3}] {};
        
        \foreach \x in {1,...,\sides}
        {\draw[dark-green, fill = white] ({Pen1Node\x}) circle (0.1);}
        
        \foreach \x in {1,4,5}
        {\draw[dark-green, fill = white] ({Pen2Node\x}) circle (0.1);}
        
        \draw[blue] ($1.5*(Sigma1)$) to[out = 315, in = 180] (Sigma2) to[out = 0, in = 225] ($1.5*(Sigma3)$);
        
        \node[blue, below left] (Sigma2) {$\sigma$};
        
    \end{tikzpicture}

    \caption{Constructing $\Gamma$-sequence for $\sigma$ going from $\mcP_i$ to $\mcP_{i+1}$ with respective $\rpoint$-points being to different sides of $\sigma$}
    
    \label{FBipartiteDissectionDecomposition2}
    
\end{figure}
The $\gpoint$-arcs $\gamma_{m_i}^{(i)}$ and $\gamma_1^{(i+1)}$ have a shared endpoint. We know that $\delta_{m'_i}^{(i)} = \delta_1^{(i+1)}$ and that $\delta_{m'_i - 1}^{(i)} = \gamma_{m_i}^{(i)}$ and $\delta_{m'_i}^{(i)}$, and $\delta_1^{(i+1)}$ and $\delta_2^{(i+1)} = \gamma_1^{(i+1)}$ share an endpoint, respectively. However, the shared endpoint has to lie on the opposite side of the intersection of $\delta_{m'_i}^{(i)} = \delta_1^{(i+1)}$ and $\sigma$ as the $\rpoint$-point correspond to the respective component in both cases, so it has be the same endpoint.

Suppose that $\sigma$ lies to the same side of the $\rpoint$-marked point in both $\mcP_i$ and $\mcP_{i+1}$; schematically, we are in the following situation up to change in orientation of $\sigma$, which is assumed to go from left to right:
\begin{figure}[H]
    \centering
    
    \begin{tikzpicture}[x = 1cm,y = 1cm]
        \def\sides{5}
        \def\pensize{3}
        
        \pgfmathsetmacro{\penfactor}{0.5*cos(180/\sides)}
        \pgfmathsetmacro{\penonecenter}{-\pensize*\penfactor}
        \pgfmathsetmacro{\pentwocenter}{\pensize*\penfactor}
        
        \node[regular polygon, regular polygon sides = 5, minimum size = \pensize cm, shape border rotate = 18] (Pen1) at (\penonecenter,0) {};
        
        \foreach \x in {1,...,\sides}
        {\node[coordinate, name = {Pen1Node\x}] at (Pen1.corner \x) {};}

        \draw[dark-green, dotted] (Pen1Node1) -- (Pen1Node2) node[coordinate, midway, name = {Sigma1}] {};
        \draw (Pen1Node2) to[out = -9, in = 71] node[coordinate, midway, name = {Pol1Red}] {} (Pen1Node3);
        \fill[pattern = north east lines] (Pen1Node2) to[out = -9, in = 71] (Pen1Node3) -- cycle;
        \draw[red, fill = red] (Pol1Red) circle (0.1);
        \draw[dark-green, dotted] (Pen1Node3) -- (Pen1Node4);
        \draw[dark-green] (Pen1Node4) -- (Pen1Node5) node[coordinate, midway, name = {Sigma2}] {} node[midway, above right, black] {$\gamma_1^{(i+1)}$};
        \draw[dark-green] (Pen1Node5) -- (Pen1Node1) node[midway, below left, black] {$\gamma_{m_i}^{(i)}$};
        
        \node[regular polygon, regular polygon sides = 5, minimum size = \pensize cm, shape border rotate = -18] (Pen2) at (\pentwocenter,0) {};
        
        \foreach \x in {1,...,\sides}
        {\node[coordinate, name = {Pen2Node\x}] at (Pen2.corner \x) {};}
        
        \draw[dark-green, dotted] (Pen2Node3) -- (Pen2Node4) node[coordinate, midway, name = {Sigma3}] {};
        \draw[dark-green, dotted] (Pen2Node4) -- (Pen2Node5);
        \draw (Pen2Node5) to[out = 171, in = 261] node[coordinate, midway, name = {Pol2Red}] {} (Pen2Node1);
        \fill[pattern = north east lines] (Pen2Node5) to[out = 171, in = 261] (Pen2Node1) -- cycle;
        \draw[red, fill = red] (Pol2Red) circle (0.1);
        \draw[dark-green, dotted] (Pen2Node1) -- (Pen2Node2);
        
        \foreach \x in {1,...,\sides}
        {\draw[dark-green, fill = white] ({Pen1Node\x}) circle (0.1);}
        
        \foreach \x in {1,4,5}
        {\draw[dark-green, fill = white] ({Pen2Node\x}) circle (0.1);}
        
        \draw[blue] ($1.25*(Sigma1)$) -- (Sigma2) to[out = -18, in = 108] ($1.5*(Sigma3)$);
        
        \node[blue, below left] (Sigma2) {$\sigma$};
        
    \end{tikzpicture}

    \caption{Constructing $\Gamma$-sequence for $\sigma$ going from $\mcP_i$ to $\mcP_{i+1}$ with respective $\rpoint$-points being to different sides of $\sigma$}
    
    \label{FBipartiteDissectionDecomposition3}
    
\end{figure}
We can employ a similar argument as above to reason that the $\gpoint$-arcs $\gamma_{m_i}^{(i)} = \delta_{m'_i}^{(i)}$ and $\gamma_1^{(i+1)} = \delta_2^{(i+1)}$ have the same endpoint.

Now, we observe that $\sigma$ is homotopic to the iterated concatenation of $\gamma_1^{(1)}, \dots, \gamma_{m_1}^{(1)},$ $\gamma_1^{(2)}, \dots, \gamma_{m_2}^{(2)}, \dots,$ $\gamma_1^{(n)}, \dots, \gamma_{m_n}^{(n)}$. If we concatenate $\gamma_{m_i}^{(i)}$ and $\gamma_{1}^{(i+1)}$ and denote the resulting arc $\gamma_{i,i+1}$, then the iterated concatenation of $\gamma_{i-1,i},$ $\gamma_2^{(i)}, \dots, \gamma_{m_i - 1}^{(i)},$ and $\gamma_{i,i+1}$ is homotopic to the segment of $\sigma$ in $\mcP_i$. The fact $\sigma$ is homotopic to the iterated concatenation of its segments in $\mcP_1, \dots, \mcP_n$ completes the argument.

By Theorem 4.1 in \cite{opper2018geometric} as recalled above, concatenation of $\gpoint$-arcs $\varrho_1$ and $\varrho_2$ at their common endpoint around which $\varrho_2$ follows $\varrho_1$ in the counter-clockwise order corresponds to forming the mapping cone of the arising morphisms from $P_{(\varrho_1, g_1)}$ to $P_{(\varrho_2, g_2)}$ up to shift (cf. the proof of Proposition 3.7 in \cite{amiot2019complete}). This observation can be used to deduce that $P_{(\sigma, h)}$ lies in the triangulated subcategory of $\Kb(\Lambda \proj)$ generated by $P_{\gamma_j^{(i)}}$.

Let $(\Gamma, p)$ be a bipartite admissible dissection of $(S, M, P)$. We recalled in the proof above that, by Proposition 1.12 in \cite{amiot2019complete}, the connected components of the complement of $\Gamma$ in $S \, \backslash \, P$ are homeomorphic either to an open disk with a single $\rpoint$-marked point on the boundary or to an open punctured disk with a $\rpoint$-puncture and no $\rpoint$-marked points on the boundary.
We consider a component $\mcP$ whose boundary contains arcs $\gamma_1, \gamma_2 \in \Gamma$ such that $p(\gamma_1) \neq p(\gamma_2)$. The component $\mcP$ needs to be of the first type. If its boundary did not contain a segment of $\partial S$, it would consist entirely of arcs in $\Gamma$, and there would invariably be curves $\gamma_3, \gamma_4 \in \Gamma$ in its boundary such that they share an endpoint, $\gamma_4$ follows $\gamma_3$ in the counterclockwise order around the endpoint, but $p(\gamma_4) < p(\gamma_3)$. In the counter-clockwise order, the boundary of $\mcP$ looks as follows: a segment of $\partial S$ adjacent, a sequence of arcs in $\Gamma$ assigned $1$ by $p$, and a sequence of arcs in $\Gamma$ assigned $2$ by $p$.
\end{proof}

It turns out that the procedure of representing a $\gpoint$-arc as an iterated concatenation of $\gpoint$-arcs in $(\Gamma, p)$ employed in the proof of the proposition above can be used to get approximation triangles for string objects in the induced semiorthogonal decomposition $\langle \mcL, \mcR \rangle$ of $\Kb(\Lambda \proj)$. Before that, however, we need to deal with some technicalities.

\begin{lemma}\label{LMappingConeReduction}
Let $A \nameto{f_i} B_i$, $i = 1, 2$, be morphisms in a triangulated category $\mcT$, then $B_1 \oplus \cone((f_1, f_2)^T)$ is isomorphic to $B_1 \oplus \cone(-f_2 \circ a_1[-1])$ such that $a_1$ comes from the distinguished triangle: $$A \nameto{f_1} B_1 \nameto{b_1} \cone(f_1) \nameto{a_1} A[1].$$ Furthermore, $\cone((f_1, f_2)^T)$ is isomorphic to $\cone(-f_2 \circ a_1[-1])$ provided that $\mcT$ is Krull-Schmidt.

Dually, let $A_i \to B$, $i = 1, 2$ be morphisms in a triangulated category $\mcT$, then $A_1[1] \oplus \cone(f_1, f_2))$ is isomorphic to $A_1[1] \oplus \cone(b_1 f_2)$ such that $b_1$ comes from the distinguished triangle: $$A_1 \nameto{f_1} B \nameto{b_1} \cone(f_1) \nameto{a_1} A_1[1].$$ Also, $\cone((f_1, f_2))$ is isomorphic to $\cone(b_1 f_2)$ provided that $\mcT$ is Krull-Schmidt.
\end{lemma}
\begin{proof}
We apply the octahedral axiom to the composition of morphisms: $$\cone(f_1)[-1] \nameto{-a_1[-1]} A \nameto{(f_1, f_2)^T} B_1 \oplus B_2$$
We note that $-a_1[-1]$ composes to zero with $f_1$.

We have the following distinguished triangles:
$$\cone(f_1)[-1] \nameto{-a_1[-1]} A \nameto{f_1} B_1 \nameto{b_1} \cone(f_1),$$
$$A \nameto{(f_1, f_2)^T} B_1 \oplus B_2 \nameto{b} \cone((f_1, f_2)^T) \nameto{a} A[1],$$
and two direct sum of two distinguished triangles:
$$\cone(f_1)[-1] \nameto{-f_2 \circ a_1[-1]} B_2 \nameto{g} \cone(-f_2 \circ a_1[-1]) \nameto{h} \cone(f_1)$$
and
$$0 \nameto{} B_1 \nameto{\id_{B_1}} B_1 \nameto{} 0,$$
which corresponds to the distinguished triangle for $-(f_1, f_2)^T \circ a_1[-1]$ because the first coordinate is zero.

Using the octahedral axiom, it is possible to construct:
$$B_1 \nameto{} B_1 \oplus \cone(-f_2 \circ a_1[-1]) \nameto{} \cone((f_1, f_2)^T) \nameto{} B_1[1]$$
The rightmost map in the distinguished triangle equals $f_1[-1] \circ a$, which is zero because $(f_1[-1], f_2[-1])^T \circ a$ is zero. Therefore, the distinguished triangle splits, and $B_1 \oplus \cone((f_1, f_2)^T) \cong B_1 \oplus \cone(-f_2 \circ a_1[-1])$.

Dually, we apply the octahedral axiom to the following composition:
$$A_1 \oplus A_2 \nameto{(f_1, f_2)} B \nameto{b_1} \cone(f_1)$$.
Combining distinguished triangles:
$$A_1 \oplus A_2 \nameto{(f_1, f_2)} B \nameto{b} \cone((f_1, f_2)) \nameto{a} (A_1 \oplus A_2)[1],$$
$$B \nameto{b_1} \cone(f_1) \nameto{a_1} A_1[1] \nameto{-f_1[1]} B[1],$$
and the direct sum of:
$$A_2 \nameto{b_1 f_2} \cone(f_1) \nameto{g} \cone(b_1 f_2) \nameto{h} A_2[1]$$
and
$$A_1 \nameto{\id_{A_1}} A_1 \nameto{} 0 \nameto{} A_1[1],$$
we obtain the following:
$$\cone((f_1, f_2)) \nameto{} A_1 \oplus \cone(b_1 f_2) \nameto{} A_1[1] \nameto{} \cone((f_1, f_2))[1]$$
This triangle splits, however, because the rightmost morphism is zero as the first coordinate of the composition $b \circ (f_1, f_2)$, which equals zero. We thus obtain that $A_1[1] \oplus \cone((f_1, f_2)) \cong A_1[1] \oplus \cone(b_1 f_2)$.

If $\mcT$ is Krull-Schmidt and $C \oplus D \cong C \oplus D'$, it is easy to see that the uniqueness of the direct decomposition into indecomposables yields that $D \cong D'$. The remaining claims of this lemma follow from this observation.
\end{proof}

\begin{proposition}\label{PBipartiteDissectionTriangles}
Let $(\Gamma, p)$ be a bipartite admissible dissection of the marked surface $(S, M, P)$, and let $\langle \mcL, \mcR \rangle$ be the associated semiorthogonal decomposition of $\Kb(\Lambda \proj)$. For every string object $P_{(\sigma, h)}$, there exist string objects $P_{(\delta_1, g_1)}, \dots, P_{(\delta_n, g_n)}$ such that:
\begin{enumerate}[(i)]
    \item $\sigma$ is homotopic to an iterated concatenation of $\delta_1, \dots, \delta_n$;
    \item either $P_{(\delta_i, g_i)} \in \mcL$ for all odd $i$ and $P_{(\delta_j, g_j)} \in \mcR$ for all even $j$ or $P_{(\delta_i, g_i)} \in \mcR$ for all odd $i$ and $P_{(\delta_j, g_j)} \in \mcL$ for all even $j$;
    \item $\delta_i$ and $\delta_{i+1}$ for $1 \leq i \leq n-1$ share an endpoint, and every shared endpoint gives rise to a map $f_{i, i+1}: P_{(\delta_i, g_i)} \to P_{(\delta_{i+1}, g_{i+1})}$ should $P_{(\delta_i, g_i)} \in \mcL$ or vice versa;
    \item the mapping cone $\bigoplus_{i \, \mathrm{even}} P_{(\delta_i, g_i)} \nameto{f} \bigoplus_{i \, \mathrm{odd}} P_{(\delta_j, g_j)}$ where $f$ consists of maps $f_{i, j}$ arising from shared endpoints of adjacent $\delta_i$ and $\delta_j$ or vice versa is isomorphic to $P_{(\sigma, h)}$.
\end{enumerate}
\end{proposition}
\begin{proof}
Assume that $P_{\sigma, h}$ is a string object reprsented by a $\gpoint$-arc $\sigma$ with grading $h$. As in the proof of Proposition \ref{PBipartiteDissectionDecomposition}, we can construct a sequence: $$\gamma_1^{(1)}, \dots, \gamma_{m_1}^{(1)}, \gamma_1^{(2)}, \dots, \gamma_{m_2}^{(2)}, \dots, \gamma_1^{(n)}, \dots, \gamma_{m_n}^{(n)}$$ of $\gpoint$-arcs in $\Gamma$ such that $\sigma$ is homotopic to the iterated concatenation thereof.

From this sequence, we concatenate its maximal segments such that all $\gpoint$-arcs in the segment belong to the same part of $(\Gamma, p)$. We obtain a sequence of $\delta_1, \dots, \delta_n$, whose iterated concatenation remains homotopic to $\sigma$. Without loss of generality, we may assume that $\delta_i$ for $i$ odd are concatenations of $\gamma \in \Gamma$ such that $p(\gamma) = 1$ and conversely for $\delta_j$ with $j$ even. Consequently, for any choice of gradings, string objects associated to $\delta_i$, $i$ odd, belong $\mcL$; whereas, string objects associated to $\delta_j$, $j$ even, belong $\mcR$.

By construction, $\delta_i$ and $\delta_{i+1}$ share an endpoint for all $1 \leq i \leq n-1$. We may choose gradings $g_1, \dots, g_n$ such that the common endpoints induce maps from $f_{i, i+1}: P_{(\delta_i, g_i)} \to P_{(\delta_{i+1}, g_{i+1})}$, for odd $1 \leq i \leq n-1$, and $f_{i, i-1}: P_{(\delta_i, g_i)} \to P_{(\delta_{i-1}, g_{i-1})}$, for odd $3 \leq i \leq n$, and that the concatenation of $(\delta_1, g_1), \dots, (\delta_n, g_n)$ has the same grading as $(\sigma, h)$ (this simply entails choosing $g_1$ such that there is map $P_{(\sigma, h)} \to P_{(\delta_1, g_1[1])}$ at their shared endpoint).

Having established that $P_{(\delta_i, g_i)}, \dots, P_{(\delta_i, g_i)}$ satisfy conditions (i) to (iii), it remains to prove that $P_{\sigma, h}$ is the mapping cone of the following map: $$\bigoplus_{i \, \mathrm{odd}} P_{(\delta_i, g_i)} \nameto{f} \bigoplus_{i \, \mathrm{even}} P_{(\delta_i, g_i)}$$
where the only components of $f$ from $P_{(\delta_i, g_i)}$ are $f_{i, i-1}$ and $f_{i, i+1}$, if defined.

We note that the case of $n = 1$ is trivial and that the case of $n = 2$ is a direct application of Theorem 4.1 in \cite{opper2018geometric}, which yields that the mapping cone of a map arising from two $\gpoint$-arcs meeting at a $\gpoint$-endpoint corresponds to their concatenation at this common endpoint. Subsequently, we proceed in calculation of the mapping cone by induction. For the induction, we abstract from the condition (ii).

Noting that $\Kb(\Lambda \proj)$ is a Krull-Schmidt category, we apply Lemma \ref{LMappingConeReduction} to the following datum:
$$P_{(\delta_1, g_1)} \oplus \bigoplus_{i > 1, \, \mathrm{odd}} P_{(\delta_i, g_i)} \nameto{(f_{1,2}, f')} \bigoplus_{i \, \mathrm{even}} P_{(\delta_i, g_i)}$$
It yields that the mapping cone of $f$ isomorphic to the mapping cone of the composition:
$$\bigoplus_{i > 1, \, \mathrm{odd}} P_{(\delta_i, g_i)} \nameto{f'} P_{(\delta_2, g_2)} \oplus \bigoplus_{i > 2 \, \mathrm{even}} P_{(\delta_i, g_i)} \nameto{p \oplus \id} \cone(f_{1, 2}) \oplus \bigoplus_{i > 2 \, \mathrm{even}} P_{(\delta_i, g_i)}$$
where $\cone(f_{1,2})$ is isomorphic to $P_{(\delta_{1,2}, h_{1,2})}$, $\delta_{1,2}$ is the concatenation of $\delta_1$ and $\delta_2$ at their shared endpoint that gives rise to $f_{1,2}$ and the grading $h_{1,2}$ is such that there is a map $P_{(\delta_{1,2}, h_{1,2})} \to P_{(\delta_1, h_1[1])}$ at the other endpoint of $\delta_1$, and $p: P_{(\delta_2, g_2)} \to \cone(f_{1,2})$ is the map from the distinguished triangle pertaining to $f_{1,2}$. The map $p$ is represented by a common endpoint of $\delta_2$ and $\delta_{1,2}$, which is the other endpoint of $\delta_2$ shared with $\delta_3$. This means that $p \circ f_{3,2}$ is non-zero and it is represented by a common endpoint.

We reduced the number of $\gpoint$-arcs by one, and we have thus performed the induction step, and our proof is complete.
\end{proof}

\subsection{Good cuts of a marked surface and semiorthogonal decompositions}
In this subsection, we associate a specific way how to cut $(S, M, P)$, a \textit{good cut}, to every bipartite admissible dissections thereof, and we prove that two bipartite admissible dissections of $(S, M, P)$ yield the same semiorthogonal decomposition of $\Kb(\Lambda \proj)$ if and only if they cut the underlying marked surface along the same lines up to homotopy.

Also, we show that for any good cut of $(S, M, P)$ it is always possible to construct an associated bipartite admissible of $(S, M, P)$ such that they cut the underlying marked surface along the same lines up to homotopy. Hence, we obtain a one-to-one correspondence between semiorthogonal decompositions of $\Kb(\Lambda \proj)$ and good cuts of $(S, M, P)$.

\begin{definition}\label{DDividingArcs}
Given $(\Gamma, p)$, a bipartite admissible dissection of $(S, M, P)$, we say that a connected component, $\mcP$, of the complement of $\Gamma$ in $S \, \backslash \, P$ is mixed if its boundary contains arcs in different parts of $(\Gamma, p)$. There is an arc in $\mcP$ from the unique $\gpoint$-point on the boundary at which arcs in different parts of $(\Gamma, p)$ meet to the $\rpoint$-point on the $\partial S$-segment of the boundary of $\mcP$. We call it a \textit{dividing arc} of $\mcP$.
\end{definition}

\begin{lemma}\label{LCrossingDividingArcs}
Let $(\Gamma, p)$ be a bipartite admissible dissection of $(S, M, P)$, and let $\sigma$ be $\gpoint$-arc in $S$. The arc $\sigma$ does not cross any dividing arc of a mixed component of the complement of $\Gamma$ in $S \, \backslash \, P$ if and only if $P_{(\sigma, h)}$ lies in $\mcL_{(\Gamma, p)}$ or $\mcR_{(\Gamma, p)}$.
\end{lemma}
\begin{proof}
We consider a sequence, $\gamma_1^{(1)}, \dots, \gamma_{m_1}^{(1)}, \gamma_1^{(2)}, \dots, \gamma_{m_2}^{(2)}, \dots, \gamma_1^{(n)}, \dots, \gamma_{m_n}^{(n)}$, of $\gpoint$-arcs in $\Gamma$ such that $\sigma$ is homotopic to the iterated concatenation thereof as constructed in the proof of Proposition \ref{PBipartiteDissectionDecomposition}. The function $p$ attains the same value for all arcs in the sequence because $\sigma$ does not cross any dividing arc of a mixed component of the complement of $\Gamma$ in $S \, \backslash \, P$.

The object $P_{(\sigma, h)}$ lies in the triangulated subcategory generated by $P_{\gamma_j^{(i)}}$. Because all $\gamma_j^{(i)}$ belong to the same part of $(\Gamma, p)$, $P_{(\sigma, h)}$ lies inside either $\mcL_{(\Gamma, p)}$ or $\mcR_{(\Gamma, p)}$.

Now, we suppose that $\sigma$ crosses a dividing arc of a mixed component $\mcP$ of the complement of $\Gamma$ in $S \, \backslash \, P$. This means that $\sigma$ has to cross $\gamma_1 \in \Gamma$ such that $p(\gamma_1) = 1$ or share an endpoint with $\gamma_1$ such that $\gamma_1$ follows $\sigma$ in the counter-clockwise order around it and that $\sigma$ has to cross $\gamma_2 \in \Gamma$ such that $p(\gamma_2) = 2$ or share an endpoint with $\gamma_2$ such that $\sigma$ follows $\gamma_2$ in the counter-clockwise order around it. Therefore, there are non-trivial morphisms $P_{(\sigma, h)} \to P_{(\gamma_1, c_1)}$ and $P_{(\gamma_2, c_2)} \to P_{(\sigma, h)}$ for suitable gradings. Because $P_{(\gamma_1, c_1)} \in \mcL_{(\Gamma, p)}$ and $P_{(\gamma_2, c_2)} \in \mcR_{(\Gamma, p)}$, we obtain that $P_{(\sigma, h)}$ has to be trivial.
\end{proof}

\begin{proposition}\label{PBipartiteDissectionsEquivalence}
Two bipartite admissible dissection of $(S, M, P)$ give rise to the same semiorthogonal decomposition of $\Kb(\Lambda \proj)$ if and only if dividing arcs of their mixed components are homotopic.
\end{proposition}
\begin{proof}
Let us have two bipartite admissible dissections $(\Gamma_1, p_1)$ and $(\Gamma_2, p_2)$. Without loss of generality, we suppose that there exists a dividing arc $\Psi$ of $(\Gamma_1, p_1)$ that is not, up to homotopy, a dividing arc of $(\Gamma_2, p_2)$. If $\Psi$ crosses an arc in $\Gamma_2$, it follows from Lemma \ref{LCrossingDividingArcs} above that $(\Gamma_1, p_1)$ and $(\Gamma_2, p_2)$ do not give the same semiorthogonal decompositions of $\Kb(\Lambda \proj)$.

We can assume that $\Psi$ lies in a connected component $\mcP$ of the complement of $\Gamma_1$ in $S \, \backslash \, P$. There are two $\gpoint$-arcs $\sigma_1, \sigma_2 \in \Gamma_2$ on the boundary of $P$ sharing a common endpoint with $\Psi$; the $\gpoint$-endpoint of $\Psi$ may not be adjacent to the $\rpoint$-endpoint on the same boundary component of $S$ since $\Psi$ would not be a dividing arc otherwise. Also, $\sigma_1$ and $\sigma_2$ need to belong to the same part of $(\Gamma_2, p_2)$ because $\Psi$ is not a dividing arc of $(\Gamma_2, p_2)$. Therefore, the object of $\Kb(\Lambda \proj)$ corresponding, up to grading, to the concatenation of $\sigma_1$ and $\sigma_2$ lies in $\mcL_{(\Gamma_2, p_2)}$ or $\mcR_{(\Gamma_2, p_2)}$, but it belongs to neither $\mcL_{(\Gamma_1, p_1)}$ nor $\mcR_{(\Gamma_1, p_1)}$ by Lemma \ref{LCrossingDividingArcs} as it crosses a dividing arc $\Psi$ of $(\Gamma_2, p_2)$.\\

Suppose that $(\Gamma_1, p_1)$ and $(\Gamma_2, p_2)$, two bipartite admissible dissections of $(S, M, P)$, give rise to different semiorthogonal decompositions of $\Kb(\Lambda \proj)$. Provided that there exists a $\gpoint$-arc $\sigma \in \Sigma_1$ such that the corresponding object of $\Kb(\Lambda \proj)$, up to grading, does not lie in either $\mcL_{(\Gamma_2, p_2)}$ or $\mcR_{(\Gamma_2, p_2)}$, it follows from Lemma \ref{LCrossingDividingArcs} that $\sigma$ crosses a dividing arc of $(\Gamma_2, p_2)$. Hence, $(\Gamma_1, p_1)$ and $(\Gamma_2, p_2)$ do not have, up to homotopy, the same dividing arcs.

Suppose now that all objects of $\Kb(\Lambda \proj)$ corresponding to $\gpoint$-arcs in $\Gamma_1$ given some grading lie either in $\mcL_{(\Gamma_2, p_2)}$ or in $\mcR_{(\Gamma_2, p_2)}$. Consider a function $p_1'$ that assigns an arc $\sigma \in \Gamma_1$ the value $1$ if the corresponding object $\Kb(\Lambda \proj)$ lies in $\mcL_{(\Gamma_2, p_2)}$ and the value $2$ if the corresponding object $\Kb(\Lambda \proj)$ lies in $\mcR_{(\Gamma_2, p_2)}$. In this way, we obtain a bipartite admissible dissection $(\Gamma_1, p_1')$. We have that $\mcL_{(\Gamma_1, p_1')} \subseteq \mcL_{(\Gamma_2, p_2)}$ and $\mcR_{(\Gamma_1, p_1')} \subseteq \mcR_{(\Gamma_2, p_2)}$; thus $(\Gamma_1, p_1')$ and $(\Gamma_2, p_2)$ give rise to the same semiorthogonal decompositions of $\Kb(\Lambda \proj)$.

From the proof of the other implication above, we get that $(\Gamma_1, p_1')$ and $(\Gamma_2, p_2)$ have the same dividing arcs. However, $(\Gamma_1, p_1)$ and $(\Gamma_1, p_1')$ cannot have the same dividing arcs because it would imply that $p_1 = p_1'$.
\end{proof}

\begin{definition}\label{DCuttingMarkedSurface}
Let $(S, M, P)$ be marked surface. A \textit{dividing arc} is a smooth map $\omega: (0,1) \to S \, \backslash \, P$ such that $g_\omega = \lim_{x \to 0} \omega(x) \in M_{\gpoint}$ and $r_\omega = \lim_{x \to 1} \omega(x) \in M_{\rpoint}$ such that $g_\omega$ and $r_\omega$ are not adjacent on a component of $\partial S$. We can obtain a new marked surface, the \textit{cut surface}, $(S_\omega, M_\omega, P_\omega)$ by cutting $(S, M, P)$ along $\omega$.

The marked points of $(S_\omega, M_\omega, P_\omega)$ are $P_\omega = P$, $M_\omega {}_{\gpoint} = (M_{\gpoint} \backslash \{g_\omega\}) \cup \{g_\omega^\ell, g_\omega^r\}$ and $M_\omega {}_{\rpoint} = (M_{\rpoint} \backslash \{r_\omega\}) \cup \{r_\omega^\ell, r_\omega^r\}$ such that $g_\omega^\ell$ to $r_\omega^\ell$ precede $g_\omega^r$ to $r_\omega^r$ in the counter-clockwise order on their shared component of the boundary of $S_\omega$. We refer to $g_\omega^\ell, r_\omega^\ell$ and $g_\omega^r, r_\omega^r$ as to \textit{left added marked points} and  \textit{right added marked points}.

\begin{figure}[H]

    \centering
    
    \begin{tikzpicture}[x = 1cm,y = 1cm]
    \node[coordinate, name = GOmega] at (-3, 1) {};
    \node[coordinate, name = ROmega] at (-3, -1) {};
    
    \draw[dashed] (GOmega) -- (ROmega) node[midway, right] {$\omega$};
    
    \draw (GOmega) arc (270:250:5) node[coordinate, name = GOmegaAuxLeft]  {};
    \draw (GOmega) arc (270:290:5) node[coordinate, name = GOmegaAuxRight] {};
    \fill[pattern = north east lines]  (GOmegaAuxRight) -- (GOmegaAuxLeft) arc (250:290:5) -- cycle;
    
    \draw (ROmega) arc (90:110:5) node[coordinate, name = ROmegaAuxLeft]  {};
    \draw (ROmega) arc (90:70:5) node[coordinate, name = ROmegaAuxRight] {};
    \fill[pattern = north east lines]  (ROmegaAuxRight) -- (ROmegaAuxLeft) arc (110:70:5) -- cycle;
    
    \draw[dark-green, fill = white] (GOmega) circle (0.1) ;
    \draw[red, fill = red] (GOmegaAuxLeft) circle (0.1) ;
    \draw[red, fill = red] (GOmegaAuxRight) circle (0.1) ;

    \draw[red, fill = red] (ROmega) circle (0.1) ;
    \draw[dark-green, fill = white] (ROmegaAuxLeft) circle (0.1) ;
    \draw[dark-green, fill = white] (ROmegaAuxRight) circle (0.1) ;
    
    \node[below left] at (GOmega) {$g_\omega$};
    \node[above left] at (ROmega) {$r_\omega$};
    
    \draw[->, decorate, decoration = {snake, segment length = 2mm, amplitude = 0.25mm}] (-1,0) -- (1,0) node[] {};
    
    \node[coordinate, name = GOmegaCut] at (3, 1) {};
    \node[coordinate, name = ROmegaCut] at (3, -1) {};
    
    \path (GOmegaCut) arc (270:265:5) node[coordinate, name = GOmegaCut1] {};
    \path (GOmegaCut) arc (270:275:5) node[coordinate, name = GOmegaCut2] {};
    
    \path (ROmegaCut) arc (90:95:5) node[coordinate, name = ROmegaCut1] {};
    \path (ROmegaCut) arc (90:85:5) node[coordinate, name = ROmegaCut2] {};
    
    \draw (GOmegaCut1) arc (265:250:5) node[coordinate, name = GOmegaCut1Aux] {};
    \draw (GOmegaCut2) arc (275:290:5) node[coordinate, name = GOmegaCut2Aux] {};
    
    \draw (ROmegaCut1) arc (95:110:5) node[coordinate, name = ROmegaCut1Aux] {};
    \draw (ROmegaCut2) arc (85:70:5) node[coordinate, name = ROmegaCut2Aux] {};
    
    \draw (GOmegaCut1) -- (ROmegaCut1);
    \draw (GOmegaCut2) -- (ROmegaCut2);
    
    \fill[pattern = north east lines] (GOmegaCut1Aux) arc (250:265:5) -- (ROmegaCut1) arc (95:110:5) -- (ROmegaCut2Aux) arc (70:85:5) -- (GOmegaCut2) arc (275:290:5) -- cycle;
    
    \draw[dark-green, fill = white] (GOmegaCut1) circle (0.1) ;
    \draw[red, fill = red] (ROmegaCut1) circle (0.1) ;
    \draw[dark-green, fill = white] (ROmegaCut1Aux) circle (0.1) ;
    \draw[red, fill = red] (GOmegaCut1Aux) circle (0.1) ;
    
    \draw[dark-green, fill = white] (GOmegaCut2) circle (0.1) ;
    \draw[red, fill = red] (ROmegaCut2) circle (0.1) ;
    \draw[dark-green, fill = white] (ROmegaCut2Aux) circle (0.1) ;
    \draw[red, fill = red] (GOmegaCut2Aux) circle (0.1) ;
    
    \node[below right] at (GOmegaCut2) {$g_\omega^\ell$};
    \node[above right] at (ROmegaCut2) {$r_\omega^\ell$};
    
    \node[below left] at (GOmegaCut1) {$g_\omega^r$};
    \node[above left] at (ROmegaCut1) {$r_\omega^r$};
    
    \end{tikzpicture}
    
    \caption{Cutting $S$ along $\omega$}
    
    \label{FCutSurface}
    
\end{figure}
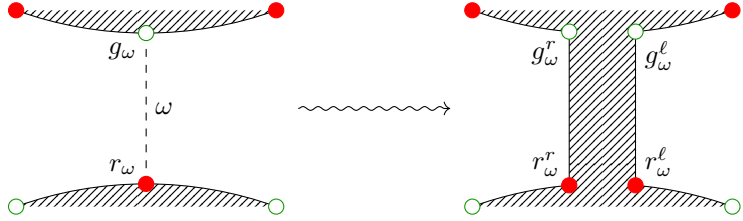

\end{definition}

\begin{remark}
It is easy to observe that $(S_\omega, M_\omega, P_\omega)$ is a well-defined marked surface because the new marked points are added in way that ensures that $\gpoint$-marked points and $\rpoint$-marked points still alternate on the component of $\partial S_\omega$ that they lie on. 
\end{remark}

\begin{remark}
It is always possible to reconstruct $(S, M, P)$ from $(S_\omega, M_\omega, P_\omega)$ by identifying the segments of $\partial S_\omega$ from $g_\omega^\ell$ to $r_\omega^\ell$ and from $g_\omega^r$ to $r_\omega^r$.
\end{remark}

We observe that if $\omega_1$ and $\omega_2$ are two dividing arcs on $(S, M, P)$ that do not intersect or share a common endpoint, it is possible to cut $(S, M, P)$ along them simultaneously. This observation is used in the following definition:

\begin{definition}\label{DGoodCut}
Let $(S, M, P)$ be marked surface. A collection of dividing arcs up to homotopy, $\Omega$, is called a \textit{good cut} if:
\begin{enumerate}[(i)]
    \item any two different arcs in $\Omega$ neither intersect nor meet at endpoints,
    \item no connected component of the cut surface $(S_\Omega, M_\Omega, P_\Omega)$ contains an left added marked point and an right added marked point,
    \item no connected component of the cut surface $(S_\Omega, M_\Omega, P_\Omega)$ is trivial, \textit{i.e.} homeomorphic to an open disk with only two marked points, one $\gpoint$-point and $\rpoint$-point, on its boundary.
\end{enumerate} 
\end{definition}

\begin{proposition}\label{PDissectionGoodCut}
Dividing arcs of a bipartite admissible dissection $(\Gamma, p)$ form a good cut $\Omega_{(\Gamma, p)}$ of the associated marked surface. On the other hand, a good cut $\Omega$ gives rise to a bipartite admissible dissection $(\Gamma_\Omega, p_\Omega)$ of the associated marked surface such that the dividing arcs of $(\Gamma_\Omega, p_\Omega)$ are the same as the arcs in $\Omega$.
\end{proposition}
\begin{proof}
Suppose that $(\Gamma, p)$ is a bipartite admissible dissection and that $\Omega_{(\Gamma, p)}$ is the associated set of dividing arcs. We will prove that $\Omega_{(\Gamma, p)}$ is a good cut of the underlying marked surface. Firstly, we observe that dividing arcs of $(\Gamma, p)$ lie in different connected components of the complement of $\Gamma$ in $S \, \backslash \, P$. Therefore, two dividing arcs of $\Gamma$ may not intersect. For the same reason, they may also not meet at $\rpoint$-marked point endpoints. If two dividing arcs of $(\Gamma, p)$ were to meet at a $\gpoint$-marked point, it would mean that in the counter-clockwise order there would be arcs $\gamma_1, \gamma_2, \gamma_3, \gamma_4$ with $p(\gamma_i) = 1, i \in \{1, 3\}$ and $p(\gamma_i) = 2, i \in \{2, 4\}$, a contradiction.

Assume now that, some connected component $\mcQ$ of the complement of $\Omega_{(\Gamma, p)}$ in $S \, \backslash \, P$ has both left and right added marked points on its boundary. It is easy to observe that if there is $g_\omega^\ell \in \partial \mcQ$ and $g_\omega^r \in \partial \mcQ$ for $\omega \in \Omega_{(\Gamma, p)}$, then we have $r_\omega^\ell \in \partial \mcQ$ and $r_\omega^r \in \partial \mcQ$, respectively, and vice versa because dividing arcs of $(\Gamma, p)$ do not intersect.

So there are $g_{\omega_1}^\ell, r_{\omega_1}^\ell, g_{\omega_2}^r$, and $r_{\omega_2}^r$ on the boundary of $\mcQ$ for some $\omega_1, \omega_2 \in \Omega_{(\Gamma, p)}$. This means that there are $\gpoint$-arcs $\gamma_1, \gamma_2 \in \Gamma$ such that $p(\gamma_i) = i$ that lie in $\mcQ$. Because $\mcQ$ is connected, we can take a $\gpoint$-arc $\sigma$ that connects $\gpoint$-marked points which are endpoints of the two dividing arcs pertaining to $\gamma_1$ and $\gamma_2$. Without loss of generality, $\sigma$ starts in the connected component of the complement of $\Gamma$ in $S \, \backslash \, P$ where $\omega_1$ lies, and $\sigma$ ends in the connected component of the complement of $\Gamma$ in $S \, \backslash \, P$ where $\omega_2$ lies.

Now, we consider a sequence, $\gamma_1^{(1)}, \dots, \gamma_{m_1}^{(1)}, \gamma_1^{(2)}, \dots, \gamma_{m_2}^{(2)}, \dots, \gamma_1^{(n)}, \dots, \gamma_{m_n}^{(n)}$, of $\gpoint$-arcs in $\Gamma$ such that $\sigma$ is homotopic to the iterated concatenation thereof as constructed in the proof of Proposition \ref{PBipartiteDissectionDecomposition}. The $\gpoint$-arc $\sigma$ in such a way that $\gamma_1^{(1)} = \gamma_1$ and $\gamma_{m_n}^{(n)} = \gamma_1$. As $p(\gamma_1) = 1$ and $p(\gamma_2) = 2$, the concatenation of $\gamma_1^{(1)}, \dots, \gamma_{m_1}^{(1)}, \gamma_1^{(2)}, \dots, \gamma_{m_2}^{(2)}, \dots, \gamma_1^{(n)}, \dots, \gamma_{m_n}^{(n)}$, which is $\sigma$, has to cross a dividing arc in $\Omega_{(\Gamma, p)}$. We have thus obtained a contradiction with the assumption that $\sigma$ is connected in the complement of $\Omega_{(\Gamma, p)}$ in $S \, \backslash \, P$.

To establish that $\Omega_{(\Gamma, p)}$ is a good cut, it remains to be shown that no connected component of its complement in $S \backslash P$ is trivial. However, for each $\omega \in \Omega_{(\Gamma, p)}$ there are two $\gpoint$-arcs $\gpoint$-arcs $\gamma_1, \gamma_2 \in \Gamma$ such that they share a $\gpoint$-endpoint with $\omega$ and that $p(\gamma_i) = i$. No dividing may cross either $\gamma_1$ or $\gamma_2$, so $\gamma_1$ lies in the component of the complement of $\Omega_{(\Gamma, p)}$ in $S \backslash P$ that has the left added marked point $g\_omega^\ell$ on this boundary, similarly for $\gamma_2$. Because every component of the complement of $\Omega_{(\Gamma, p)}$ in $S \backslash P$ contains an added marked point, it contains a $\gpoint$-arc, and so it is non-trivial.

Denote the marked surface $(S, M, P)$. Consider the connected components of the complement of $\Omega$ in $S \, \backslash \, P$. Each connected components of the complement of $\Omega$ in $S \, \backslash \, P$ is a marked surface and has an admissible dissection.

We choose an admissible dissection for each connected component and observe that together, as a set $\Gamma_\Omega$, they from an admissible dissection of the marked surface $(S, M, P)$. Arcs in admissible dissections of components that come from components with only left added marked points on their boundary are assigned $2$ by $p_\Omega$; other arcs, coming from components with only right added marked points are assigned $1$ by $p_\Omega$.

The $\gpoint$-arcs in $\Gamma_\Omega$ may clearly only meet at ther common $\gpoint$-endpoints since this is the case for every constituent connected component of the complement of $\Omega$ in $S \, \backslash \, P$, which are glued on segments on their boundary. Each connected component of the complement of $\Gamma_\Omega$ in $S \, \backslash \, P$ is glued from at most two connected components of the complement of $\Gamma_\Omega$ in the respective constituent connected components of the complement of $\Omega$ in $S \, \backslash \, P$ in the following way:

\begin{figure}[H]
    \centering
    
    \begin{tikzpicture}[x = 1cm, y = 1cm]
    
    \node[regular polygon, regular polygon sides = 5, minimum size = 3 cm, shape border rotate = 18] (Pen1) at (-2,0) {};
        
    \foreach \x in {1,...,5}
    {\node[coordinate, name = {Pen1Node\x}] at (Pen1.corner \x) {};}
    
    \node[regular polygon, regular polygon sides = 5, minimum size = 3 cm, shape border rotate = -18] (Pen2) at (2,0) {};
    
    \foreach \x in {1,...,5}
    {\node[coordinate, name = {Pen2Node\x}] at (Pen2.corner \x) {};}
    
    \draw (Pen1Node5) arc (36:-108:1.5) node[midway, coordinate, name = Pen1Red] {};
        
    \draw[dark-green] (Pen1Node5) -- (Pen1Node1);
    
    \draw[dark-green, dotted] (Pen1Node1) -- (Pen1Node2);
    
    \draw[dark-green] (Pen1Node2) -- (Pen1Node3);
    
    \draw (Pen2Node2) arc (144:288:1.5) node[midway, coordinate, name = Pen2Red] {};
        
    \draw[dark-green] (Pen2Node4) -- (Pen2Node5);
    
    \draw[dark-green, dotted] (Pen2Node5) -- (Pen2Node1);
    
    \draw[dark-green] (Pen2Node1) -- (Pen2Node2);
        
    \fill[pattern = north east lines] (Pen1Node3) arc (-108:36:1.5) to[out = 20, in = 160] (Pen2Node2) arc (144:288:1.5) to[out = 200, in = -20] (Pen1Node3); 
        
    \foreach \x in {1,2,3,5}
    {\draw[dark-green, fill = white] ({Pen1Node\x}) circle (0.1);}
    
    \foreach \x in {1,2,4,5}
    {\draw[dark-green, fill = white] ({Pen2Node\x}) circle (0.1);}
    
    \draw[red, fill = red] (Pen1Red) circle (0.1);
    
    \draw[red, fill = red] (Pen2Red) circle (0.1);
    
    \node[below left] at (Pen1Node5) {$g^r_\omega$};
    
   	\node[below right] at (Pen2Node2) {$g^\ell_\omega$};

    \node[above left] at (Pen1Red) {$r^r_\omega$};
    
    \node[above right] at (Pen2Red) {$r^\ell_\omega$};
    
    \end{tikzpicture}
    
    \begin{tikzpicture}[x = 1cm, y = 1cm]
    
    \path (0,1) -- (0,0.5);
    
   	\draw[->, decorate, decoration = {snake, segment length = 2mm, amplitude = 0.25mm}] (0,0.5) -- (0,-0.5) node[] {};
    	    
    \end{tikzpicture}
    
    \begin{tikzpicture}[x = 1cm, y = 1cm]
    
    \def\sides{5}
    \def\pensize{3}
        
    \pgfmathsetmacro{\penfactor}{0.5*cos(180/\sides)}
    \pgfmathsetmacro{\penonecenter}{-\pensize*\penfactor}
    \pgfmathsetmacro{\pentwocenter}{\pensize*\penfactor}
    
    \node[regular polygon, regular polygon sides = 5, minimum size = \pensize cm, shape border rotate = 18] (Pen1) at (\penonecenter,0) {};
        
    \foreach \x in {1,...,5}
    {\node[coordinate, name = {Pen1Node\x}] at (Pen1.corner \x) {};}
    
    \node[regular polygon, regular polygon sides = 5, minimum size = 3 cm, shape border rotate = -18] (Pen2) at (\pentwocenter,0) {};
    
    \foreach \x in {1,...,5}
    {\node[coordinate, name = {Pen2Node\x}] at (Pen2.corner \x) {};}
    
    \draw (Pen1Node3) to[out = 45, in = 135] node[midway, coordinate, name = Red] {} (Pen2Node4);
        
    \draw[dark-green] (Pen1Node5) -- (Pen1Node1);
    
    \draw[dark-green, dotted] (Pen1Node1) -- (Pen1Node2);
    
    \draw[dark-green] (Pen1Node2) -- (Pen1Node3);
            
    \draw[dark-green] (Pen2Node4) -- (Pen2Node5);
    
    \draw[dark-green, dotted] (Pen2Node5) -- (Pen2Node1);
    
    \draw[dark-green] (Pen2Node1) -- (Pen2Node2);
    
    \draw[dashed] (Pen1Node5) -- (Red);
    
    \fill[pattern = north east lines] (Pen1Node3) to[out = 45, in = 135] node[midway, coordinate, name = Red] {} (Pen2Node4) -- cycle; 
    
    \foreach \x in {1,2,3,5}
    {\draw[dark-green, fill = white] ({Pen1Node\x}) circle (0.1);}
    
    \foreach \x in {1,4,5}
    {\draw[dark-green, fill = white] ({Pen2Node\x}) circle (0.1);}
    
    \draw[red, fill = red] (Red) circle (0.1);
    
    \node[below left] at (Pen1Node5) {$g_\omega$};
    
    \node[above left] at (Red) {$r_\omega$};
    
    \path (Pen1Node5) -- (Red) node[midway, right] {$\omega$};
    	
    \end{tikzpicture}

    \caption{Gluing connected components of of the complement of $\Omega$ in $S \, \backslash \, P$.}
    \label{FGluingComponentsDissection}
\end{figure}
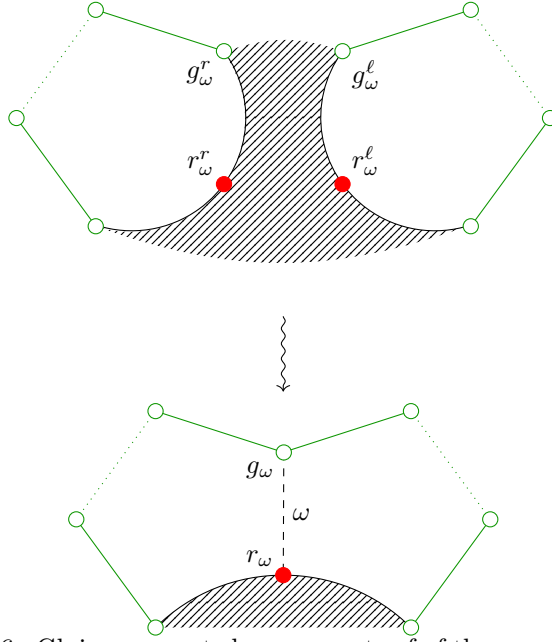

Hence every connected component of the complement of $\Gamma_\Omega$ in $S \, \backslash \, P$ has a single $\rpoint$-point, and $\Gamma_\Omega$ is an admissible dissection.

If two arcs in different parts of $(\Gamma_\Omega, p_\Omega)$ meet, they have to meet at a $\gpoint$-endpoint of one of the dividing arcs in $\Omega$. The $\gpoint$-arc in part $2$ lies follows the dividing arc in the counterclockwise order around the common endpoint because the component to which the $\gpoint$-arc belong lies to the left of the dividing arc; whereas, the $\gpoint$-arc in part $1$ has to precede in that order. This is enough to conclude that $(\Gamma_\Omega, p_\Omega)$ is a bipartite admissible dissection as well as that $\Omega$ is the set of dividing arcs of $(\Gamma_\Omega, p_\Omega)$.
\end{proof}

\begin{remark}
The proof that $\Gamma_\Omega$ may be also performed numerically using the following inductive argument as in the proof of Proposition 1.11 in \cite{amiot2019complete}. By $\gpoint$-arcs in admissible dissection equals $|M_{\gpoint}| + |P| + b + 2g - 2$ arcs, where $g$ is the genus of $S$ and $b$ is the number of connected components of $\partial S$. Consider $\omega$ a dividing arc of $S$ and the cut surface $S_\omega$. If $\omega$ connects two different connected components of the boundary $\partial S$ (Case 3 in the proof of Proposition 1.11 in \cite{amiot2019complete}), then the number of connected components of $S_\omega$, $b_\omega$, equals $b-1$. We have an extra $\gpoint$-marked point, so $(M_\omega)_{\gpoint} = M_{\gpoint} + 1$. It is easy to observe that the genus of $S_\omega$, $g_\omega$, and the number $\rpoint$-punctures remain unchanged, so $g = g_\omega$ and $|P_\omega| = |P|$. We get that any admissible dissection of $S_\omega$ has to have the same number of $\gpoint$-arcs as that of $S$. If $\omega$ starts and ends on the same boundary component, we need to distinguish whether $\omega$ is separating or non-separating. Provided that $\omega$ is non-separating (Case 1 in the proof of Proposition 1.11 in \cite{amiot2019complete}), the cut surface $S_\omega$ has a single connected component with $b+1$ boundary components and genus $g-1$. It has an extra $\gpoint$-marked point and the same number of punctures. Because $|(M_\omega)_{\gpoint}| + |P_\omega| + b_\omega + 2g_\omega - 2 = (|M_{\gpoint}| + 1) + |P| + (b + 1) + 2(g - 1) - 2 = |M_{\gpoint}| + |P| + b + 2g - 2$, the any admissible dissection of $S_\omega$ has to have the same number of $\gpoint$-arcs as that of $S$. Finally, if $\omega$ is separating (Case 6 in the proof of Proposition 1.11 in \cite{amiot2019complete}), then $S_\omega$ has two connected components $S_\omega^{(1)}$ and $S_\omega^{(2)}$ with combined genus, $g_\omega^{(1)} + g_\omega^{(2)}$, equal to $g$ and combined number of boundary components, $b_\omega^{(1)} + b_\omega^{(2)}$, equal to $b+1$. They also have an additional marked point together, in other words $|(M_\omega^{(1)})_{\gpoint}| + |(M_\omega^{(2)})_{\gpoint}| = |M_{\gpoint}|$. Therefore the number of $\gpoint$-arcs in a pair admissible dissections of $S_\omega^{(1)}$ and $S_\omega^{(2)}$ equals to $|(M_\omega^{(1)})_{\gpoint}| + |P_\omega^{(1)}| + b_\omega^{(1)} + 2g_\omega^{(1)} - 2 + |(M_\omega^{(2)})_{\gpoint}| + |P_\omega^{(2)}| + b_\omega^{(2)} + 2g_\omega^{(2)} - 2 = (|M_{\gpoint}| + 1) + |P| + (b + 1) + 2g - 4$, which is the same as for $S$.
\end{remark}

\begin{theorem}\label{TDecompositionsGoodCuts}
There is a one-to-one correspondence between semiorthogonal decompositions of $\Kb(\Lambda\proj)$ and good cuts of the marked surface associated to $\Lambda$.
\end{theorem}
\begin{proof}
For the purposes of this proof, we denote $\mcS \mcO$ the set of semiorthogonal decompositions of $\Kb(\Lambda \proj)$ and $\mcG \mcC$ the set of cuts of $(S, M, P)$.

At first, we construct a map $F\colon \mcS \mcO \to \mcG \mcC$. Let us have $\langle \mcL, \mcR \rangle$, a semiorthogonal decomposition of $\Kb(\Lambda\proj)$. Consider the following approximation of $R_i \to P_i \to L_i \to R_i[1]$ such that and $R_i \in \mcR$, $L_i \in \mcL$, and $P_i$ is one of the indecomposable direct summands of $\Lambda$. Theorem \ref{TProjectivesDecomposition} paired with Lemma 3.10 in \cite{opper2018geometric} yields that: $$R_i \cong \bigoplus_{j} P_{(\varrho_j^{(i)}, g_j^{(i)})} \, \, \, \mbox{and} \, \, \, L_i \cong \bigoplus_{j'} P_{(\lambda_{j'}^{(i)}, \ell_{j'}^{(i)})}$$ for $\gpoint$-arcs $\varrho_j^{(i)}$ and $\lambda_{j'}^{(i')}$; these $\gpoint$-arcs meet only at endpoints, and that $\varrho_j^{(i)}$ may meet $\lambda_{j'}^{(i')}$ at a $\gpoint$-point only if $\varrho_j^{(i)}$ follows $\lambda_{j'}^{(i')}$ in the counter-clockwise order around a common endpoint.

These $\gpoint$-arcs gives rise to a bipartite collection $(\Gamma, p)$ with $p(\lambda_{j'}^{(i')}) = 1$ and $p(\varrho_j^{(i)}) = 2$ by possibly removing redundant $\gpoint$-arcs so that each connected component of its complement in contains at least one $\rpoint$-point.

It easy to observe that the removed $\gpoint$-arcs are simply concatenations of the remaining $\gpoint$-arcs, and the objects corresponding to the removed $\gpoint$-arcs can be generated using the objects corresponding to the remaining $\gpoint$-arcs for both parts respectively. This holds because if there is a component of of the complement of without a $\rpoint$-marked point, then all $\gpoint$-arcs on its boundary have to be either $\lambda_{j'}^{(i')}$ or $\varrho_j^{(i)}$ as discussed in the proof of Proposition \ref{PBipartiteDissectionDecomposition}.

Proposition \ref{PBipartiteDissectionCompletion} assures existence of a bipartite admissible dissection $(\Gamma', p')$ extending $(\Gamma, p)$. If $\langle \mcL_{(\Gamma', p')}, \mcR_{(\Gamma', p')} \rangle$ is the semiorthogonal decomposition of $\Kb(\Lambda \proj)$ associated to $(\Gamma', p')$ by Proposition \ref{PBipartiteDissectionDecomposition}, we obtain that $\mcL \subseteq \mcL_{(\Gamma', p')}$ and $\mcR \subseteq \mcR_{(\Gamma', p')}$ necessitating that the inclusions are equalities.

We note that $(\Gamma, p)$ must have already been a bipartite admissible dissection. Suppose that we have a $\rpoint$-arc $\psi$. Since the $\gpoint$-arcs $\sigma_i$ corresponding to indecomposable projectives $P_i$ form an admissible dissection, there exists $\sigma_i$ such that $\psi$ and $\sigma_i$ intersect. The uniqueness of approximation triangles in $\langle \mcL, \mcR \rangle = \langle \mcL_{(\Gamma', p')}, \mcR_{(\Gamma', p')} \rangle$ per Proposition \ref{PSemiorthogonalDecomposition} means that we can use Proposition \ref{PBipartiteDissectionTriangles} to show that $\sigma_i$ homotopic to an iterated concatenation of $\gpoint$-arcs $\lambda_{j'}^{(i)}$ and $\varrho_{j}^{(i)}$. It is an easy consequence that $\psi$ has to cross some of these $\gpoint$-arcs. We can therefore conclude that $(\Gamma, p)$ separates all $\rpoint$-points, and so it is a bipartite admissible dissection. 

We set $F(\langle \mcL, \mcR \rangle)$ equal to the set of dividing arcs $(\Gamma', p')$. The bipartite admissible dissection $(\Gamma', p')$ may not be unique, but all possible choices of $(\Gamma', p')$ have the same dividing arcs. Therefore, the map $F\colon \mcS \mcO \to \mcG \mcC$ is well defined.\\\\
Conversely, we construct a map $G: \mcG \mcC \to \mcS \mcD$. Take a good cut $\Omega$ of the underlying marked surface. By Proposition \ref{PDissectionGoodCut}, there exists a bipartite admissible dissection $(\Gamma', p')$ such that its set of dividing arcs equals $\Omega$. By Propostion \ref{PBipartiteDissectionDecomposition}, we can associate a semiorthogonal decomposition $G(\Omega)$ of $\Kb(\Lambda \proj)$ to $\Omega$. Since two bipartite admissible dissections with same dividing arcs give rise to the same semiorthogonal decomposition by Proposition \ref{PBipartiteDissectionsEquivalence}, the map $G: \mcG \mcC \to \mcS \mcD$ is well-defined.\\\\
Proposition \ref{PDissectionGoodCut} gives us that two bipartite admissible dissections give rise to the same semiorthogonal decomposition by Proposition \ref{PBipartiteDissectionsEquivalence} if and only if they have the same dividing arcs. This fact is then used to assure that $F$ and $G$ are mutually inverse bijections.
\end{proof}

\subsection{Semiorthogonal decompositions of \texorpdfstring{$\Db(\Lambda \modl)$}{Db(Lambda-mod)} and \texorpdfstring{$\Kb(\Lambda \proj)$}{Kb(Lambda-proj)}}
Finally, we prove that there is a bijective correspondence between semiorthogonal decompositions of $\Db(\Lambda \modl)$ and semiorthogonal decompositions of $\Kb(\Lambda \proj)$, which goes from left to right by restriction. This results represents the final step in establishing the bijective correspondence between semiorthogonal decompositions of $\Db(\Lambda \modl)$ and good cuts of the underlying marked surface $(S, M, P)$.

\begin{theorem}\label{TSemiorthogonalDecompositionRestriction}
A semiorthogonal decomposition of $\Db(\Lambda\modl)$ restricts to a semiorthogonal decomposition of $\Kb(\Lambda \proj)$. On the other hand, a semiorthogonal decomposition of $\Kb(\Lambda \proj)$ can be uniquely extended to a semiorthogonal decomposition of $\Db(\Lambda\modl)$.
\end{theorem}
\begin{proof}
At first, we prove that a semiorthogonal decomposition of $\Db(\Lambda\modl)$ restricts to a semiorthogonal decomposition of $\Kb(\Lambda \proj)$. The first part of the argument is similar to the argument made in the proof of Theorem \ref{TDecompositionsGoodCuts}, when constructing the map from semiorthogonal decompositions of $\Kb(\Lambda \proj)$ to the good cuts of the underlying marked surface.

Suppose that $\langle \mcL, \mcR \rangle$ is a semiorthogonal decomposition of $\Db(\Lambda\modl)$. Consider the following approximation of $R_i \to P_i \to L_i \to R_i[1]$ such that and $R_i \in \mcR$, $L_i \in \mcL$, and $P_i$ is one of the indecomposable direct summands of $\Lambda$. Theorem \ref{TProjectivesDecomposition} paired with Lemma 3.10 in \cite{opper2018geometric} yields that: $$R_i \cong \bigoplus_{j} P_{(\varrho_j^{(i)}, g_j^{(i)})} \, \, \, \mbox{and} \, \, \, L_i \cong \bigoplus_{j'} P_{(\lambda_{j'}^{(i')}, \ell_{j'}^{(i')})}$$ for $\gpoint$-arcs $\varrho_j^{(i)}$ and $\lambda_{j'}^{(i')}$; these $\gpoint$-arcs meet only at endpoints, and that $\varrho_j^{(i)}$ may meet $\lambda_{j'}^{(i')}$ at a $\gpoint$-point only if $\varrho_j^{(i)}$ follows $\lambda_{j'}^{(i')}$ in the counter-clockwise order around a common endpoint.

Clearly, $\mcL \cap \Kb(\Lambda \proj)$ contains all objects corresponding to $\lambda_{j'}^{(i')}$, and $\mcR \cap \Kb(\Lambda \proj)$ contains all objects corresponding to $\varrho_j^{(i)}$. Also, there are no homomorphisms from $\mcR \cap \Kb(\Lambda \proj)$ to $\mcR \cap \Kb(\Lambda \proj)$. Because $\mcL \cap \Kb(\Lambda \proj)$ and $\mcR \cap \Kb(\Lambda \proj)$ generate $\Kb(\Lambda \proj)$ via its generator $\Lambda$, which is the direct sum of all indecomposable projectives $P_i$, they form a semiorthogonal decomposition. Therefore, $\langle \mcL, \mcR \rangle$ restricts to a semiorthogonal decomposition of $\Kb(\Lambda \proj)$.

Now, it remains to be proved that a semiorthogonal decomposition of $\Kb(\Lambda \proj)$ can be uniquely extended to a semiorthogonal decomposition of $\Db(\Lambda\modl)$.

Let $\langle \mcL, \mcR \rangle$ be a semiorthogonal decomposition of $\Kb(\Lambda \proj)$. Our strategy is to find enough addditonal indecomposable objects in $\Db(\Lambda\modl)$ to add to $\mcL$ and $\mcR$, so that enriched $\mcL$ and $\mcR$ generated the bounded derived category, but there are still no morphisms from the enriched $\mcR$ to enriched $\mcL$ (cf. Proposition \ref{PSemiorthogonalDecomposition}). This enrichment will be constructed in a way that will assure its uniqueness.

The indecomposable objects of $\Db(\Lambda\modl)$ that are not in $\Kb(\Lambda \proj)$; these are exactly objects corresponding to infinite string complexes. In the geometric model, infinite string complexes are represented by arcs between a $\gpoint$-point and a $\rpoint$-puncture or two $\rpoint$-punctures (Theorem 2.12 and Remark 1.20 in \cite{opper2018geometric}). Such arcs infinitely wrap around the $\rpoint$-puncture which is its endpoint in the counter-clockwise direction.

Using Theorem \ref{TDecompositionsGoodCuts}, we can find a bipartite admissible dissection $(\Gamma, p)$ of the underlying marked surface $(S, M, P)$ that gives $\langle \mcL, \mcR \rangle$ by Proposition \ref{PBipartiteDissectionDecomposition}.

Consider the complement of $\Gamma$ in $S \, \backslash \, P$. Since $\Gamma$ is an admissible dissection of the marked surface, each $\rpoint$-puncture lies in a component of the complement of $\Gamma$ in $S \, \backslash \, P$ whose boundary is composed solely of arcs in $\Gamma$ and their endpoints (Proposition 1.12 in \cite{amiot2019complete}). It follows from the fact that $(\Gamma, p)$ is a bipartite admissible dissection that all arcs enclosing a component with a $\rpoint$-puncture need to belong to the same part.

For each $\rpoint$-puncture $t \in P_{\rpoint}$, we choose a $\gpoint$-endpoint of an $\gpoint$-arc in $\Gamma$ that lies on the boundary of the connected component $\mcP_t$ of the complement of $\Gamma$ in $S \backslash P$ pertaining to $t$, and we consider a $\gpoint$-infinite arc $\pi_t$ from this $\gpoint$-endpoint that infinitely wraps around the $\rpoint$-puncture $t$ in the counterclockwise direction and is contained within $\mcP_t$.

We now observe that every $\gpoint$-infinite arc on the marked surface is a concatenation of a $\gpoint$-arc and at most two arcs $\pi_s, \pi_t$ for $\rpoint$-punctures $t, s \in P$. This decomposition is induced simply by taking the complement of $\Gamma$ in $S \backslash P$.

Consider $\mcL'$ and $\mcR'$ where $\mcL'$ and $\mcR'$ are generated by $\mcL$ and $\mcR$ and $\pi_t$ such that $\mcP_t$ is enclosed by $\gpoint$-arcs in the first part and the second part, respectively. Our observation above can be combined with Theorem 2.1 in \cite{opper2018geometric} and Proposition \ref{PBipartiteDissectionDecomposition} to yield that the triangulated category generated by $\mcL'$ and $\mcR'$ contains all infinite string objects string objects of $\Db(\Lambda \modl)$. The triangulated category generated by $\mcL'$ and $\mcR'$ equals $\Db(\Lambda \modl)$ because it contains all the indecomposable objects thereof.

At this stage, we can extend a semiorthogonal decomposition of $\Kb(\Lambda \proj)$; finally, we will demonstrate this extension is unique.

Suppose that $\langle \tilde{\mcL}, \tilde{\mcR} \rangle$ is semiorthogonal decomposition of $\Db(\Lambda \modl)$ that extends $\langle \mcL, \mcR \rangle$. For each $\rpoint$-puncture $t \in P$, we consider the approximation $R_{(\pi_t, d_t)} \to P_{(\pi_t, d_t)} \to L_{(\pi_t, d_t)} \to R_{(\pi_t, d_t)}[1]$ of $P_{(\pi_t, d_t)}$ for some grading $d_t$ that exists by Proposition \ref{PSemiorthogonalDecomposition}. Our aim is to show that this approximation is trivial, and $P_{(\pi_t, d_t)} = L_{(\pi_t, d_t)}$ or $P_{(\pi_t, d_t)} = R_{(\pi_t, d_t)}$ based on whether the border of $\mcP_t$ consists of $\gpoint$-arcs in $\Gamma$ in the first part or the second part, respectively. This will assure that $\langle \tilde{\mcL}, \tilde{\mcR} \rangle = \langle \mcL', \mcR' \rangle$.

Let us have a $\rpoint$-puncture $t \in P$ and suppose that all the objects corresponding to arcs on the boundary of $\mcP_t$ belong to $\mcL$. To prove that  $P_{(\pi_t, d_t)} = L_{(\pi_t, d_t)}$, it is enough to show that there is no non-trivial morphism $f\colon R \to P_{(\pi_t, d_t)}$ for any $R \in \mcR$. Without loss of generality, we may assume that $R$ is indecomposable and that $R$ is isomorphic to $P_{(\sigma, h)}$, for $\sigma$, either a $\gpoint$-arc, an $\gpoint$-infinite arc, or an arc representing an infinite string complex, with some grading $h$, or $P_{(\sigma, h, M)}$, for a closed arc $\sigma$ representing a one-dimensional band complex with grading $h$ and $M$ the associated isomorphism class of $k[x]$-modules (cf. Theorem \ref{TIndecomposablesStringsBands} and Theorem 2.12 in \cite{opper2018geometric}).

By Theorem 3.3 in \cite{opper2018geometric}, the map $f$ corresponds to an intersection of $\pi_t$ and $\sigma$ or a common endpoint provided that $P_{(\sigma, h)} \in \Kb(\Lambda \proj)$. If $\pi_t$ and $\sigma$ shared a $\gpoint$-endpoint inducing a map $P_{(\sigma, h)} \to P_{(\pi_t, d_t)}$, it would mean that $\sigma$ would precede $\pi_t$ in the counter-clockwise order around it. However, as $\pi_t$ is succeeded in the counter-clockwise order around its $\gpoint$-endpoint by a $\gpoint$-arc $\gamma \in \Gamma$, part of boundary of $\mcP_t$. This would yield an existence of a map $P_{\sigma, h} \to P_{(\gamma, c)}$ for some grading $c$ (Remark 3.8 in \cite{opper2018geometric}), a contradiction given that $P_{(\gamma, c)} \in \tilde{\mcL}$.

So any potential map must come from an intersection of $\sigma$ and $\pi_t$ in the interior of the underlying marked surface, and, up to homotopy, $\sigma$ needs to be contained in $\mcP_t$; otherwise, $\sigma$ would have an intersection with a $\gpoint$-arc $\gamma \in \Gamma$ on the boundary of $\mcP_t$, which would imply an existence of a map $R \to P_{(\gamma, c)}$ for some grading $c$, a contradiction given that $P_{(\gamma, c)} \in \tilde{\mcL}$.

We can assume that $\sigma$ lies in $\mcP_t$. If $\sigma$ is a $\gpoint$-arc, this guarantees that it has a $\gpoint$-endpoint on the boundary of $\mcP_t$. However, there needs to be an arc $\gamma$ on the boundary of $\mcP_t$ that follows $\sigma$ in the counter-clockwise order around the $\gpoint$-endpoint implying an existence of a map $R \to P_{(\gamma, c)}$ for some grading $c$, a contradiction as $P_{(\gamma, c)} \in \tilde{\mcL}$. If $\sigma$ is a closed arc, it is, up to homotopy, a loop that winds around $t$ a given number of times as $\mcP_t$ is homotopy equivalent to a circle whose fundamental group is $\Z$. Such a closed arc cannot be equipped with a grading to yield an object of $\Db(\Lambda \modl)$ as each time a loop winds around $t$ the grading decreases or increases by more than $1$ (Definition 2.10 and Remark 2.11 in \cite{opper2018geometric}):

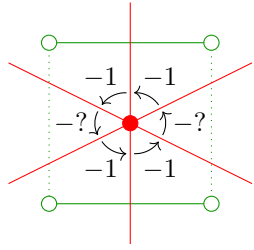
\begin{figure}[H]
	\centering
	
	\begin{tikzpicture}[x = 1cm, y = 1cm]
	
	    \pgfmathsetmacro{\sides}{4}
	    \pgfmathsetmacro{\sideshalf}{\sides/2}
	
        \node[regular polygon, regular polygon sides = 4, minimum size = 3cm] (Pol) at (0,0) {};
        
        \foreach \x in {1,...,\sides}
        {\node[coordinate, name = {PolNode\x}] at (Pol.corner \x) {};}
        
        \foreach \x in {1,...,\sideshalf}
        {\pgfmathtruncatemacro{\xstart}{2*\x - 1}
         \pgfmathtruncatemacro{\xend}{2*\x}
        \draw[dark-green] ({PolNode\xstart}) -- ({PolNode\xend}) node[coordinate, midway, name = {PolNodeAux\xstart}] {} ;}
        
        \foreach \x in {1,...,\sideshalf}
        {\pgfmathtruncatemacro{\xstart}{2*\x}
         \pgfmathtruncatemacro{\xend}{mod(2*\x + 1, \sides)}
        \draw[dark-green, dotted] ({PolNode\xstart}) -- ({PolNode\xend}) node[coordinate, near start, name = {PolNodeAuxRight\xstart}] {} node[coordinate, near end, name = {PolNodeAuxLeft\xstart}] {};}
        
        \foreach \x in {1,...,\sides}
        {\draw[dark-green, fill = white] ({PolNode\x}) circle (0.1);}
        
        \node[coordinate, name = Puncture] at (0,0) {};
        
        \draw[red, fill = red] (Puncture) circle (0.1);
        
        \pgfmathsetmacro{\redlinefactor}{1.5}
        
        \foreach \x in {1,...,\sideshalf}
        {\pgfmathtruncatemacro{\xstart}{2*\x - 1}
        \draw[red] (Puncture) -- ($\redlinefactor*({PolNodeAux\xstart})$) node[coordinate, near start, name = {GradingAux\xstart}] {};}
        
        \foreach \x in {1,...,\sideshalf}
        {\pgfmathtruncatemacro{\xstart}{2*\x}
        \draw[red] (Puncture) -- ($\redlinefactor*({PolNodeAuxLeft\xstart})$) node[coordinate, near start, name = {GradingAuxLeft\xstart}] {};}
        
        \foreach \x in {1,...,\sideshalf}
        {\pgfmathtruncatemacro{\xstart}{2*\x}
        \draw[red] (Puncture) -- ($\redlinefactor*({PolNodeAuxRight\xstart})$) node[coordinate, near start, name = {GradingAuxRight\xstart}] {};}
        
        \pgfmathsetmacro{\gradingfactor}{1.7}
        
        \foreach \x in {1,...,\sideshalf}
        {\pgfmathtruncatemacro{\xstartone}{2*\x}
        \pgfmathtruncatemacro{\xstarttwo}{mod(2*\x + 1, \sides)}
        \path ({GradingAuxLeft\xstartone}) edge[->, bend right, shorten >= 0.05cm, shorten <= 0.05cm] node[midway, name = {LabelAuxLeft\xstartone}] {} ({GradingAux\xstarttwo});
        \path (Puncture) -- ($\gradingfactor*({LabelAuxLeft\xstartone})$) node[at end] {$-1$};}
        
        \foreach \x in {1,...,\sideshalf}
        {\pgfmathtruncatemacro{\xstart}{2*\x}
        \path ({GradingAuxRight\xstart}) edge[->, bend right, shorten >= 0.05cm, shorten <= 0.05cm] node[midway, name = {LabelAuxRight\xstart}] {} ({GradingAuxLeft\xstart});
        \path (Puncture) -- ($\gradingfactor*({LabelAuxRight\xstart})$) node[at end] {$-?$};}
        
        \foreach \x in {1,...,\sideshalf}
        {\pgfmathtruncatemacro{\xstartone}{2*\x - 1}
        \pgfmathtruncatemacro{\xstarttwo}{2*\x}
        \path ({GradingAux\xstartone}) edge[->, bend right, shorten >= 0.05cm, shorten <= 0.05cm] node[midway, name = {LabelAux\xstartone}] {} ({GradingAuxRight\xstarttwo});
        \path (Puncture) -- ($\gradingfactor*({LabelAux\xstartone})$) node[at end] {$-1$};}
        
    \end{tikzpicture}
			
	\label{FGradingAroundPucnture}
	\caption{Changes in grading for graded curves in $\mcQ_t$}
\end{figure}
This can be observed by taking the admissible dissection $\Delta$ of the underlying marked surface that encodes the structure of $\Lambda$. The $\rpoint$-point puncture $t$ is enclosed in a connected component $\mcQ_t$ of the complement of $\Delta$ in $S \backslash P$ by $\gpoint$ $\delta_1, \dots, \delta_n \in \Delta$ that follow each other in the counter-clockwise order around $t$. In the dual admissible $\Delta^*$ (cf. Proposition \ref{PAdmissibleDissectionProperties}), there is a $\rpoint$-arc $\delta^*_i$ for each $\delta_i$ such that crossing two following dual $\rpoint$-point arcs in the counter-clockwise order decreases the grading by 1 and vice versa by Definition \ref{DGradedCurves}. A loop that winds around $t$ in $\mcP_t$ can be assumed to lie in $\mcQ_t$ as well. So every time, this loop winds around its grading would increase or decrease by $n \geq 1$; such a loop cannot represent a one-dimensional band object as it is not possible to grade it properly.

Lastly, $\sigma$ may be a $\gpoint$-infinite arc. By our discussion above, $\sigma$ needs to be contained in $\mcP_t$, otherwise crossing a boundary $\gpoint$-arc of $\mcP_t$ yielding a morphism from $\tilde{\mcR}$ to $\tilde{\mcL}$. In this instance, however, it has have a $\gpoint$-endpoint on the boundary of $\mcP_t$ and a boundary $\gpoint$-arc of $\mcP_t$ following it in the counter-clockwise order around the $\gpoint$-enpoint, which is also not permitted by our discussion above.

The other case, when the boundary of $\mcP_t$ is formed by $\gpoint$-arcs in $\Gamma$ that give rise to objects in $\mcR$, is discussed dually.
\end{proof}

\begin{theorem}\label{TDecompositionsGoodCuts2}
There is a one-to-one correspondence between semiorthogonal decompositions of $\Db(\Lambda\modl)$ and good cuts of the marked surface associated to $\Lambda$.
\end{theorem}
\begin{proof}
This a simple combination of Theorem \ref{TSemiorthogonalDecompositionRestriction} and Theorem \ref{TDecompositionsGoodCuts}.
\end{proof}

\subsection{Semiorthogonal decompositions of \texorpdfstring{$\Db(\Lambda \modl)$}{Db(Lambda-mod)} with multiple terms}
We conclude this section by a discussion on how the results on characterization of two-term semiorthogonal decompositons in Theorem \ref{TDecompositionsGoodCuts2} extends to semiorthogonal decomposition of $\Db(\Lambda \modl)$ with more than two terms.

\begin{definition}\label{DProperSequenceCuts}
Let $(S, M, P)$ be marked surface. A sequence of good cuts $(\Omega_1, \dots, \Omega_n)$, is called a \textit{proper} if the good cuts are pair-wise different up to homotopy and if for every $1 \leq i \leq n-1$:
\begin{enumerate}[(i)]
    \item each $\omega \in \Omega_{i+1}$ that lies on a connected component of the complement of $\Omega_{i}$ in $S \backslash P$ with right added marked points is homotopic to a component of the boundary thereof;
    \item each $\omega \in \Omega_{i}$ that lies on a connected component of the complement of $\Omega_{i_+1}$ in $S \backslash P$ with left added marked points is homotopic to a component of the boundary thereof.
\end{enumerate}
\end{definition}

\begin{proposition}\label{PProperSequenceDecomposition}
Let $(S, M, P)$ be marked surface, and let $(\Omega_1, \Omega_2)$ be a proper sequence of good cuts. Denote $\langle \mcL_i, \mcR_i \rangle$ the semiorthogonal decomposition of $\Db(\Lambda \modl)$ corresponding to $\Omega_i$ by Theorem \ref{TDecompositionsGoodCuts2}, for $i = 1, 2$. Then, $\mcL_1 \subsetneq \mcL_2$.
\end{proposition}
\begin{proof}
Consider a $\gpoint$-arc $\gamma$ on $(S, M, P)$ whose corresponding objects in $\Db(\Lambda \modl)$ lie in $\mcL_1$ of the semiorthogonal decomposition $\langle \mcL_1, \mcR_1 \rangle$ thereof. The $\gpoint$-arc $\gamma$ is contained within a connected component $\mcP$ of the complement of $\Omega_1$ in $S \backslash P$. By definition, the component $\mcP$ would be with right added marked points. Pick $\omega \in \Omega_1$ that lies on the boundary of $\mcP$ and a $\gpoint$-arc $\gamma'$ such that: it begins in $g_\omega$, the $\gpoint$-endpoint of $\omega$, and it is contained in $\mcP$. Therefore, $\gamma'$ has to precede $\omega$ in the counter-clockwise order around $g_\omega$.

Because $(\Omega_1, \Omega_2)$ is proper, the connected component $\mcP$ does not contain any $\omega \in \Omega_2$ not homotopic to a segment of its boundary. For this reason, neither $\gamma$ nor $\gamma'$ cross an arc in $\Omega_2$, and the corresponding objects of $\Db(\Lambda \modl)$ need to both lie either in $\mcL_{\Omega_2}$ or in $\mcR_{\Omega_2}$ by Lemma \ref{LCrossingDividingArcs} combined with Proposition \ref{PDissectionGoodCut}.

Furthermore, as $(\Omega_1, \Omega_2)$ is proper, we may assume that without loss of generality $\omega$ lies within a connected component $\mcQ$ of the complement of $\Omega_2$ in $S \backslash P$ with right added marked points or on a boundary thereof. If $\omega$ lies on $Q$, then so does $\gamma'$, which does not cross any arc in $\Omega_2$ and $\gamma'$ precedes $\omega$ in the counter-clockwise order around $g_\omega$. Hence, the objects of $\Db(\Lambda \modl)$ corresponding to $\gamma'$ lie in $\mcL_{\Omega_2}$; it is also the case for the objects of $\Db(\Lambda \modl)$ corresponding to $\gamma$, so $\mcL_1 \subseteq \mcL_2$.

The fact that $\mcL_1 \subsetneq \mcL_2$ follows from the assumption that $\Omega_1$ and $\Omega_2$ are not equal up to homotopy. Theorem \ref{TDecompositionsGoodCuts2} then assures that the corresponding semiorthogonal decompositions differ as well.
\end{proof}

\begin{theorem}\label{TDecompositionsProperSequences}
There is a one-to-one correspondence between semiorthogonal decompositions of $\Db(\Lambda\modl)$ with arbitrarily many terms and proper sequences of good cuts of the marked surface associated to $\Lambda$.
\end{theorem}
\begin{proof}
At first, we note that it is a straight-forward consequence of Proposition \ref{PSemiorthogonalDecomposition} that a semiorthogonal decomposition $\langle \mcC_1, \dots, \mcC_n \rangle$ of $\mcT$ uniquely corresponds to a sequence of two-term semirothogonal decompositions $\langle \mcL_i, \mcR_i \rangle$ of $\mcT$, for $1 \leq i \leq n-1$, such that $\mcL_i$ is the full triangulated subcategory of $\mcT$ generated by $\mcC_1, \dots, \mcC_i$ and $\mcR_i$ is the full triangulated subcategory of $\mcT$ generated by $\mcC_{i+1}, \dots, \mcC_n$, and vice versa.

Combining Theorem \ref{TDecompositionsGoodCuts2} with Proposition \ref{PProperSequenceDecomposition}, we obtain that there is a one-to-one correspondence with proper sequence of good cuts of the marked surface $(S, M, P)$ associated to $\Lambda$ and a sequence $\langle \mcL_i, \mcR_i \rangle$ of semiorthogonal decompositions of $\Db(\Lambda \modl)$ with $\mcL_i \subsetneq \mcL_{i+1}$ for each $i$ smaller than the length of the sequence. This correspondence is carried forward to a one-to-one correspondence between proper sequences of good cuts by our observation above.
\end{proof}

%\section{New material}
%\subsection{Surface-like categories}

\bibliographystyle{plain}
\bibliography{bibliography}

\begin{thebibliography}{10}

\bibitem{amiot2019complete}
Claire Amiot, Pierre-Guy Plamondon, and Sibylle Schroll.
\newblock A complete derived invariant for gentle algebras via winding numbers
  and arf invariants.
\newblock {\em arXiv preprint arXiv:1904.02555}, 2019.

\bibitem{arnesen2016morphisms}
Kristin~Krogh Arnesen, Rosanna Laking, and David Pauksztello.
\newblock Morphisms between indecomposable complexes in the bounded derived
  category of a gentle algebra.
\newblock {\em Journal of Algebra}, 467:1--46, 2016.

\bibitem{assem1981generalized}
Ibrahim Assem and Dieter Happel.
\newblock Generalized tilted algebras of type an.
\newblock {\em Communications in Algebra}, 9(20):2101--2125, 1981.

\bibitem{assem1987iterated}
Ibrahim Assem and Andrzej Skowronski.
\newblock Iterated tilted algebras of type an.
\newblock {\em Math. Z}, 195(2):269--290, 1987.

\bibitem{baur2021geometric}
Karin Baur and Raquel Coelho Sim\~{o}es.
\newblock A geometric model for the module category of a gentle algebra.
\newblock {\em Int. Math. Res. Not. IMRN}, (15):11357--11392, 2021.

\bibitem{beilinson1982faisceaux}
A.~A. Beilinson, J.~Bernstein, and P.~Deligne.
\newblock Faisceaux pervers, {A}nalyse et {T}opologie sur les {E}spaces
  {S}ingulier (i).
\newblock {\em Asterisque}, 100, 1982.

\bibitem{bekkert2003indecomposables}
Viktor Bekkert and H{\'e}ctor~A. Merklen.
\newblock Indecomposables in derived categories of gentle algebras.
\newblock {\em Algebras and Representation Theory}, 6(3):285--302, 2003.

\bibitem{bondal1990representation}
Alexei~Igorevich Bondal.
\newblock Representation of associative algebras and coherent sheaves.
\newblock {\em Mathematics of the USSR-Izvestiya}, 34(1):23, 1990.

\bibitem{ccanakcci2019mapping}
{\.I}lke {\c{C}}anak{\c{c}}{\i}, David Pauksztello, and Sibylle Schroll.
\newblock Mapping cones in the bounded derived category of a gentle algebra.
\newblock {\em Journal of Algebra}, 530:163--194, 2019.

\bibitem{ccanakcci2021corrigendum}
{\.I}lke {\c{C}}anak{\c{c}}{\i}, David Pauksztello, and Sibylle Schroll.
\newblock Corrigendum to “{M}apping cones for morphisms involving a band
  complex in the bounded derived category of a gentle algebra” [{J}.
  {A}lgebra 530 (2019) 163--194].
\newblock {\em Journal of Algebra}, 569:856--874, 2021.

\bibitem{chang2022recollements}
Wen Chang, Haibo Jin, and Sibylle Schroll.
\newblock Recollements of derived categories of graded gentle algebras and
  surface cuts.
\newblock {\em arXiv preprint arXiv:2206.11196}, 2022.

\bibitem{chang2020geometric}
Wen Chang and Sibylle Schroll.
\newblock A geometric realization of silting theory for gentle algebras.
\newblock {\em arXiv preprint arXiv:2012.12663}, 2020.

\bibitem{haiden2014flat}
Fabian Haiden, Ludmil Katzarkov, and Maxim Kontsevich.
\newblock Flat surfaces and stability structures.
\newblock {\em Publ. Math. Inst. Hautes \'{E}tudes Sci.}, 126:247--318, 2017.

\bibitem{holm2010triangulated}
Thorsten Holm, Peter J{\o}rgensen, and Rapha{\"e}l Rouquier.
\newblock {\em Triangulated categories}, volume ({L}ondon {M}athematical
  {S}ociety {L}ecture {N}ote {S}eries) 375.
\newblock Cambridge University Press, 2010.

\bibitem{hovey1997stablehomotopy}
Mark Hovey, John~H. Palmieri, and Neil~P. Strickland.
\newblock Axiomatic stable homotopy theory.
\newblock {\em Mem. Amer. Math. Soc.}, 128(610):x+114, 1997.

\bibitem{kuznetsov2014semiorthogonal}
Alexander Kuznetsov.
\newblock Semiorthogonal decompositions in algebraic geometry.
\newblock In {\em Proceedings of the {I}nternational {C}ongress of
  {M}athematicians---{S}eoul 2014. {V}ol. {II}}, pages 635--660. Kyung Moon Sa,
  Seoul, 2014.

\bibitem{opper2019auto}
Sebastian Opper.
\newblock On auto-equivalences and complete derived invariants of gentle
  algebras.
\newblock {\em arXiv preprint arXiv:1904.04859}, 2019.

\bibitem{opper2018geometric}
Sebastian Opper, Pierre-Guy Plamondon, and Sibylle Schroll.
\newblock A geometric model for the derived category of gentle algebras.
\newblock {\em arXiv preprint arXiv:1801.09659}, 2018.

\bibitem{bergh2000abstract}
Michel Van~den Bergh.
\newblock Abstract blowing down.
\newblock {\em Proceedings of the American Mathematical Society},
  128(2):375--381, 2000.

\end{thebibliography}

\end{document}